\documentclass[12pt]{amsart}

\usepackage{amssymb,latexsym}

\usepackage{color}

\usepackage{enumerate}

\usepackage[T1]{fontenc}

 \usepackage[french,english]{babel}

\makeatletter

\@namedef{subjclassname@2010}{

  \textup{2010} Mathematics Subject Classification}

\makeatother
\usepackage{cite}
\usepackage{enumerate}
\usepackage[utf8]{inputenc}
\usepackage[english]{babel}
\usepackage{amssymb}
\usepackage{tikz}
\usetikzlibrary{calc,decorations.markings}
\usepackage{tkz-euclide}
\usepackage{pgfplots}
\usepackage{comment}
\usepackage{caption}
\usepackage{float}
\usepackage{multicol}
\usepackage{pgfplots}
\usepackage[normalem]{ulem}
\usepackage{cancel} 
 \usepackage{mdframed}

 \usepackage{color}
\definecolor{red}{rgb}{1,0,0}

\newlength\mylen
\newtheorem{theorem}{Theorem}[section]
\newtheorem{lemma}{Lemma}[section]
\newtheorem{corollary}{Corollary}[section]
\newtheorem{proposition}{Proposition}[section]

\theoremstyle{definition}

\newtheorem{example}[]{Example}[section]
\theoremstyle{remark}
\newtheorem{remark}[theorem]{Remark}

\numberwithin{equation}{section}

\renewcommand{\le}{\leqslant}
\renewcommand{\leq}{\leqslant}
\renewcommand{\ge}{\geqslant}
\renewcommand{\geq}{\geqslant}

\newcommand{\ex}{\mathbb{E}}
\newcommand{\re}{\textup{Re}}

\newcommand{\pr}{\mathbb{P}}
\newcommand{\ep}{\varepsilon}
\newcommand{\mc}{\mathcal}

\newcommand{\F}{\mathcal{F}}

\newcommand{\eS}{\mathcal S}
\newcommand{\eL}{\mathcal L}

\newcommand{\newabstract}[1]{%
  \par\bigskip
  \csname otherlanguage*\endcsname{#1}%
  \csname captions#1\endcsname
  \item[\hskip\labelsep\scshape\abstractname.]
}

\begin{document}

\baselineskip=17pt

\title[]{Large values of exponential sums with multiplicative coefficients}

\author{Andrew Granville}
\author{Youness Lamzouri}

\address{Universit\'e de Lorraine, CNRS, IECL,
F-54000 Nancy, France}

\email{youness.lamzouri@univ-lorraine.fr}

\address{D{\'e}partment  de Math{\'e}matiques et Statistique,   Universit{\'e} de
Montr{\'e}al, CP 6128 succ Centre-Ville, Montr{\'e}al, QC  H3C 3J7, Canada.}

\email{andrew.granville@umontreal.ca} 
%\date{}

\begin{abstract} 

In 1977 Montgomery and Vaughan gave tight bounds for exponential sums of the form $\sum_{n\leq x} f(n) e(n\alpha)$ where $f$ is a $1$-bounded multiplicative function and $\alpha\in\mathbb R$, close to the conjectured $\ll  \frac{x}{\sqrt{q}}+ \frac{x}{\log x} $
 where $\alpha$  is  best approximated by $|\alpha-a/q|\leq 1/(qx)$, showing their results to be ``best-possible'' by observing that the first part of their bound is more-or-less attained when $f(n)=\chi(n), \alpha=\frac aq$ where $\chi$ is a primitive character mod  $q$, and the second part when $f(p)=e(-\alpha p)$ for all large primes $p$.
La Bretèche and Granville proved that when $\alpha$ lies on a major arc the exponential sum is significantly smaller unless  $f$  ``pretends to be'' $\chi(n) n^{it}$ for some character $\chi$ and real number $|t|<\log x$; and herein we prove that when $\alpha$ lies on a minor arc, the exponential sum is significantly smaller unless $f(p)$ pretends to be $e(-hp\alpha)$ for primes $p\leq x$ for some bounded integer $h$.

We also study exponential sums $\sum_{n\leq x, P^+(n)\leq y}  f(n) e(n\alpha)$ restricted to $y$-smooth  (or $y$-friable) integers $n$.
We conjecture that this  sum is $\ll  \frac{\Psi(x, y)}{\sqrt{q}}+ \frac{\sqrt{xy}}{\log x} $ in a wide range of parameters, show that if  true this is best possible, and prove an upper bound in a wide range that is only slightly weaker than the conjecture.

Finally we study  the logarithmically weighted exponential sums $\sum_{n\leq x} \frac{f(n)}{n} e(n\alpha)$. We conjecture that this  sum is $\ll \frac{\log x}{\sqrt{q}}+\log q$  in a wide range of parameters, show that if  true this is best possible, and prove an upper bound in a wide range that is only slightly weaker than the conjecture.
 Moreover,  we show that the sum is much smaller than the  conjectured bound unless
 $\alpha$ lies on a major arc ($q\leq (\log x)^{2+\varepsilon}$) and  $f$ ``pretends to be'' $\chi(n) n^{it}$, where $\chi$ is a character mod  $q$, and $|t|\leq (\log x)^{-2/3}$ (in which case we   obtain  an asymptotic series expansion for our exponential sum); 
 or  $\alpha$ lies on a semi-major arc ($(\log x)^{2+\varepsilon} \leq q\leq (\log x)^A$)  and $f(p)$ pretends to be  $e(-hp\alpha)$ for primes $p\leq x$ for some bounded integer $h$; 
 or $\alpha$ lies on a minor arc ( $(\log x)^{2+\varepsilon} \leq q\leq x$) and there is a secondary approximation to $\alpha$, given by $|\alpha-b/r|\leq 1/(rq^c)$ with  $c\in (0,1)$ and $r\ll 1$   and $f$ ``pretends to be'' $\chi(n)$, where $\chi$ is a  character mod  $r$. Moreover one can construct hybrids of the last two cases.

  Along the way, we will prove various technical results about multiplicative functions (such as a Hal\'asz type result for  mean values of logarithmically weighted multiplicative functions in arithmetic progressions) which may be of use elsewhere.

\end{abstract}


\maketitle

\tableofcontents

%%%%%%%%%%%%%%%%%%%%%%%%%%%%%%%%%%%%%%%%%%%%%%%%%%%%%%%%%%%%%%%%%%%%%%%%%%%%%%%%%%

 \section{Introduction and statement of results}

Let $\mathbb{U} := \{z \in \mathbb{C} : |z| \leq 1\}$ be the unit disc and throughout $f: \mathbb{N}\to \mathbb{U}$ is a multiplicative function.  Bounds for exponential sums of the form
\begin{equation}\label{Eq:UnweightedSum}
 \sum_{n\leq x}  f(n) e(n\alpha),   
\end{equation} 
where $ e(\alpha):=\exp(2\pi i \alpha)$, and related sums,
have many applications in analytic number theory. By Dirichlet's theorem for diophantine approximation, one can approximate $\alpha\in[0,1)$ by a rational $a/q$ with $(a, q)=1$ and $q\leq x$ such that 
$|\alpha-a/q|\leq 1/(qx)$.  

%%%%%%%%%%%%%%%%
\subsection{Motivating conjectures}
A folklore conjecture asserts that%We conjecture that
 \begin{equation}\label{Eq:MVZero.3}
  \sum_{n\leq x}  f(n) e(n\alpha) %\lesssim
  \ll  \frac x{\log x} + \frac x{\sqrt{q}},
  \end{equation} 
whenever $q\leq x/(\log x)^{2+\varepsilon}$. In the case where $q\leq (\log x)^2$, we make the following more precise conjecture
 \begin{equation}\label{Eq:MVZero.3Plus}
  \sum_{n\leq x}  f(n) e(n\alpha) %\lesssim
  \lesssim \frac x{\sqrt{q}}
  \end{equation} where $f\lesssim g$ means that $|f|\leq (1+o(1)) g$.
If true these conjectures are best possible (as we will discuss below).

Let $P^+(n)$ be the largest prime factor of $n$, with the standard convention that $P^+(1)=1$. We conjecture that if we restrict the exponential sum to $y$-smooth  (or $y$-friable) numbers (that is, integers $n$ with $P^+(n)\leq y$) then 
\begin{equation}\label{Eq:Conj}
\sum_{\substack{n\leq x\\ P^+(n)\leq y}}  f(n) e(n\alpha)\ll %\lesssim
 \frac{\sqrt{xy}}{\log x} +  \frac{\Psi(x, y)}{\sqrt{q}},
   \end{equation}
  in the range $x\geq y\geq q^2$,\footnote{And it is plausible that  this holds in a wider range of the parameters $x, y, q$.}
   where $\Psi(x,y)$ is the number of $y$-smooth integers $\leq x$,
and if true this  is best possible  in wide ranges of the parameters $x, y$ and $q$  (as we will discuss {in section \ref{Sec:SmoothNumbers} below).  Note that the $y=x$ case gives \eqref{Eq:MVZero.3}.

We conjecture that the same exponential sum, but now with ``logarithmic weights'', satisfies
\begin{equation}\label{Eq:ConjectureLog}
 \sum_{n\leq x} \frac{f(n)}{n}e(n\alpha) \ll \log q+\frac{\log x}{\sqrt{q}},
\end{equation}
for all $q\leq x$. When $q\leq (\log x/\log\log x)^2$, we make the following more precise conjecture
\begin{equation}\label{Eq:ConjectureLogPlus}
 \sum_{n\leq x} \frac{f(n)}{n}e(n\alpha) \lesssim\frac 4{\sqrt{6}}\cdot \frac{\log x}{\sqrt{q}}.
\end{equation}
We will also see in sections \ref{Sec:Log1} and \ref{Sec:Log2} below that if the conjectures \eqref{Eq:ConjectureLog} and \eqref{Eq:ConjectureLogPlus} are true then they are best possible.

We will establish slightly weaker versions of all of these conjectures (which imply the conjectures in a wide range).
%We have made ambitious guesses here;  it would have been safer to make each conjecture with ``$\ll$'' in place of 
%``$\lesssim$''.

%%%%%%%%%%%%%%%%
\subsection{Bounds  on the exponential sums in \eqref{Eq:UnweightedSum}}
Montgomery and Vaughan \cite[Theorem 1]{MV77} proved that if $q\leq x/(\log x)^2(\log\log x)^3$ then
\begin{equation}\label{Eq:MVZero}
 \sum_{n\leq x}  f(n) e\bigg( \frac{an}q \bigg) \ll  \frac{x}{\log x} + \frac x{\sqrt{\phi(q)}}
 \end{equation}
 and therefore one   deduces (as in section \ref{sec2.ex0} below) that if $|\alpha-\frac aq|\leq \frac 1{x} $ then\footnote{If   $|\alpha-\frac aq|\leq \frac 1{\sqrt{q}\, x} $ then we have the simpler  $\sum_{n\leq x}  f(n) (e(n\alpha) -  e(\frac{an}q)) \ll \sum_{n\leq x} \frac n{\sqrt{q}\, x}\ll \frac x{\sqrt{q}}$.}
  \begin{equation}\label{Eq:MVZero.2}
 \sum_{n\leq x}  f(n) e(n\alpha) \ll  \frac{x}{\log x} + \frac x{\sqrt{\phi(q)}}.
 \end{equation}

 %%%%%%%%%%%%%%%
 \subsection{The minor arcs}
 Hence if $\alpha\in [0,1)$ lies on a \emph{minor arc},\footnote{There are different ways we could define our minor arcs and come to the same conclusion.  One way would be to define $Q=\frac{x}{(\log x)^{2+\ep}}$ and only require that $|\alpha-\frac aq|\leq \frac 1{qQ} $ since this is $\leq \frac 1x$ in our range. We could also be more precise about bounds on $q/\phi(q)$ and therefore let
 $(\log x)^2(\log\log\log x)\leq q\leq x/(\log x)^2(\log\log x)^3$ rather than the range $(\log x)^{2+\varepsilon} \leq  q\leq x/(\log x)^{2+\ep}$ given in \eqref{eq: Minor Arc Range} to deduce \eqref{Eq:MVOne}.}  which means that for fixed $\varepsilon>0$, there exists coprime integers $a$ and $q$ such that
 \begin{equation} \label{eq: Minor Arc Range}
 \bigg|\alpha-\frac aq\bigg|\leq \frac 1{qx} \text{ with }  (\log x)^{2+\varepsilon} \leq q\leq  \frac{x}{(\log x)^{2+\ep}},
 \end{equation}
then \eqref{Eq:MVZero.2} implies that
  \begin{equation}\label{Eq:MVOne}
 \sum_{n\leq x}  f(n) e(n\alpha) \ll  \frac{x}{\log x}
 \end{equation}
 so that the conjecture \eqref{Eq:MVZero.3} holds in this range.
 
 Montgomery and Vaughan observed that this is   ``best possible''  up to a constant factor by the simplest case of the following general construction:  For 
 any choice of $f(p)$ for $p\leq \frac x2$,  define
$\sum_{n\leq x, p|n\implies p\leq x/2}  f(n) e(n\alpha)=:r e(\theta)$. Then  let
$f(p)=e(\theta-p\alpha)$ for all $p\in (\frac x2,x]$ which  implies that $\Big|\sum_{n\leq x}  f(n) e(n\alpha)\Big|\gtrsim  \displaystyle{\frac{x}{2\log x}}.$

In Corollary \ref{Cor:Unweighted} we will show that all examples in which the exponential sum  $\sum_{n\leq x} f(n) e(n\alpha)$ is this large (up to a constant) stem from a similar bias of $f(p)$-values for large primes $p$.

 %%%%%%%%%%%%%%%
 \subsection{The major arcs}
La Bret\`eche and Granville (see \cite[Theorem 4]{dlBG2} and the discussion following it) showed that if $f: \mathbb{N}\to \mathbb{U}$ is completely multiplicative then
 \[
  \sum_{n\leq x}  f(n) e(\alpha n) \ll   \frac x{\sqrt{q}} 
 \]
when $q\leq (\log x)^{2-\varepsilon}$ which implies    that the conjecture \eqref{Eq:MVZero.3} holds in this range. 
(In section \ref{sec2.excm} we deduce that this bound also holds for all  multiplicative functions $f: \mathbb{N}\to \mathbb{U}$ when $q$ is squarefree.)
In section \ref{sec2.excn} we will show that their results also imply that 
 if $ \left|\sum_{n\leq x}  f(n) e(\alpha n)\right| \geq c\frac x{\sqrt{q}}$ in this range for completely multiplicative $f$ then, for the ``pretentious'' distance defined in \eqref{Eq:ThePretentiousDistance},
 \[
 \mathbb D(f,\chi(n)n^{it}; x)^2  \leq \log(1/c) +\log\log (1/c)+O(1),
 \]
 for some primitive character $\chi \pmod r$, where $r$ divides $q$ and $|t|\leq \log x$.  Moreover we will deduce, in this range, that
  \[
 \bigg|  \sum_{n\leq x}  f(n) e(\alpha n) \bigg| \lesssim \frac x{\sqrt{q}},
 \]
 and hence conjecture \eqref{Eq:MVZero.3Plus} holds for completely multiplicative functions in this range of $q$.
 Furthermore, one obtains $\sim   \frac x{\sqrt{q}} $ if and only if $\alpha=\frac aq+o(\frac 1x), r=q$ and
  $\mathbb D (f(n),\chi(n)n^{it},x)=o(1)$  with 
 $t=o(\frac 1{\log x})$.   Based on this we guess that conjecture \eqref{Eq:MVZero.3Plus} should hold for all multiplicative functions when $q\leq (\log x)^2$. %one can replace ``$\ll$'' by ``$\lesssim$'' in the conjectured bound in \eqref{Eq:MVZero.3}.
 
In summary, the conjectured bound in \eqref{Eq:MVZero.3} is now completely proved other than in the range $q=(\log x)^{2+o(1)}$, and it is even proved in this range when $\phi(q)\gg q$. Moreover we have shown that both of the main terms in the bound play a role, and cannot be removed.

%%%%%%%%%%%%%%%%%%%%%%%%%%%%%%%%%%%%%%%%%%%%%%%%
\subsection{Exponential sums on minor arcs} 

Our first result implies that if $q>(\log x)^{2+\ep}$ then the sum $\sum_{n\leq x} f(n) e(n\alpha)$ can only be large if the analogous sum is large when   restricted to only those integers which have a very large prime factor (of size $\asymp x$).  Here and throughout we define
\begin{equation}\label{Eq:DefinitionL(x)}
\mathcal{L}(x):=\exp(\sqrt{\log x\log\log x}).
\end{equation}

\begin{theorem}\label{Thm:Unweighted}
Fix $\ep\in (0, \frac 1{10})$ and suppose that $|\alpha-\frac aq|\leq \frac1{qx},$ where $(a, q)=1$ and $(\log x)^{2+\ep} \leq q\leq x/(\log x)^{3+\ep}$.  

If $q\leq x^{1-\ep}$ and $M$ is an integer for which 
\[
2\leq M\ll  \min\left\{q^{1-\ep/4}/(\log x)^2,  \mathcal{L}(x)^{1/2-\ep}\right\}
\]
 then for all multiplicative functions $f:\mathbb{N}\to \mathbb{U}$ we have   
% The second term could be \exp((1-\varepsilon)(\log x \log\log x)^{1/2}\
\begin{equation}  \label{Eq:BestApproxI}
 \sum_{n\leq x} f(n) e(n\alpha)= \sum_{m\leq M}f(m) \sum_{x/M< p\leq x/m} f(p) e(m p \alpha) +O\left(\frac{x}{\sqrt{M} 
 \log x}\right).
\end{equation}

If $x^{1-\ep}\leq q\leq x/(\log x)^{3+\ep}$ and $M$ is an integer for which
\[
2\leq M\ll  \min\left\{\sqrt{x}/(\sqrt{q}\log x),  x/(q\log(2x/q)(\log x)^3), \mathcal{L}(x)^{1-\ep}\right\}
\]
 then for all multiplicative functions $f:\mathbb{N}\to \mathbb{U}$ we have
\begin{equation}  \label{Eq:BestApproxI2}
 \sum_{n\leq x} f(n) e(n\alpha)= \sum_{m\leq M}f(m) \sum_{x/M< p\leq x/m} f(p) e(m p \alpha) +O\left(\frac{x\sqrt{\log M}}{\sqrt{M}\log x}\right).
\end{equation}
\end{theorem}

 A key ingredient in the proof of Theorem \ref{Thm:Unweighted} is  an improved estimate for such exponential sums over smooth (or friable) numbers (see section \ref{Sec:SmoothNumbers}).
 Each prime $p>x/M$ in the sum on the right-hand side of \eqref{Eq:BestApproxI} and \eqref{Eq:BestApproxI2} is only involved in a small number of terms, and never with one another, so they can be viewed as independent of each other. Moreover, by setting $M=2$ we obtain a new proof of the Montgomery-Vaughan bound \eqref{Eq:MVOne} in the minor arcs (namely for $q$ in the range $(\log x)^{2+\ep}\leq q\leq x/(\log x)^{3+\ep}$).

An immediate consequence of this result is the following:

\begin{corollary}\label{Cor:Unweighted}
  Fix $\ep\in (0, \frac 1{10})$ and suppose that $|\alpha-\frac aq|\leq \frac1{qx},$ where $(a, q)=1$ and $(\log x)^{2+\ep} \leq q\leq x/(\log x)^{3+\ep}$. 
 For any fixed $c\in (0,1)$, if 
$f:\mathbb{N}\to \mathbb{U}$ is a multiplicative function for which
\begin{equation}  \label{eq: LowerBound1}
\Big|\sum_{ n\leq x} f(n) e(n\alpha)\Big|\geq   \frac{cx}{\log x}, 
\end{equation}
then there exists a positive integer $h\ll (1/c)^2\log(1/c)$ and a real number $z$ satisfying $ c^2x/(\log(1/c)) \ll z\leq x/2h$ such that 
\begin{equation} \label{eq: LowerBound1Conseq}
\sum_{z < p\leq 2z} f(p)e(h p\alpha) \gg  \frac{c^3}{\log(1/c)}\frac{z}{\log z}.  
\end{equation}
\end{corollary}

We next give a family of examples in the full range  \eqref{eq: Minor Arc Range} for which \eqref{eq: LowerBound1}  and
  \eqref{eq: LowerBound1Conseq} hold, so that Corollary \ref{Cor:Unweighted} is ``best possible''.
We also show that one does \emph{not} necessarily obtain  \eqref{eq: LowerBound1} whenever \eqref{eq: LowerBound1Conseq} holds.

\begin{example}\label{Exa1}  
Select any values for the $f(p), z < p\leq 2z$  with $x\ll_c z\leq x/2$ for which \eqref{eq: LowerBound1Conseq} holds (for example, by letting each  $f(p)= e(-hp\alpha)$). 
Then allow each $f(p)$ for $p\leq z$ and $2z<p\leq x$ to be independent random variables uniformly distributed on $\mathbb U$.
We will show in section \ref{sec2.ex1} that  \eqref{eq: LowerBound1} holds with probability $\gg 1$ for this family of completely multiplicative functions $f$.
  \end{example}
  
\begin{example}\label{Exa1b}    
However,   \eqref{eq: LowerBound1Conseq} does not guarantee  that \eqref{eq: LowerBound1} holds as may be seen by taking  
 $y$ to be that integer for which $\pi(x)-\pi(y)=\pi(y)-\pi(\frac x2)-\delta$ where $\delta= 0$ or $1$ (so that $y=2z\sim \frac {3x}4$)
and then $f(p)=e(-p\alpha)$ for all primes $p\in (\frac x2,y]$ with $f(p)=-e(-p\alpha)$ for all primes $p\in (y,x]$, and otherwise $f(p)=0$ for $p\leq \frac x2$. Then $\sum_{z < p\leq 2z} f(p)e(p\alpha)\sim \frac x{4\log x}$ whereas
$\sum_{ n\leq x} f(n) e(n\alpha)= e(\alpha)+\delta$.  (One can also obtain $\sum_{ n\leq x} f(n) e(n\alpha)=x^{1/2+o(1)}$ with probability $1+o(1)$ if the $f(p), p\leq z$ are independent random variables uniformly distributed on $\mathbb U$, but the $f(p), p>z$ stay the same.)
\end{example}

 %%%%%%%%%%%%%%%%%%%%%%%%%%%%%%%%%%%%%%%%%%%%%%%%%%%%%%%%%%%%%%%%%%%%%%
 \subsection{Exponential sums over smooth numbers} \label{Sec:SmoothNumbers}

 Estimates for the exponential sums 
\begin{equation}\label{Eq:ExponentialSumsSmoothMulti}
\sum_{\substack{n\leq x\\ P^+(n)\leq y}}  f(n) e(n\alpha), 
\end{equation}
where $f:\mathbb{N}\to \mathbb{U}$ is a multiplicative function, were studied by several authors, notably in the special case $f=1$ (see for example \cite{Dr16}, \cite{FoTe91}  and \cite{Har16}), in which case such bounds were used in the framework of the circle method to study the distribution of smooth numbers in certain special sets.

Similarly to the work of Montgomery-Vaughan \cite{MV77}, which corresponds to the case $x=y$, we are interested in the following natural question: if $\alpha=a/q$ (for simplicity), can one obtain best possible bounds for the exponential sum  \eqref{Eq:ExponentialSumsSmoothMulti} over the class of multiplicative functions $f:\mathbb{N}\to \mathbb{U}$, in a large range of the parameters $x, y$ and $q$? 

Let $x$ be large, $2\leq y\leq x$ and suppose that  $|\alpha-a/q|\leq 1/q^2$ where $(a, q)=1$ and $q\leq x$. Improving Proposition 1 of \cite{dlB98}, La Bret\`eche and Granville \cite[Eq. (10.1)]{dlBG2} proved that for all multiplicative functions $f:\mathbb{N}\to \mathbb{U}$ one has
 \begin{equation}\label{Eq:AndrewRegis}
\sum_{\substack{n\leq x\\ P^+(n)\leq y}}  f(n) e(n\alpha)\ll 
\left(\frac{\sqrt{xy}}{\log y}+ \frac{x}{\sqrt{q}}+\sqrt{xq\log(2x/q)}\right) \log y  + x \mathcal L(x)^{-\frac 12+o(1)}
\end{equation}
where $\mathcal L(x)$ was defined in \eqref{Eq:DefinitionL(x)},
and asked whether one can remove the factor $\log y$ from the first terms in \eqref{Eq:AndrewRegis}: By 
optimizing one step in their argument, we can save a  factor of $\sqrt{\log y}$ from the first terms (see Proposition \ref{Pro:ImprovingRegis} below).  However, one would prefer an upper bound in terms of the counting function $\Psi(x, y)$ instead of $x$, since the trivial upper bound $\Psi(x, y)$ might  be smaller than the upper bound in \eqref{Eq:AndrewRegis} when $y$ is small (indeed this holds once $y\leq \mathcal L(x)^{1-\varepsilon}$  by \eqref{Eq:ChangeRegimSmoothCount}, since the last term on the right hand side of \eqref{Eq:AndrewRegis} exceeds the trivial bound $\Psi(x, y)$ in this range).

When $\alpha=a/q$, and $y$ and $q$  are ``small'', La Bret\`eche and Tenenbaum \cite{BrTe07} obtained better bounds, improving on work of  Maier \cite{Mai06}: For fixed $A>0$ they showed that if $2\leq q\leq (\log x)^A$ then
 \begin{equation}\label{Eq:BretecheTenenbaum1}
\sum_{\substack{n\leq x\\ P^+(n)\leq y}}  f(n) e(n\alpha) \ll  \frac{\Psi(x, y)}{\sqrt{\log\log(3q)}}
 \end{equation}
for $\exp(c(\log\log x)^2)\leq y\leq x/q$ for some positive constant $c=c(A)>0$, and  
\begin{equation}\label{Eq:BretecheTenenbaum2}
\sum_{\substack{n\leq x\\ P^+(n)\leq y}}  f(n) e(n\alpha) \ll  \frac{2^{\omega(q)/2}}{(q\phi(q))^{1/4}} \Psi(x, y),
 \end{equation}
  in the smaller range $\exp((\log x)^{\ep})\leq y\leq x/q$,  where $\omega(q)$ denotes the number of distinct prime factors of $q$. 
 We improve both of these bounds, by showing in particular that 
 \[
 \sum_{\substack{n\leq x\\ P^+(n)\leq y}}  f(n) e(n\alpha) \ll    \dfrac{2^{\omega(q)/2}}{\sqrt{q}} \Psi(x, y)
 \]
  throughout the   range $\exp(c(\log\log x)^2)\leq y\leq x/q$ when $q\leq (\log x)^A$, and for $q$ up to 
  $ \mathcal L(x)^{\frac 1{\sqrt{2}}-\varepsilon}$ when $y\leq x/q$ is in a certain range. Moreover, for larger values of $y$, namely $y\geq x/q$ (and $q=x^{o(1)}$), we also improve $\eqref{Eq:AndrewRegis}$ by removing   the factor $\log y$ from the first terms, which answers the question of La Bretèche and Granville \cite{dlBG2} in the affirmative. Furthermore, we will  construct multiplicative functions $f:\mathbb{N}\to \mathbb{U}$ for which our bounds are best possible, in certain ranges of the parameters $x, y$ and $q$.  
 
\begin{theorem}\label{Thm:Smooth}
Fix $A>0$. For $x$   large, suppose that  $|\alpha-\frac aq|\leq \frac1{qx},$ where $(a, q)=1$ and $q\leq \exp\big(\log x/(\log\log x)^{1/4}\big)$.
Let $(2q)^{\log(u+1)/A}\leq y \leq x$ where  $u=\log x/\log y$.  
Then for any multiplicative function $f:\mathbb{N}\to \mathbb{U}$ we have
 \begin{equation}\label{Eq:ySmallTheoremSmooth}
\sum_{\substack{n\leq x\\ P^+(n)\leq y}}  f(n) e(n\alpha)\ll  \frac{\sqrt{xy}}{\log x} + 
\frac{2^{\omega(q)/2}}{\sqrt{q}}\Psi(x, y) +\Psi(x, q^2).
   \end{equation} 
Furthermore, we can omit the $\frac{\sqrt{xy}}{\log x} $ term unless $y>x/q$.
\end{theorem}

Now if  $q\leq \mathcal L(x)^{\frac 1{\sqrt{2}}-\varepsilon}$  then $\Psi(x, q^2) = o(\Psi(x, y)/\sqrt{q})$ by 
\eqref{Eq:ChangeRegimSmoothCount} below,  so  one removes the third term $\Psi(x, q^2)$ in \eqref{Eq:ySmallTheoremSmooth}.

 We have $\frac{\sqrt{xy}}{\log x}=o(\Psi(x, y)/\sqrt{q})$ for $(\log x)^{2+\varepsilon}<y\leq x/q$ using \eqref{Eq:CEPSmooth} and \eqref{Eq:SmoothPowersLog} below, so the first term may be omitted in this range;
and it may be omitted for $y\leq (\log x)^{2+\varepsilon}$ by the proof of Theorem \ref{Thm:Smooth}.

  Here and throughout we let $\mathbf{1}_{\mathcal{A}}$ denote the indicator function of an event $\mathcal{A}$.
\begin{corollary}\label{Cor:Smooth} Fix $A>0$ and $0<\ep<1/10$. Suppose that $x, y, \alpha,  a, q$ satisfy the assumptions of Theorem \ref{Thm:Smooth} and that   $q\leq  \mathcal L(x)^{\frac 1{\sqrt{2}}-\varepsilon}$. 
Then for any multiplicative function $f:\mathbb{N}\to \mathbb{U}$ we have
\begin{equation}\label{Eq:ySmallTheoremSmooth2}
\sum_{\substack{n\leq x\\ P^+(n)\leq y}}  f(n) e(n\alpha)\ll 
\frac{\sqrt{xy}}{\log x}\cdot  \mathbf{1}_{y>x/q}+  \frac{2^{\omega(q)/2}}{\sqrt{q}}\Psi(x, y).
   \end{equation}
% Moreover, if $x/q< y\leq x$ then for any multiplicative function $f:\mathbb{N}\to \mathbb{U}$ we have 
% \begin{equation}\label{Eq:yLargeTheoremSmooth2}
% \sum_{\substack{n\leq x\\ P^+(n)\leq y}}  f(n) e(n\alpha)\ll \frac{\sqrt{xy}}{\log x}+ 
% \frac{2^{\omega(q)/2}}{\sqrt{q}}x.
%    \end{equation} 
\end{corollary}

We will now prove that if \eqref{Eq:Conj} holds then it is best possible:

Taking $\alpha=a/q$ and $\chi$ to be a primitive Dirichlet character modulo $q$, we will show in section \ref{SubSec:MajorArcsSmooth} below that 
\begin{equation}\label{Eq:LargeCharacterMajorSmooth0} 
\sum_{\substack{n\leq x\\ P^+(n)\leq y}} \chi(n) e(n\alpha)\gg \frac{\Psi(x, y)}{\sqrt{q}}, 
 \end{equation}
in certain ranges of the parameters $x,y$ and $q$, which implies that \eqref{Eq:Conj} is best possible and the upper bound in Corollary \ref{Cor:Smooth} is  optimal up to the factor $2^{\omega(q)/2}$ (at least if $y\leq x/q$). 

To describe these ranges we restrict $q$ here to be prime   for simplicity (though our proofs work for arbitrary $q$): For fixed $A>0$ there exists $B=B(A)>0$ such that we 
always take $(\log x)^B\leq y\leq x$ and either $q\leq (\log x)^A$, or
\[
  (\log x)^A \leq  q\leq \exp\left(c_1\frac{\sqrt{\log x}}{\log\log x}\right)  
   \text{ and }  q^{c_2} \leq y\leq \exp\left(c_3\frac{\log x}{\log q(\log\log q)^2}\right),
\]
for some positive constants $c_1, c_2$ and $c_3.$
Furthermore, assuming the Generalized Riemann Hypothesis (GRH), we can extend these ranges to 
\[
  y\geq  q^3(\log q)^5.
\]

Now if $y>x/q$ in these ranges for $q$ then $\Psi(x,y)\asymp x$ and so $\frac{x}{\sqrt{q}}\gg \frac{\sqrt{xy}}{\log x}$
for $y\ll x(\log x)^2/q$, and therefore Corollary \ref{Cor:Smooth} is best possible up to the factor $2^{\omega(q)/2}$
with this further restriction.

This range can be rewritten as $q\leq (x/y) (\log x)^2$  which suggests that the ``major arcs'' should go up to $  q\leq (x/y) (\log x)^2$. This  gives the usual $q\leq (\log x)^2$ in the classical $y=x$ case, but something rather
larger  for smaller $y$. 

 Corollary \ref{Cor:Smooth} also implies that 
 $$\sum_{\substack{n\leq x\\ P^+(n)\leq y}}  f(n) e(n\alpha)\ll \frac{\sqrt{xy}}{\log x}$$
 in the range 
 $$
x/\mathcal L(x)^{\frac 1{\sqrt{2}}-\varepsilon}\leq y\leq x \text{ with } \left(\frac{x}{y} (\log x)^2\right)^{1+\ep/2}\leq q\leq \mathcal L(x)^{\frac 1{\sqrt{2}}-\varepsilon}.
$$  
We will now show that in a slightly smaller range of the parameters $x, y$ and $q$, we can construct a multiplicative function $f$ taking values on the unit disc, for which this bound is best possible.

\begin{proposition}\label{Pro:OptimalMinorArcSmooth}
 Fix $0<\ep<1/10$. Let $x$ be large, 
   $ xe^{-(1/2-\ep)\sqrt{\log x}}\leq y \leq x$
  and let $q$ be  a positive integer such that
  $$\left(\frac{x^2}{y^2} (\log x)^{3/2}\right)^{1+\ep}\leq q\leq x e^{-\frac32\sqrt{\log x}}.$$
    Put $\alpha=a/q$ where $(a, q)=1$. Then there exists a multiplicative function $f:\mathbb{N}\to \mathbb{U}$ such that 
\begin{equation}\label{Eq:ProSmoothLargeMinor}
\sum_{\substack{n\leq x\\ P^+(n)\leq y}}  f(n) e(n\alpha)\gg \frac{\sqrt{xy}}{\log x}.
\end{equation}
\end{proposition}

%%%%%%%%%%%%%%%%%%%%%%%%%%%%%%%%%
\subsection{Logarithmic exponential sums on minor arcs}  \label{Sec:Log1} 
Logarithmically weighted exponential sums
\begin{equation}\label{Eq:WeightedExpSum}
\sum_{n\leq x}  \frac{f(n)}{n} e(n\alpha),
\end{equation}
appear in various contexts.  For example, one explores  \eqref{Eq:WeightedExpSum} with $f=\overline{\chi}$ to study the distribution of character sums $\sum_{n\leq x} \chi(n)$ via P\'olya's Fourier series expansion. Thus
the exponential sums \eqref{Eq:WeightedExpSum} have been extensively studied (by  Montgomery-Vaughan \cite{MV77}, Granville-Soundararajan \cite{GS07}, Goldmakher \cite{Go12}, Goldmakher-Lamzouri \cite{GoLa12,GoLa14}, Lamzouri-Mangerel \cite{LaMa22} and others).

  By partial summation on \eqref{Eq:MVOne}, we deduce  that if 
$(\log x)^{2+\ep}\leq q\leq \sqrt{x}$ then
\begin{equation}\label{Eq:MVTwo}
\sum_{q^2< n\leq x}  \frac{f(n)}{n} e(n\alpha) \ll  \log \log x;
\end{equation}
and, using the trivial bound $|f(n)e(n\alpha)|\leq 1$ for all $n\leq q^2$, we then obtain 
\begin{equation}\label{Eq:MVThree}
\sum_{n\leq x}  \frac{f(n)}{n} e(n\alpha) \ll  \log q,
\end{equation}
which holds trivially for $\sqrt{x}\leq q\leq x$. Therefore \eqref{Eq:MVThree} holds throughout \eqref{eq: Minor Arc Range}. 

We will show that the bounds \eqref{Eq:MVTwo} and \eqref{Eq:MVThree} are best possible\footnote{This does not follow from   variants of the Montgomery-Vaughan construction   since $\sum_{z< p\leq 2z}\frac{1}{p}\ll \frac 1{\log z}$.}, and better understand when these bounds are attained.
First, analogous to Theorem \ref{Thm:Unweighted}, we show that sums like \eqref{Eq:WeightedExpSum}
can only be large if the analogous sum is large when restricted to  integers that are the product of a very small number and a large prime. 

\begin{theorem} \label{cor: Cut sum2} 
 Fix $\ep\in (0, \frac 1{10})$ and suppose that $\alpha$ lies on  \eqref{eq: Minor Arc Range} with  $q \leq x^{1/2-\varepsilon}$.    Then
uniformly for all multiplicative functions $f:\mathbb{N}\to \mathbb{U}$ we have 
 \begin{equation}
\label{Eq:BestApprox}
 \sum_{q^2< n\leq x} \frac{f(n)}{n} e(n\alpha)=   
    \sum_{m\leq M} \frac{f(m)}{m} \sum_{q^2< p\leq x/m} \frac{f(p)}{p}   e(mp\alpha)
  +O(1),
\end{equation}
where $M=(\log\log x)^3$.
\end{theorem}
%\end{corollary}

Each prime $p>q^2$ in the sum on the right-hand side of \eqref{Eq:BestApprox} is only involved in a small number of terms, and never with one another, so they can be viewed as independent of each other.

Using Theorem \ref{cor: Cut sum2} we can  classify when the upper bound in \eqref{Eq:MVThree} is attained.
 For multiplicative functions $f,g: \mathbb{N} \rightarrow \mathbb{U}$ we define the ``pretentious'' distance by
\begin{equation}
\label{Eq:ThePretentiousDistance}
\mathbb{D}(f,g;x) := \left(\sum_{p \leq x} \frac{1-\text{Re}(f(p)\overline{g}(p))}{p}\right)^{1/2}.
\end{equation}

\begin{theorem}\label{Thm:Classify_Large}
 Fix $\ep\in (0, \frac 1{10})$ and suppose that $\alpha$ lies on   \eqref{eq: Minor Arc Range}.
 For any fixed $B>1$, there exists a positive constant $A$ that depends only on $B$ such that  if 
$f:\mathbb{N}\to \mathbb{U}$ is a multiplicative function for which
\begin{equation}
\label{Eq:LogLowerBd}
\left|\sum_{ n\leq x} \frac{f(n)}{n} e(n\alpha)\right|\geq \frac 1B \log q, 
\end{equation}
then at least one of the following holds:
\begin{itemize}
    \item[(a)] There exist coprime integers $1\leq b<r \ll_B 1$ such that   $ |\alpha-b/r|\ll_B q^{-1/3B}$  and a primitive character $\chi\pmod \ell$ with $\ell \mid r$ such that 
    $$ \mathbb{D}(f, \chi; q) \ll_B 1 ; $$
    \item[(b)] $q$ is in the range $(\log x)^{2+\ep}\leq q\leq (\log x)^A$ and there  exists a positive integer  $h\ll B^6$ such that  
\begin{equation}
\label{Eq:LogLowerBdConseq}
    \left|\sum_{p\leq x} \frac{f(p)}{p} e(hp\alpha)\right|\gg   \frac{\log\log x}{B\log 2B}.
 \end{equation}
\end{itemize} 
\end{theorem}

We next give examples in the range  \eqref{eq: Minor Arc Range} with $q\leq x^{1/2-\ep}$ where (a) or (b) hold and imply \eqref{Eq:LogLowerBd}, so Theorem \ref{Thm:Classify_Large} is ``best possible''.
We also show that one does \emph{not} necessarily obtain  \eqref{Eq:LogLowerBd} whenever (a) or (b) hold.

\begin{example} [The bound \eqref{Eq:MVTwo} is best possible  for all   $q\leq x^{1/2-\varepsilon}$] \label{Exa2}
Assume that $f(p^e)$ with $p>q^2$ are given so that \eqref{Eq:LogLowerBdConseq} holds. Let the $f(p^e)$ with $p\leq q^2$
be i.i.d. random variables each uniformly distributed on  $\mathbb{U}$. 
In section \ref{sec2.ex2a} we will show that 
\[
\left|\sum_{q^2<n\leq x} \frac{f(n)}{n} e(n\alpha)\right|\gg \log\log x
\]
 with probability $\gg 1$. This implies that   \eqref{Eq:MVTwo} is best possible   for all   $q\leq x^{1/2-\varepsilon}$, and
therefore \eqref{Eq:MVThree} is best possible  for any $q\in [(\log x)^{2+\ep}, (\log x)^A]$.
 (We obtain a wider range in Example \ref{Exa3}.)
\end{example}
 
\begin{example} [Condition (b) and \eqref{Eq:LogLowerBd}] \label{Exa2b}
It is important to note that Condition (b) does not imply \eqref{Eq:LogLowerBd}:  Let $f(2)=-1$ and $f(p^k)=0$ for $2<p^k\leq M$.
Then let $f(p)=e(-2p\alpha)$ for $q^2\leq p \leq  y:=\exp\big((\log x)^{2/3}(2\log q)^{1/3}\big)$ and 
$f(p)=e(-p\alpha)$ for $y<p\leq x$. The equation \eqref{Eq:BestApprox} implies that 
 \[ 
 \bigg|\sum_{q^2< n\leq x} \frac{f(n)}{n} e(n\alpha) \bigg| \ll 1 
 \]
 (see section \ref{sec2.ex2a} for details).  
Therefore \eqref{Eq:LogLowerBd} fails
yet \eqref{Eq:LogLowerBdConseq} holds for both $h=1$ and $h=2$.
\end{example}

\begin{example}[Condition (a) and \eqref{Eq:LogLowerBd}. The bound \eqref{Eq:MVThree} is best possible  for any $q\in [(\log x)^A,  x^{1/2})$ for $A$ sufficiently large] \label{Exa3}  
% Example \ref{Exa2} shows that \eqref{Eq:MVThree} is best possible for $q\in [(\log x)^{2+\ep}, (\log x)^A]$ so we can assume   $q\geq (\log x)^{A}$, for some suitably large constant $A$. 
The estimate \eqref{Eq:MVTwo} can be rewritten as  
\begin{equation} \label{Eq:MVTwo with small changes}
\sum_{ n\leq x} \frac{f(n)}{n} e(n\alpha)=\sum_{ n\leq q^2} \frac{f(n)}{n} e(n\alpha)+O(\log\log x).
\end{equation}
Assume that there exist coprime integers $1\leq b<r <q$ such that   $ |\frac aq-\frac br|=\frac 1{qr}$. 
Let $\chi$ be a primitive character modulo $r$ and let $f$ be a completely multiplicative function with $f(p)=\chi(p)$ for all primes $p\leq q^2$ (and $f(p)$ can take any value on the unit disc for larger primes).  We will show in section \ref{sec2.ex2} that 
\begin{equation} \label{Eq:MVTwo with estimates}
\sum_{ n\leq q^2} \frac{f(n)}{n} e(n\alpha)= \frac{\overline{\chi}(b)\tau(\chi)} r \log q+O(\log r)
\end{equation}
where 
$\tau(\chi):=\sum_{j=1}^r\chi(j)e(\frac{j}r)$ is the Gauss sum so that $|\overline{\chi}(b)\tau(\chi)|=\sqrt{r}$.
Therefore, if $r\ll A^2$ (where the implicit constant here depends  on the implicit constant in the error term of \eqref{Eq:MVTwo with small changes}) then  
\[
\sum_{ n\leq x} \frac{f(n)}{n} e(n\alpha) \gg \log q,
\]
showing that \eqref{Eq:MVThree} and Theorem \ref{Thm:Classify_Large}(a) are best possible.  
 \end{example}

 \begin{example} [A hybrid example] \label{Exa4} We select $q\in [(\log x)^{2+\ep},  \exp((\log x)^{1-\ep})]$
for which 
\begin{equation} \label{logBV}
 \sum_{ \substack{y<p\leq x/y \\ p\equiv b \bmod q}} \frac{1}{p} = \frac 1{\phi(q)} \bigg(  \sum_{ \substack{y<p\leq x/y }} \frac{1}{p} +O(1)\bigg)
\text{  whenever } (b,q)=1
\end{equation} 
   holds    with $y=q^3$. (This is true for  almost all $q$ by Remark \ref{Rem2} below). 
   We let $f(p)=e(-p\alpha)$ for all primes $p> q^2$ so that the $m=1$ terms dominate in  \eqref{Eq:BestApprox}.
 We will show in section \ref{sec2.ex3} that indeed, no matter what the values of $f(p)$ for $p\leq q^2$  are, we have 
\begin{equation}\label{eq: ex3partI}
 \sum_{q^2\leq n\leq x} \frac{f(n)}{n} e(n\alpha) \geq \frac{1}{6} \log\bigg(\frac{\log x}{\log q}\bigg) \geq \frac{\ep}{6} \log\log x
\end{equation}
once $x$ is sufficiently large.

 Next we proceed exactly as in Example \ref{Exa3}, assuming that there exists coprime integers $1\leq b<r\ll 1$ such that   $ |\frac aq-\frac br|=\frac 1{qr}$
and letting $f$ be the completely multiplicative function with $f(p)=\chi(p)$ for all primes $p\leq q^2$, so that
$$\sum_{n\leq q^2} \frac{f(n)}{n} e(\alpha n) \gg \log q.$$
 \end{example}

%%%%%%%%%%%%%%%%%%%%%%%%%%%%%%%%%%%%%%%%%%%%%%%%%%%%%%%%%%%%%%%%%%%%%%%%%%%%%%%%%

\subsection{Logarithmic exponential sums: The major arcs case}\label{Sec:Log2} Fix $\varepsilon>0$. For a given value of $x$,   the \emph{major arcs} are given by those $\alpha\in [0,1)$ for which there exists coprime integers $a$ and $q$ such that
 \begin{equation} \label{eq: Major Arc Range}
 \bigg|\alpha-\frac aq\bigg|\leq \frac 1{qx} \text{ with }  q\leq (\log x)^{2+\varepsilon}. 
 \end{equation}
 
The upper bound \eqref{Eq:MVZero.2} is 
% Montgomery and Vaughan \cite{MV77} proved that, uniformly for all multiplicative functions $f: \mathbb{N}\to \mathbb{U}$, we have
% \[ \sum_{n\leq x}  f(n) e(n\alpha) \ll   \frac{x}{\sqrt{q}}(\log q)^{3/2} + \frac{x}{\log x} \]
% and  Bachmann \cite{Bach}  replaced the $(\log q)^{3/2}$ factor by  $(\log q \log\log q)^{1/2}$. This is 
more-or-less best possible for the major arcs
 \eqref{eq: Major Arc Range} since if $\chi$ is a primitive character mod $q$ then  
\begin{equation}\label{Eq:GaussSums}
 \sum_{n\leq x}  \chi(n) e\bigg( \frac{an}{q}\bigg) = \frac{\overline{\chi}(a)\tau(\chi)}{q} x+O(q)
\end{equation}
which has size $\sim  x/\sqrt{q}$ if $q=o(x^{2/3})$ (and this is $>x/\log x$ for $q<(\log x)^2$).
 
  By partial summation the upper bound \eqref{Eq:MVZero.2} implies that
 % the Montgomery-Vaughan  estimate (with Bachmann's modification) implies that 
\begin{equation}\label{Eq:MVLargeIntegers}
  \sum_{q^2\leq n\leq x} \frac{f(n)}{n}e(n\alpha) \ll \frac{\log x}{\sqrt{\phi(q)}} +\log\log x
\end{equation}
(which holds for all $q\leq \sqrt{x}$ by  \eqref{Eq:MVTwo}). 
Using the trivial bound $|f(n)e(n\alpha)|\leq 1$ for  $n\leq q^2$ one obtains\footnote{If $q\leq \log x$, the first term of \eqref{Eq:MVLargeIntegers+} is $\gg (\log x)^{1/2}$, which is much larger than $\log\log x$, and if $q>\log x$ then $\log q\geq \log\log x$, so the term $
\log\log x$ can be removed in all cases. Moreover, the bound is trivial for $\sqrt{x}\leq q\leq x$.} 
\begin{equation}\label{Eq:MVLargeIntegers+}
  \sum_{n\leq x} \frac{f(n)}{n}e(n\alpha) \ll \frac{\log x}{\sqrt{\phi(q)}}+\log q,
\end{equation}
for all $q\leq x$. 
 For $(\log x)^2<q\leq x$ we attain the bound $\ll \log q$ (and therefore the conjecture \eqref{Eq:ConjectureLog}) by  \eqref{Eq:MVLargeIntegers+}, and this  is best possible by the results of section \ref{Sec:Log1}. 

On the other hand, the first term of our conjectural upper bound \eqref{Eq:ConjectureLog} is best possible for smaller $q$, since by applying partial summation to \eqref{Eq:GaussSums} we have
\begin{align*}
 \sum_{n\leq x} \frac{\chi(n)}{n}e\bigg( \frac {an}q\bigg) &=
\frac{\overline{\chi}(a)\tau(\chi)}q \log x+ \sum_{n\leq q^2} \frac{\chi(n)}{n}e\bigg( \frac {an}q\bigg) + O(1)\\
&=
\frac{\overline{\chi}(a)\tau(\chi)}q \log x+ O(\log q) 
\end{align*}
 which has size $\asymp \frac{\log x}{\sqrt{q}}$ if $q\ll (\frac{\log x}{\log\log x})^2$. 
 
  We want to classify all multiplicative functions $f:\mathbb{N}\to \mathbb{U}$ for which our conjectural bound \eqref{Eq:ConjectureLog} is attained when $q\leq (\log x)^{2+\ep}$, and to obtain asymptotic formulas for $\sum_{n\le x} f(n) e(\alpha n)/n$ in such cases. This was done by La Bretèche and Granville \cite{dlBG2} for the unweighted exponential sums \eqref{Eq:UnweightedSum} when $f$ is completely multiplicative. We shall adapt the  ``pretentious'' approach
arguments from \cite{dlBG2} and \cite{GHS19}   to the logarithmically weighted setting. This is more delicate  due to the slow variation of the exponential sums \eqref{Eq:WeightedExpSum}, compared to the unweighted sums \eqref{Eq:UnweightedSum}. 
 We will summarize the main results in section \ref{sec: PretendLogs}, but this is the key consequence:

\begin{corollary} \label{Cor:Smallq3} Let $x$ be large. Suppose that  $\alpha$ lies in a major arc given by 
 \eqref{eq: Major Arc Range}. 
For any multiplicative function $f:\mathbb N\to \mathbb U$ there exists a primitive character $\chi \pmod \ell$, for some integer $\ell \mid q$, such that 
%\[ \sum_{n\leq x} \frac{f(n)e(\alpha n) }n=  \frac{\overline{\chi}(a)\tau(\chi)}{\phi(q)} 
% \sum_{\substack{m\geq 1 \\ p|m\implies p|q}}  \frac{\kappa_f(m\frac q\ell)\overline{\chi}(m)}{m} 
%  \sum_{\substack{n\leq x\\ (n, q)=1}} \frac{f(n)\overline{\chi}(n)}{n}   +O\left(\frac{(\log x)^{2/3+o(1)}}{ \sqrt{q}}+ e^{O(\sqrt{\log\log x})} \right),\]
% where c\[ F_f(q/\ell):= \prod_{p^e\|q/\ell}\left(\sum_{j=0}^{\infty} \frac{\overline{\chi}(p^j)\kappa_f(p^{e+j})}{p^j}\right),\]
\[
\sum_{n\leq x} \frac{f(n)e(\alpha n) }n=  \frac{\overline{\chi}(a)\tau(\chi)}{\phi(q)} 
 \sum_{n\leq x} \frac{F(n)\overline{\chi}(n)}{n} 
 +O\left(\frac{(\log x)^{2/3+o(1)}}{ \sqrt{q}}+ e^{O(\sqrt{\log\log x})} \right),
\]
where $F$ is the multiplicative function  defined by $F(p^k)=f(p^k)$ if $p\nmid q$, and 
$F(p^k)=\kappa_f(p^{k+e})$ if $p|q$ and $p^e\| q/\ell$, where  $\kappa_f$ is defined by the convolution  $f  = \kappa_f*\chi$. The main term here is  $\displaystyle{ \lesssim  \frac 4{\sqrt{6}} \frac{\log x}{\sqrt{q}}}$.
\end{corollary}

\begin{example} \label{exa2.3}  In section \ref{sec2.ex2.3} we will show that if $f$ is the multiplicative function for which $f(n)=\chi(n)$ whenever $(6,n)=1$, and
 $f(n)=-\chi(n)$ if $n$ is a power of 2 or a power of 3, and $q=6\ell$ with $(6,\ell)=1$ and $\ell$ squarefree, then 
 \[
\bigg|\sum_{n\leq x} \frac{f(n)e ( \frac nq) }n\bigg| \sim  \frac 4{\sqrt{6}} \frac{\log x}{\sqrt{q}}.
 \]
 This is best possible by Corollary \ref{Cor:Smallq3}; moreover
 the proof of Corollary \ref{Cor:Smallq3} implies that the only multiplicative $f$ which attain this asymptotic are either the $f$ given here, or some very minor variation of $f$.
\end{example}

If $f$ is completely multiplicative  then  % $\kappa_f(m\frac q\ell)=\kappa_f(\frac q\ell) f(m)$ and so 
$F(n) = \kappa_f(\frac q\ell) f(n)$ in which case the main term in Corollary \ref{Cor:Smallq3} can be written as 
\[
  \frac{\overline{\chi}(a)\tau(\chi) \kappa_f(q/\ell) }{\phi(q)}
 \sum_{n\leq x} \frac{f(n)\overline{\chi}(n)}{n}  .
\]
This will be proved directly in Corollary \ref{Cor:Smallq2}.

\begin{remark} Under the same hypothesis as Corollary \ref{Cor:Smallq3}, one can prove
\begin{equation}\label{Eq:BretGran}
\sum_{n\leq x}  f(n)e(\alpha n) =  \frac{\overline{\chi}(a)\tau(\chi)}{\phi(q)} 
 \sum_{n\leq x}  F_t(n)\overline{\chi}(n) 
 +O\left(\frac x {(\log x)^{1-1/\sqrt{2}+o(1)}} \right),
\end{equation}
where   $F_t(N)=F(N)m^{it}$ for some $|t|\leq \log x$ and $N/m$ is the largest divisor of $N$ that is coprime to $q$.
One begins with
La Bret\`eche and Granville's \cite[Theorem 1]{dlBG2} which proves this for completely multiplicative functions, and then suitably modifies the proof of %compares the appropriate sum with the result in  
Corollary \ref{Cor:Smallq3}. %We sketch this proof in section \ref{app: unweighted}.
\end{remark}  

We will prove Corollary \ref{Cor:Smallq3} in section \ref{sec: VariousCharacters}.
Corollary \ref{Cor:Smallq3} implies the   bound 
$$\bigg|\sum_{n\leq x} \frac{f(n)e(\alpha n) }n\bigg|\lesssim  \frac 4{\sqrt{6}} \frac{\log x}{\sqrt{q}},
$$ and so the conjecture
\eqref{Eq:ConjectureLogPlus}, for $q \ll (\log x)^2e^{-c\sqrt{\log\log x}}$ for some constant $c>0$.
  Therefore, by \eqref{Eq:MVLargeIntegers+}, conjecture \eqref{Eq:ConjectureLog} remains open  \emph{only} in the range
\[
(\log x)^2 \exp( -c \sqrt{\log\log x}) < q < \bigg(\frac{\log x}{\log\log x}\bigg)^2\log\log\log x.
\]

%%%%%%%%%%%%%%%%%%%%%%%%%%%%%%%%%%%%%%%%%%%%%%%%%%%%%%%%%%%%%%%%%%%%%%%%%%%%%%%%%

  \subsection{Logarithmic exponential sums: Developing asymptotics on the major arcs} \label{sec: PretendLogs}  
 For  any $f:\mathbb{N}\to \mathbb{U}$ and $T>0$ we define 
$$ M(f; x, T):= \min_{|t|\leq T} \mathbb{D}(f, n^{it}; x)^2,
$$ 
where $ \mathbb{D}(f, n^{it}; x)$ is defined in \eqref{Eq:ThePretentiousDistance}. 

Given an integer $q$ we 
order the primitive characters $\chi\bmod \ell$ for all $\ell\mid q$ as $\{\chi_j \bmod \ell_j\}_{j\geq 1}$, so that 
\begin{equation*}
M(f\overline{\chi_j}, x; T) \leq M(f\overline{\chi_{j+1}};x, T) \text{ for all } j\geq 1.
\end{equation*}
 
Let  $\mathcal{F}$ be the class of completely multiplicative functions $f:\mathbb{N}\to \mathbb{U}$.

\begin{theorem} \label{Thm:Smallq}  Let $x$ be large, $J\leq \log\log x$ be a positive integer and
 $T=(\log x)^{-1/\sqrt{J}}$. Suppose that  $\alpha$ lies in a major arc given by  \eqref{eq: Major Arc Range}.   Then,  for all $f\in \F$ we have 
\begin{align*}
\sum_{n\leq x} \frac{f(n)e(\alpha n) }n&= \sum_{j=1}^{J-1}\frac{\overline{\chi}_j(a)\kappa_j(q/\ell_j) \tau(\chi_j)}{\phi(q)}\sum_{n\leq x} \frac{f(n) \overline{\chi}_j (n)}{n} \\
& \quad \quad \quad +O\left((\log x)^{1/\sqrt{J} }e^{O(\sqrt{\log\log x})} \right),
\end{align*}
 where  $\kappa_j$ is  defined by the convolution  $f(n) = (\kappa_j*\chi_j)(n)$. Furthermore, if $J=2$ we have the better error term $(\log x)^{2/3+o(1)}$ and we can take $T=(\log x)^{-2/3}$ in this case.
\end{theorem}

It is surprising, but convenient, that the $n^{it_j}$ (where $t_j$ corresponds to where the minimum in $M(f\overline{\chi_j}, x; T)$ occurs) do not appear in the $j$th sum (unlike in \eqref{Eq:BretGran} or Theorems 1 and 4 of \cite{dlBG2}), since  the $t_j$'s are sufficiently small in our case. Theorem \ref{Thm:Smallq}  implies Corollary \ref{Cor:Smallq3} for completely multiplicative $f$ as well as the bound given below in  Corollary \ref{Cor:Smallq}.

 The key ingredient in the proof of Theorem \ref{Thm:Smallq}   is the following ``pretentious'' large sieve inequality for mean values of logarithmically weighted  multiplicative functions twisted by Dirichlet characters modulo $q$. 
 We will restrict attention to the class of multiplicative functions $f$ such that $|\Lambda_f(n)|\leq \Lambda(n)$ for all $n\geq 1$, where $\Lambda(n)$ is the von Mangoldt function and $\Lambda_f(n)$ is the Dirichlet coefficient of $-\frac{F'}{F}(s)$ where $F(s)=\sum_{n\geq 1}f(n)/n^s$ is the Dirichlet series associated to $f$ and Re$(s)>1$. Note that this class contains the class $\mathcal{F}$ of $1$-bounded completely multiplicative functions. For a  subset $\mathcal X$   of the characters modulo $q$, we define
 $$
 M_{\mathcal X, f}(x, T) := \min_{\substack{\chi \bmod q\\ \chi\not\in {\mathcal X}}} M(f\overline{\chi}; x, T).   
 $$

\begin{theorem}\label{Thm:PretentiousLSieve}
Let $\ep>0$ be small and fixed. Let $x$ be large with $T\in [ \frac 2{\log x}, 1]$  and $f$ be a multiplicative function such that $|\Lambda_f(n)|\leq \Lambda(n)$ for all $n\geq 1$. For an integer $q>1$ with $q\leq (\log x)^{10}$ and any subset $\mathcal X$   of the characters modulo $q$,  we have 
\begin{equation}\label{eq: PLS log}
\sum_{\substack{\chi \bmod q\\ \chi\not\in {\mathcal X}}} \Big|\sum_{y  \leq n\leq x}\frac{f(n)\overline{\chi}(n)}{n}\Big|^2
  \ll \bigg( (\log x)^2    \exp\left(-2M_{\mathcal X, f}(x, T)\right)     + |\mathcal X|^2 + \frac{1}{T^2} \bigg) (\log\log x)^6
\end{equation}
for all real numbers  $q^{1+\ep}\leq y\leq (\log x)^{100}$ , where the implicit constant is absolute. 
\end{theorem}

To extend the sums in \eqref{eq: PLS log} from $y  \leq n\leq x$ to $n\leq x$ one might try to add the sum of the squares of the contributions for $n< y$. However  that   will add an extra $O(\phi(q))$ to the  right-hand side of \eqref{eq: PLS log} which  is too big for applications (see the proof in section \ref{sec: summands<=y}).

We shall deduce Theorem \ref{Thm:PretentiousLSieve} from the following general form of a Hal\'asz type result for mean values of logarithmically weighted multiplicative functions in arithmetic progressions.

\begin{theorem}\label{Thm:LogHala}
Let $\ep, x, y, T, q$ and  $\mathcal X$ be as in Theorem \ref{Thm:PretentiousLSieve}. Let $1\leq a\leq q$ be an integer such that $(a, q)=1$.  Let $f$  be a multiplicative function such that $|\Lambda_f(n)|\leq \Lambda(n)$ for all $n\geq 1$. Then we have
\begin{equation}\label{Eq:InequalityHALASZ2}
\begin{aligned}
& \sum_{\substack{y \leq n\leq x\\ n\equiv a \bmod q}}  \frac{f(n)}{n} - \frac 1{\phi(q)} \sum_{\substack{\chi \bmod q\\ \chi\in {\mathcal X}}} \chi(a) \sum_{y\leq n\leq x} \frac{f(n)\overline{\chi}(n)}{n}\\
&\hskip1in \ll \bigg(  (\log x)     \exp\left(-M_{\mathcal X, f}(x, T)\right)   +  |\mathcal X|  + \frac 1T \bigg)  
\frac { (\log\log x)^3} {\phi(q)}  
\end{aligned}
\end{equation}
where the implicit constant is absolute.
\end{theorem}
To prove this result we follow the approach of \cite{GHS19}. However, there are several new complications partly because in Perron's formula we need to work on a contour $s=c_0+it$ with  $0<c_0\ll 1/\log x$ (so that $x^{s}\ll 1$), which makes the analysis more delicate since we are very close to the pole at $s=0.$ The analysis is also performed more carefully  in order to show that if $f$ ``pretends to be close'' to $\chi(n)n^{it}$ then $|t|$ is very small.

   %%%%%%%%%%%%%%%%%%%%%%%%%%%%%%%%%
\subsection{Summands $\leq y$ in the pretentious large sieve} \label{sec: summands<=y}
We are interested in extending  the sums in \eqref{eq: PLS log} from $y  \leq n\leq x$ to $n\leq x$. The contribution of the terms $n\leq y$, over all characters mod $q$ is given by 
  %$y= \exp(o(\sqrt{q}))$  %(when $\mathcal{X}=\emptyset$)
\begin{align*} \sum_{\chi \bmod q} \Big|\sum_{  n\leq y}\frac{f(n)\overline{\chi}(n)}{n}\Big|^2
&= \sum_{m\leq y} \frac{f(m)}m \sum_{\substack{n\leq y\\ n\equiv m \bmod q}} \frac{\overline{f(n)}}n\\
&= \phi(q)\sum_{n\leq y}\frac{|f(n)|^2}{n^2}+ O\Bigg(\phi(q)\sum_{m\leq y}\frac{1}{m}\sum_{\substack{q<n\leq y\\ n\equiv m \bmod q}}\frac{1}{n}\Bigg)\\
& = 
\phi(q)\sum_{n=1}^{\infty}\frac{|f(n)|^2}{n^2}+ O\left(\frac{\phi(q)(\log y)^2}{q}\right).
\end{align*}
 since $y\geq q^{1+\ep}$. The $n=1$ term shows that the main term is $\geq \phi(q)$ and this dominates if $y<e^{c\sqrt{q}}$ for a sufficiently small constant $c>0$, which follows when $y\leq (\log x)^{100}$ (the range of Theorem \ref{Thm:PretentiousLSieve}) for $q\gg (\log\log x)^2$.

\subsection*{Acknowledgements} 
AG is partially supported by a grant from the the Natural Sciences and Engineering Research Council of Canada, and is a CRM Distinguished Researcher.\\
  YL is supported by a junior chair of the Institut Universitaire de France, and was partially supported by a Simons-CRM visiting professorship while part of this work was carried out. In particular, YL would like to thank the Centre de Recherches Math\'ematiques for its excellent working conditions. %Thanks also to
  The authors would like to thank Haozhe Gou for his helpful remarks that helped improve several of the examples given in this article.
%%%%%%%%%%%%%%%%%%%%%%%%%%%%%%%%%%%%%%%%%%%%%%%%%%%%%%%%%%%%%%%%%%%%%%%%%%%%%%%%%%%%%%%%%%%%%%%%%%%%%%%%%%%%%%%%%%%%%%%%%%%%%%%%%%%%%%%%%%%%%%%%%%%%%%%%%%%%%

 \section{Calculations in the introduction}
To not overburden the introduction we have worked out various details, especially of the examples, in this section.

 %%%%%%%%%%%%%%%%%%%%%%%%%%%%%%%%%
\subsection{From $\frac aq$ to nearby $\alpha$; deducing \eqref{Eq:MVZero.2} }\label{sec2.ex0} 

Let $S(t)=\sum_{n\leq x}  f(n) e(\frac aq n) $. If $\alpha=\frac aq+\beta$ then
\begin{align*}
 \sum_{n\leq x}  f(n) e(\alpha n) &= \int_1^x e(\beta t) dS(t)=S(x)e(\beta x) - 2\pi i \beta  \int_1^x e(\beta t) S(t) dt\\
 &\ll (1+|\beta|x) \max_{t\leq x} |S(t)| .
\end{align*}
  Therefore if $|\beta|\leq \frac 1x$ then
\[
 \sum_{n\leq x}  f(n) e(\alpha n) \ll \max_{N\leq x}  \bigg|   \sum_{n\leq N}  f(n) e(\tfrac aq n)  \bigg|,
\]
 and hence  \eqref{Eq:MVZero.2} follows from  \eqref{Eq:MVZero}.

 %%%%%%%%%%%%%%%%%%%%%%%%%%%%%%%%%
 \subsection{Completely multiplicative to multiplicative}\label{sec2.excm}  
La Bret\`eche and Granville  \cite[Theorem 4]{dlBG2} showed that if $f$ is completely multiplicative then
\begin{equation} \label{eq: DelB-G.I}
  \sum_{n\leq x}  f(n) e(\alpha n) \ll   \frac x{\sqrt{q}} 
\end{equation}
when $q\leq (\log x)^{2-\varepsilon}$. 

We want to prove the same bound for all $1$-bounded multiplicative $f$. So given $f$ define $F$ to be the completely multiplicative function with each $F(p)=f(p)$, and then $g$ to be the multiplicative function with $g(p^k)=f(p^k)-f(p)f(p^{k-1})$ so that $f=g*F$. Therefore
\[
  \sum_{n\leq x}  f(n) e(\alpha n) =   \sum_{mr\leq x} g(m) F(r) e(m\alpha r)
  \leq  \sum_{\substack{m\leq x \\ m \text{ powerful}}} 2^{\omega(m)} \bigg|  \sum_{r\leq x/m} F(r)   e(m\alpha r) \bigg|  
\]
since $|g(m)|\leq 2^{\omega(m)}$, and $g(m)$ is only supported on powerful numbers. Now the inner sum is always
$\leq x/m$ (since each summand has absolute value $\leq 1$), and so the sum over all $m>M$ is 
$$ \ll x\sum_{\substack{m >M \\ m \text{ powerful}}} \frac{2^{\omega(m)}}{m} \ll \frac{x}{M^{1/2-\ep}} \sum_{\substack{m\geq 1 \\ m \text{ powerful}}} \frac{2^{\omega(m)}}{m^{1/2+\ep}} \ll_{\ep} \frac{x}{M^{1/2-\ep}} $$
%$\ll \frac{x \log M}{\sqrt{M}}$,
 Therefore, we can take $M=q^2$ to get a negligible contribution.

Now   $|\alpha-\frac aq|\leq \frac 1{qx} $ and so if $\frac {am}q= \frac{a_m}{q_m}$ where $(a_m,q_m)=1$ then $|m\alpha-\frac{a_m}{q_m}|\leq \frac 1{qx/m}\leq \frac 1{q_m x/m} $
so we can apply \eqref{eq: DelB-G.I} to the inner sums with  $m\leq q^2$ to obtain
\begin{equation}\label{Eq:CompToMult}
  \sum_{n\leq x}  f(n) e(\alpha n) \ll     x\sum_{\substack{m\leq q^2 \\ m \text{ powerful}}}  \frac {2^{\omega(m)}}{m\sqrt{q_m}}  + \frac{x}{q^{3/2}}.
\end{equation}
 Now if $d=rs$ where $r$ is squarefree, $s=\prod_p p^{s_p}$ is powerful and $(r,s)=1$, then $d|m$ implies 
$m=\prod_{p|r} p^{m_p} \prod_{p|s} p^{m_p}  n$ where $(n,d)=1$ and $m_p\geq 2$ if $p|r$ with $m_p\geq s_p$ if $p|s$. Therefore,  as $q_m=q/(q,m)$ we get
\begin{align*}
\sum_{\substack{m\leq q^2 \\ m \text{ powerful}}}   \frac {2^{\omega(m)}}{m\sqrt{q_m}} 
&\leq \frac{1}{\sqrt{q}}\sum_{d\mid q} \sqrt{d}\sum_{\substack{m\leq q^2 \\ m \text{ powerful}\\ d\mid m}} \frac {2^{\omega(m)}}{m}\\
&= \frac 1{\sqrt{q}} \sum_{rs|q} 2^{\omega(rs)}  \sum_{\substack{ m_p\geq 2 \text{ for } p|r \\  m_p\geq s_p \text{ for } p|s}}
 \frac {1}{\prod_{p|r} p^{m_p-\frac 12} \prod_{p|s} p^{m_p-\frac{s_p}2}  } 
 \sum_{\substack{n\leq q^2/rs \\ n \text{ powerful}}}      \frac {2^{\omega( n)}}{  n } \\
 &\ll \frac 1{\sqrt{q}} \sum_{rs|q} \frac{2^{\omega(rs)} s^{1/2} }{ r^{3/2} \phi(s) }   \ll \frac 1 {\sqrt{q}} \cdot \prod_{p^2|q} \frac{p+1}{p-1}.
 \end{align*}
Thus \eqref{eq: DelB-G.I} holds for all  $1$-bounded multiplicative functions when $q$ is squarefree, and even whenever $Q/\phi(Q)\ll 1$ where $Q$ is the largest powerful divisor of $q$.

 %%%%%%%%%%%%%%%%%%%%%%%%%%%%%%%%%
\subsection{A precise evaluation of an exponential sum}\label{sec2.excn}  

% By taking the trivial bound for $n\leq q$, and partial summation with (1.3) for $q\leq n\leq N=x^{\frac{\phi(q)/q}{\log\log x}}$, and this new estimate for $N\leq n\leq x$, we deduce
We now show that that if $q\leq (\log x)^{2-\varepsilon} $  and $f:\mathbb{N}\to \mathbb{U}$ is completely multiplicative then 
\[
 \sum_{n\leq x}  f(n) e(\alpha n) \lesssim \frac {x}{\sqrt{q}} .
 \]
Moreover,  we will also prove that if $ \sum_{n\leq x}  f(n) e(\alpha n) \geq c \frac {x}{\sqrt{q}}$ then $\mathbb D(f(n),\chi(n)n^{it},x)^2\leq \log(1/c)+\log\log(1/c)+O(1)$.

\begin{proof} Let $\beta=\alpha-a/q$. By \cite[Theorem 4]{dlBG2}, if $J$ is sufficiently large as a function of $\varepsilon$,
then there exist primitive characters $\chi_j \pmod {r_j}, j=1,\dots,J$ where each $r_j$ divides $q$, and real $t_j$ with $|t_j|<\log x$,
such that 
\begin{align}\label{Eq:ExpansionExpSum}
\sum_{n\leq x}  f(n) e(\alpha n) &= \sum_{j=1}^J \frac{\overline{\chi_j}(a) \kappa_{\chi_j}(q/r_j)\tau(\chi_j)}{\phi(q)} I(x,\beta,t_j) 
  \sum_{n\leq x}  \frac{f(n)\overline{\chi}_j(n)}{n^{it_j}}\nonumber\\
  & \quad \quad  + o\bigg( \frac {x}{\sqrt{q}}  \bigg),
\end{align}
where $\kappa_\chi$ is defined by the convolution $f(n)/n^{it}= (\kappa_\chi*\chi)(n)$, and 
$$ 
I(x, \beta, t)= \frac{1}{x} \int_0^ x e(\beta v) v^{it} dv. 
$$
We now bound each term in the sum (dropping the subscripts for convenience):

 Let $\chi \pmod {r}$ be a primitive character such that $r\mid q$. Using \cite[(2.3)]{dlBG2} we deduce (since the value of $t$ there will equal $0$) that if $d\leq q$ then 
\[
\sum_{n\leq x/d} \frac{f(n)\overline{\chi}(n)}{n^{it}} = \frac 1{d} \sum_{n\leq x} \frac{f(n)\overline{\chi}(n)}{n^{it}} +O\bigg( \frac x{d(\log x)^\eta}\bigg)
\]
for some $\eta>0$ (independent of $f$). Therefore we have
\begin{align*}
\sum_{\substack{n\leq x\\ (n,q)=1}} \frac{f(n)\overline{\chi}(n)}{n^{it}} 
&= \sum_{d|q} \mu(d) \sum_{\substack{n\leq x \\ d|n}} \frac{f(n)\overline{\chi}(n)}{n^{it}}
= \sum_{d|q} \mu(d)\frac{f(d)\overline{\chi}(d)}{d^{it}} \sum_{ m\leq x/d} \frac{f(m)\overline{\chi}(m)}{m^{it}}\\
&=\sum_{d|q} \mu(d)\frac{f(d)\overline{\chi}(d)}{d^{1+it}}  \sum_{n\leq x} \frac{f(n)\overline{\chi}(n)}{n^{it}} 
+O\bigg( \sum_{d|q} \frac {\mu(d)^2x}{d(\log x)^\eta}\bigg)\\
&=\prod_{p|q} \bigg( 1 -  \frac{f(p)\overline{\chi}(p)}{p^{1+it}}  \bigg)  \sum_{n\leq x} \frac{f(n)\overline{\chi}(n)}{n^{it}} 
+O\bigg(   \frac x{(\log x)^{\eta+o(1)}} \bigg) .
\end{align*}
This implies that 
\begin{align*}
\bigg|  \kappa_{\chi}(q/r)  \sum_{n\leq x} \frac{f(n)\overline{\chi}(n)}{n^{it}} \bigg| 
&=  \bigg|\prod_{p|q/r}( f(p)p^{-it}-\chi(p) )\bigg( 1 -  \frac{f(p)\overline{\chi}(p)}{p^{1+it}}  \bigg)^{-1}
  \sum_{\substack{n\leq x\\ (n,q)=1}} \frac{f(n)\overline{\chi}(n)}{n^{it}}\bigg| \\
& \quad \quad \quad +O\bigg(   \frac x{(\log x)^{\eta+o(1)}} \bigg) .
\end{align*}
 If $p|r$ then the $p$th term in the Euler product is $|f(p)p^{-it}|\leq 1$ so to maximize we can assume $|f(p)|=1$.
 Let $R=\prod_{p|q,\ p\nmid r} p$.  If $p|R$ then the $p$th term in the Euler product is 
  \[
 \frac{|1-\alpha_p|}{ | 1 - \alpha_p/p  |} \leq  \frac{2}{ 1+1/p}
 \]
 where $\alpha_p=f(p)\overline{\chi}(p)p^{-it}$, with equality only when $\alpha_p=-1$.  
 We also have
%  \[  \bigg|  \sum_{\substack{n\leq x\\ (n,q)=1}} \frac{f(n)\overline{\chi}(n)}{n^{it}}\bigg| \leq \sum_{\substack{n\leq x\\ (n,q)=1}}  1 \sim  \frac {\phi(q)}q x \] and 
$|I(x,\beta,t)|\leq 1$; note that $|I(x,\beta,t)|\sim 1$ if and only if  $\beta=o(1/x)$ and $t=o(1/\log x)$.\footnote{To see this note that $|I(x,\beta,t)|\sim 1$ iff $e(\beta v) v^{it}$ is more or less fixed for all $v$ in the range $0<v\leq x$ and the claim follows.}
 Therefore
 \[
 \bigg| \frac{\overline{\chi}(a)\kappa_{\chi}(q/r)\tau(\chi)}{\phi(q)} I(x,\beta,t) 
  \sum_{n\leq x}  \frac{f(n)\overline{\chi}(n)}{n^{it}}  \bigg| \lesssim 
 \prod_{\substack{p|R}}\frac{2}{ \sqrt{p}(1+1/p)}  \cdot   \frac{\sqrt{rR}}{\phi(q)}   \bigg| \sum_{\substack{n\leq x\\ (n,q)=1}}  \frac{f(n)\overline{\chi}_j(n)}{n^{it_j}} \bigg| 
 \]
as $|\tau(\chi)|=\sqrt{r}$. Now,  noting that $\frac{2}{ \sqrt{p}(1+1/p)}\leq (\frac 89)^{1/2}<1$ for each prime $p$ and
$rR\leq q$,  yields
\begin{equation}\label{Eq:ExpansionExpSum2}
\bigg| \sum_{n\leq x}  f(n) e(\alpha n) \bigg|   \lesssim   \sum_{j=1}^J  \frac{\sqrt{q}}  {\phi(q)}
 \bigg| \sum_{\substack{n\leq x\\ (n,q)=1}}  \frac{f(n)\overline{\chi}_j(n)}{n^{it_j}} \bigg|    + o\bigg( \frac {x}{\sqrt{q}}  \bigg).
\end{equation}
Now each summand is  $\leq  \frac{\sqrt{q}}  {\phi(q)}  \sum_{n\leq x,\ (n,q)=1}  1\sim \frac{x}{\sqrt{q}}$.  Moreover by Hal\'asz's theorem we have
\[
\frac{\sqrt{q}}  {\phi(q)}   \sum_{\substack{n\leq x\\ (n,q)=1}} \frac{f(n)\overline{\chi_j}(n)}{n^{it_j}} \ll (1+M_j) e^{-M_j}  \frac {x}{\sqrt{q}},
\]
where, by definition, $M_j=\mathbb D(f(n),\chi_j(n)n^{it_j};x)^2$. Denote by $(\chi, t)$ the pair $(\chi_j, t_j)$ for which $M_j$ attains its minimum over all $j\leq J$. Then it follows from Lemma 3.3 of \cite{BGS13} that
\[
M_k\geq (\tfrac 13 +o(1)) \log\log x
\]
for all $k\leq J$ with $k\neq j$, and hence for such $k$
 %Now, by the triangle inequality we have
%\[
%\mathbb D(f(n),\chi_j(n)n^{it_j};x)+\mathbb D(f(n),\chi_k(n)n^{it_k};x) \geq 
%\mathbb D(\chi_j(n)n^{it_j},\chi_k(n)n^{it_k};x)= \mathbb D(\xi(n),n^{it};x)
%\]
%where $\xi=\chi_j\overline{\chi_k}\ne \chi_0$, the principal character $\pmod q$ and $t=t_j-t_k$. Now
%\[
%\mathbb D(\xi(n),n^{it};x)^2=\sum_{p\leq x} \frac{1-\text{Re}(\psi(p)p^{-it})}p=\log\log x +O( \log\log q(1+|t|))\sim \log\log x
%\]
%by \cite[Lemma 3.5.1]{GS18}.
%Therefore
%\[
%\max\{ M_j,M_k\}\geq (\tfrac 14 -\epsilon) \log\log x,
%\]
%and so for all but one $j$-value we have
\[
\frac{\sqrt{q}}  {\phi(q)}   \sum_{\substack{n\leq x\\ (n,q)=1}} \frac{f(n)\overline{\chi_k}(n)}{n^{it_k}} \ll  \frac {x}{\sqrt{q}(\log x)^{1/4}}.
\]
Summing over $k\leq J\ll_{\ep} 1$ in \eqref{Eq:ExpansionExpSum2} yields
$$ 
\sum_{n\leq x} f(n) e(n\alpha) \lesssim \frac{x}{\sqrt{q}}, 
$$
as desired.

 To obtain $\sum_{n\leq x} f(n) e(n\alpha)\sim \frac x{\sqrt{q}} $ we need  $R=1, r=q$ and $\mathbb D(f(n),\chi(n)n^{it};x)=o(1)$ so that
$  \big|\sum_{n\leq x} \frac{f(n)\overline{\chi}(n)}{n^{it}}\big| \sim \frac{\phi(q)}qx$.
 
Suppose now that  $ \sum_{n\leq x} f(n) e(n\alpha)\gg c \frac x{\sqrt{q}} $, and let 
 $T=\log x$.  Then $c\ll (1+M) e^{-M}$ and so $M\leq \log(1/c)+\log\log(1/c)+O(1)$.
 \end{proof}

 %%%%%%%%%%%%%%%%%%%%%%%%%%%%%%%%%
\subsection{Multiplicative functions for which the bound \eqref{Eq:MVOne} is sharp }\label{sec2.ex1} 

\begin{proof} [Proof of the claims in Example   \ref{Exa1}]  
In our construction we select any values for the $f(p), z < p\leq 2z$  with $x\ll_c z\leq x/2$ for which \eqref{eq: LowerBound1Conseq} holds. Let $M:=[x/z]\ll_c 1$, and put
 \[
 S_m:=\sum_{z<p\leq \min\{ 2z, x/m\} } f(p) e(mp\alpha) \text{ for all } m\leq M.
\]
Then by our assumption we have  
$$
\frac x{\log x} \ll  |S_h|\leq |S_{\ell}|:=\max_{m\leq M} |S_m| \ll \frac x{\log x}.
$$
 %where $M:=[x/z]\ll_c 1$
% (for example, if each  $f(p)= e(-hp\alpha)$). %$f(p)=1$).
We now allow each $f(p)$ for $p\leq z$ and $2z<p\leq x$ to be i.i.d. random variables uniformly distributed on $\mathbb U$.
Let $P$ be the product of the primes in $(z,2z]$.

If $n\leq x$ and prime $p$ divides $(n, P)$ then we can write $n=mp$ with $m\leq M$.
Therefore
\[
\sum_{ n\leq x} f(n) e(n\alpha) = \sum_{m\leq M} f(m) S_m + \sum_{ \substack{n\leq x\\ (n,P)=1}} f(n) e(n\alpha).
\]
Now
\[
\mathbb{E}\bigg( \Big|\sum_{ \substack{n\leq x\\ (n,P)=1}} f(n) e(n\alpha)\Big|^2\bigg)  = \sum_{ \substack{n\leq x\\ (n,P)=1}} 1\leq x
\]
since $\mathbb{E}( f(m)\overline{f(n)})=\mathbf{1}_{m=n}$ when $(mn,P)=1$.  Analogously we have
\[
\mathbb{E}\bigg( \Big| \sum_{m\leq M} f(m) S_m\Big|^2\bigg)  =  \sum_{m\leq M} |S_m|^2\geq |S_{\ell}|^2\gg \bigg( \frac x{\log x}\bigg)^2.
\]
We also have
\[
\mathbb{E}\bigg( \Big| \sum_{m\leq M} f(m) S_m\Big|^4\bigg)  =  \sum_{\substack{a,b,c,d\leq M\\ ab=cd}}  S_aS_b\overline{S_cS_d}
\leq |S_{\ell}|^4 \sum_{m\leq M^2} d(m)^2
\ll M^2(\log M)^3 \cdot |S_{\ell}|^4,
\]
where $d(m)=\sum_{d\mid m}1$ is the divisor function.
Let       $X_f$ and $Y_f$ be  random variables such that $X_f:=| \sum_{m\leq M} f(m) S_m|$ and $Y_f=X_f$ if $|X_f|>\frac 12 |S_{\ell}|$ with $Y_f=0$ otherwise, so that $\mathbb{E}(Y_f^2)\geq \mathbb{E}(X_f^2)-(\frac 12 |S_{\ell}|)^2\geq \frac 34 |S_{\ell}|^2$, and
$\mathbb{E}(Y_f^4)\leq \mathbb{E}(X_f^4)\ll M^2(\log M)^3 \cdot |S_h|^4$.   Now $\mathbb{E}(Y_f^2)^2\leq \mathbb{P}(Y_f\ne 0)\mathbb{E}(Y_f^4)$ by the Cauchy-Schwarz inequality and so 
\[
\text{Prob}\bigg( |X_f|>\frac 12 |S_{\ell}|\bigg) =  \mathbb{P}(Y_f\ne 0) \geq \frac{ \mathbb{E}(Y_f^2)^2}{\mathbb{E}(Y_f^4)} \gg \frac 1{M^2( \log M)^3}\gg_c 1.
\]
Thus, we deduce that  the probability that \eqref{eq: LowerBound1} holds is $\gg_c 1$ for this family of multiplicative functions $f$.

\end{proof}

Many thanks for Haozhe Gou for supplying this argument, an improvement on our original one.

 %%%%%%%%%%%%%%%%%%%%%%%%%x%%%%%%%%
\subsection{Multiplicative functions for which the bound \eqref{Eq:MVTwo} is sharp }\label{sec2.ex2a} 

\begin{proof} [Proof of the claims in Example   \ref{Exa2}]  
Let $M= (\log\log x)^3$ and assume that $f(p^e)$ with $p>q^2$ are given so that  \eqref{Eq:LogLowerBdConseq} holds for some $h\ll 1$. For $m\leq M$ we define   
\[
 G_m:=\sum_{q^2<p\leq  x/m} \frac{ f(p) } p e(mp\alpha).
\]
 Let the $f(p^e)$ with $p\leq q^2$
be i.i.d. random variables each uniformly distributed on $\mathbb{U}$. We first observe that 
$$
\mathbb E\Big|\sum_{p\leq q^2} \frac{f(p)}{p} e(hp\alpha)\Big|^2= \sum_{p\leq q^2}\frac{1}{p^2}=O(1),
$$
and hence we have $\log\log x\gg \max_{m\leq M} |G_m|\geq G_h \gg \log\log x$ with probability $1+o(1)$ by \eqref{Eq:LogLowerBdConseq}.
We now put 
$$ 
\mathcal{L}_f(x, \alpha)=  \sum_{m\leq M} \frac{f(m)}{m} \sum_{q^2< p\leq x/m} \frac{f(p)}{p}   e(mp\alpha) =  \sum_{m\leq M} \frac{f(m)}{m} G_m .
$$
Then we have 
\[
\mathbb E\left|\mathcal{L}_f(x, \alpha)\right|^2 =\sum_{m\leq M} \frac{|G_m|^2}{m^2}\geq  \frac{|G_{h}|^2}{h^2}
\]
and
\[
\mathbb E\left|\mathcal{L}_f(x, \alpha)\right|^4 =\sum_{\substack{a,b,c,d\leq M \\  ab=cd}} \frac{G_aG_b\overline{G_cG_d}}{abcd}\leq  
\max_{m\leq M} |G_m|^4 \sum_{m\leq M^2} \frac{d(m)^2}{m^2} \ll |G_{h}|^4.
\]
Arguing as in the previous subsection   implies that 
\[
\text{Prob}\bigg(|\mathcal{L}_f(x, \alpha)|>\frac 1{2h} |G_{h}|\bigg) \gg \frac 1{h^4} \gg 1,
\]
and so, by 
 \eqref{Eq:BestApprox} we deduce that for such $f$ we have
\[
\bigg|\sum_{ q^2<n\leq x} \frac{f(n)}{n} e(n\alpha)\bigg|\gg \log\log x. \qedhere
\]

\end{proof}

\begin{proof} [Proof of the claims in Example \ref{Exa2b}]
 Substituting the given values into \eqref{Eq:BestApprox} gives that 
 \[
 \sum_{q^2< n\leq x} \frac{f(n)}{n} e(n\alpha)=    \sum_{y< p\leq x} \frac{1}{p}     - \frac 12  \sum_{q^2< p\leq y} \frac{1}{p}    
+  \sum_{q^2< p\leq y} \frac{e(-p\alpha)}{p} - \frac 12  \sum_{y< p\leq x/2} \frac{e(p\alpha)}{p}   
  +O(1).
\]
The contribution of the first two terms equal $O(1)$ by the prime number theorem.
Daboussi's   refinement \cite{Dab96}  of Vinogradov's bound for exponential sums over primes, yields 
\begin{equation}\label{Eq:Vinogradov}
\sum_{p\leq N} (\log p)e(p\alpha)\ll \left(\frac{N}{\sqrt{q}}+\sqrt{Nq\log q}\right)(\log N)^{3/4}\sqrt{\log\log N}+ Ne^{-(\frac12-\ep)\sqrt{\log N}},
\end{equation}
provided $|\alpha-\frac aq|\leq  \frac 1{q^2}$
which gives 
\[
 \sum_{q^2< p\leq y} \frac{e(-p\alpha)}{p} - \frac 12  \sum_{y< p\leq x/2} \frac{e(p\alpha)}{p}\ll \frac{ (\log x)^{3/4+o(1)} }{\sqrt{q}}+\frac 1{ \log q}=O(1)
\]
as $q>(\log x)^2$. The claim now follows.
\end{proof}
   
 %%%%%%%%%%%%%%%%%%%%%%%%%%%%%%%%%
\subsection{Multiplicative functions for which the bound \eqref{Eq:MVThree} is sharp} \label{sec2.ex2}

\begin{proof} [Proof of the claims in Example   \ref{Exa3}]
We can replace $\alpha$ by $\alpha_q:=a/q$ at a cost of $O(q/x)=O(1)$ on the right-hand side of \eqref{Eq:MVTwo with small changes}.
Writing $n=kr+j$ for $k\geq 1$ we have
\[
\sum_{n=kr+1}^{kr+r} \frac{\chi(n)e(n\alpha_q)}{n}=e(kr\alpha_q) \sum_{j=1}^r \frac{\chi(j)e(j\alpha_q)}{kr+j}=\frac{e(kr\alpha_q) \sum_{j=1}^r\chi(j)e(j\alpha_q)}{kr}+O\bigg( \frac 1{k^2}\bigg).
\]
Now $\sum_{j=1}^r\chi(j)e(j\alpha_q) = \sum_{j=1}^r\chi(j)e(\frac{jb}r)+O(r^2|\alpha_q -\frac br|)=\overline{\chi}(b)\tau(\chi) +O(r/q)$,  and therefore
\[
\sum_{ r<n\leq q^2} \frac{\chi(n)}{n} e(n\alpha_q) = \overline{\chi}(b)\tau(\chi) \sum_{1\leq k\leq q^2/r} \frac{e(\frac{kar}q) }{kr} +O(1).
\]
Now $A=ar=bq\pm 1\equiv \pm 1 \pmod q$ and so 
\[
\sum_{1\leq k\leq q^2/r} \frac{e(\frac{kar}q) }k = \sum_{1\leq k\leq q^2/r} \frac{e(\pm \frac{k}q) }{k} = -\log  (1-e(\pm \tfrac{1}q)) +O(r/q) =\log q+O(1),
\]
 by again summing over full periods. Taking the trivial upper bound for the terms with $n\leq r$ in the left hand side of \eqref{Eq:MVTwo with estimates} completes the proof. 
\end{proof}

 %%%%%%%%%%%%%%%%%%%%%%%%%%%%%%%%%%%%%%%%%%%%%
\subsection{First part of Example \ref{Exa4}} \label{sec2.ex3}  
If $\mathcal M=\{ m\in  [2,M]\}$ in Proposition \ref{Prop: small m} below then
$\mathcal L_2\leq ( \sum_{ m\geq 2} \frac{1}{m^2} )^{1/2}= (\frac{\pi^2}6-1)^{1/2}=0.80307\dots$, and so
\[
\bigg|\sum_{\substack{n=mp \\ 2\leq m\leq M \\ y<p\leq x/y}} \frac{f(n)}{n} e(n\alpha)\bigg| \leq  (\mathcal L_2 +o(1)) \sum_{ y<p\leq x/y} \frac{1}{p} .
\]
Therefore \eqref{Eq:BestApprox} with $M=(\log\log x)^3$ yields
 \begin{align*}
 \sum_{ q^2<n\leq x} \frac{f(n)}{n} e(n\alpha) &=   \sum_{\substack{n=mp \\  m\leq M \\ y<p\leq x/y}} \frac{f(n)}{n} e(n\alpha) +O(\log M)\\
 & \geq   (1-\mathcal L_2 +o(1)) \sum_{ y<p\leq x/y} \frac{1}{p}\geq \frac 16 \log(\tfrac{\log x}{\log q})
\end{align*}
for sufficiently large $x$, since the $m=1$ term gives $\sum_{ y<p\leq x/y} \frac 1p \sim \log(\frac{\log x}{\log q})$.
Therefore, no matter what the values of $f$ at the primes $p\leq q^2$, we always have \eqref{eq: ex3partI}.

%%%%%%%%%%%%%%%%%%%%%%%%%%%%%%%%%
\subsection{Multiplicative functions for which the bound in Corollary \ref{Cor:Smallq3} is sharp} \label{sec2.ex2.3}

 \begin{proof} [Proof of the claims in Example   \ref{exa2.3}]   Let $\ell$ be squarefree, $q=6\ell$ with $(6,\ell)=1$ and let $\chi$ be a primitive character mod $\ell$.
Let  $f$ be the multiplicative function for which $f(n)=g(n)\chi(n)$ where $g(2^k)=g(3^k)=-1$ for all $k\geq 1$ and
$g(p^k)=1$ if $p>3$ and $k\geq 1$.  Note that $g$ is periodic modulo $6$. Then we have
 \[
 \sum_{n\leq x} \frac{f(n)e ( \frac nq ) }n = \sum_{n\leq x} \frac{g(n)\chi(n)e ( \frac nq ) }n 
 =\sum_{a=1}^6 g(a)  \sum_{\substack{n\leq x \\ n\equiv a \bmod 6}} \frac{\chi(n)e ( \frac nq ) }n.
 \]
 Taking $n=6m+a$ we can write this last  inner sum as 
 \[
 \frac{\chi(6)} 6  e \bigg( \frac {a}{6\ell}- \frac {b}{\ell}   \bigg) \sum_{0\leq m \leq \frac{x-a}6} \frac{\chi(m+b)e ( \frac {m+b}{\ell} ) }{m+a/6}
 \]
where $b\equiv a/6 \pmod \ell$. This last sum can be partitioned into intervals of length $\ell$ and is easily shown to be
$\frac{\tau(\chi)}\ell \log x+O(1)$. Therefore 
\[
 \sum_{n\leq x} \frac{f(n)e ( \frac nq ) }n = \frac{\chi(6)} 6 \frac{\tau(\chi)}\ell \log x \cdot \sum_{a=1}^6 g(a)   e \bigg( \frac {a}{6\ell}- \frac {b}{\ell}   \bigg)  +O(1) .
\]
Writing  $a=6k/d$ where $(a,6)=6/d$ then $(k,d)=1$ and we can take $b=\frac{k-\ell\cdot (k\ell^{-1})_d}d$ where $(t)_d$ is some residue of $t \pmod d$. Therefore $ \frac {a}{6\ell}- \frac {b}{\ell}  =    \frac{ (k\ell^{-1})_d}d \pmod 1$, and so
\begin{align*}
\sum_{a=1}^6 g(a)   e \bigg( \frac {a}{6\ell}- \frac {b}{\ell}   \bigg)&=1+ g(1) e\left(\frac{\ell^{-1}}{6}\right)+ g(2) e\left(\frac{\ell^{-1}}{3}\right)\\
& \quad \quad \quad + g(3) e\left(\frac{\ell^{-1}}{2}\right)+g(4) e\left(\frac{2\ell^{-1}}{3}\right)+g(5) e\left(\frac{5\ell^{-1}}{6}\right) \\
&= 1+e( \tfrac 16) +e(- \tfrac 16) -e( \tfrac 13) -e(- \tfrac 13)-e( \tfrac 12)=4,
\end{align*}
since $\ell^{-1}\equiv \pm 1 \pmod 6$, $g(1)=g(5)=g(6)=1$ and $g(2)=g(4)=g(3)=-1$. We deduce that 
\[
 \bigg|\sum_{n\leq x} \frac{f(n)e ( \frac nq ) }n\bigg| = \frac{4} 6 \cdot \frac{ \log x}{\sqrt{\ell}}    +O(1)= \frac{4} {\sqrt{6}} \cdot \frac{ \log x}{\sqrt{q}}    +O(1) . \qedhere
\]
\end{proof} 

%%%%%%%%%%%%%%%%%%%%%%%%%%%%%%%%%%%%%%%%%%%%%%%%%%%%%%%%%%%%%%%%%%%%%%%%%%%%%%%%%%%%%%%%%%%%%%%%%%%%%%%%%%%%%%%%%%%%%%%%%%%%%%%%%%%%%%%%%%%%%%%%%%%%%%%%%%%%%%%%%%%%%%%%%%%%%%%%%%%%%%%%%%%%%%%%

\section{Exponential sums over smooth numbers}

\subsection{Auxiliary results about smooth numbers}

In this section, we gather several important results on the distribution of smooth numbers, which will be useful in the proof of Theorem \ref{Thm:Smooth}.  We start by Buchstab's classical identity 
\begin{equation}\label{Eq:Buch}
\Psi(x, y)=1+\sum_{p\leq y} \Psi\left(\frac{x}{p}, p\right),
\end{equation}
which is valid for all $x\geq y\geq 2$. This was used by de Bruijn \cite{Bru51} to prove that   
\begin{equation}\label{Eq:AsymptoticSmoothClassic}
\Psi(x, y) \sim \rho(u)x
\end{equation} for $\exp((\log x)^{5/8+\ep})\leq y\leq x$, where  $u:=\log x/\log y$ and $\rho$ is the Dickman-de Bruijn function, defined as the continuous solution to the  differential delay equation
$u\rho'(u)=-\rho(u-1)$ for $u>1$, with initial condition $\rho(u)=1$ for $u\in [0, 1]$. Note that $\rho(u)=((e+o(1))/(u\log u))^u$ as $u\to \infty$, which shows that $y$-smooth numbers are a very sparse set when $u\to \infty$. 

Hildebrand \cite{Hi86} substantially improved the range of validity of  \eqref{Eq:AsymptoticSmoothClassic} to
 $$\exp((\log\log x)^{3/5+\ep})\leq y\leq x.$$ In particular, an important special case of \eqref{Eq:AsymptoticSmoothClassic} gives 
 \begin{equation}\label{Eq:ChangeRegimSmoothCount}
\Psi(x, \mathcal L(x)^c)= \frac{x}{\mathcal L(x)^{1/(2c)+o(1)}}, 
 \end{equation}
 for any fixed positive constant $c$, where as before $ \mathcal L(x)= \exp(\sqrt{\log x\log\log x})$.
 Canfield,  Erd\H{o}s, and Pomerance \cite{CEP83} proved the following weaker estimate, which is valid in the much larger range $(\log x)^{1+\ep}\leq y
 \leq x$:
 \begin{equation}\label{Eq:CEPSmooth}
\Psi(x, y) = \frac{x}{u^{u+o(u)}}.
\end{equation}
In particular, for any fixed $A>1$ one has 
\begin{equation}\label{Eq:SmoothPowersLog}
\Psi\big(x, (\log x)^A\big) = x^{1-1/A+o(1)}.
\end{equation}
 
Hildebrand and Tenenbaum \cite{HiTe86} established an asymptotic formula for $\Psi(x,y)$ in the larger range $x\geq y\geq 2$ such that $y\to\infty$ and $u\to \infty$ as $x\to \infty$, in terms of the ``saddle point'' $\lambda=\lambda(x, y)>0$, which is defined as the unique solution to 
$$ 
\sum_{p\leq y} \frac{\log p}{p^{\lambda}-1}=\log x.
$$
Moreover, they obtained a simple approximation for $\lambda(x, y)$ on the whole range $x\geq y\geq 2$. In particular  \cite[Lemma 2]{HiTe86} implies that  in the range $\log x< y \leq x$ we have 
\begin{equation}\label{Eq:estimateLambda}
\lambda= 1-\frac{\log(u\log (u+1))+O(1)}{\log y}.
\end{equation}
We will need the following lemma, which is an immediate consequence of their results. 
\begin{lemma}\label{Lem:HildTenebaum2x}
Let $c>1$ be a fixed constant. For all real numbers $x\geq y\geq 2$  we have 
$$ \Psi(cx, y) \ll_c  \Psi(x, y).$$

\end{lemma}
\begin{proof}
  This follows from \cite[Theorem 3]{HiTe86}.
\end{proof}
We shall also use the following nice result of Hildebrand \cite{Hi85}, which shows that  $\Psi(x, y)$ is subadditive as a function of $x$ if $y$ is large enough.
\begin{lemma}[Theorem 4 of \cite{Hi85}] \label{Lem:HildebrandSubAdd}
There exists a constant $y_0\geq 2 $ such that for all real numbers $y\geq y_0$ and $x, z\geq y$  we have 
$$ \Psi(x+z, y)\leq \Psi(x, y)+\Psi(z, y).$$
    
\end{lemma}
Let $\Psi_q(x, y)$ denote the number of $y$-smooth integers $\leq x$ which are coprime to a positive integer $q$.  Tenenbaum \cite{Te93} showed that 
\begin{equation}\label{Eq:TenSmoothCrible}
 \Psi_q(x, y)\asymp \frac{\phi(q)\Psi(x, y)}{q},   
\end{equation}
under the  general conditions 
$$ 
P^+(q)\leq y
\quad \textup{ and } \log (1+\omega(q))\leq   \frac{c\log y}{\log (u+1)},
$$
for some positive constant $c$.
In particular, \eqref{Eq:TenSmoothCrible} is valid for all real numbers $x\geq y\geq 2$ and positive integers $q$ in the range $q\leq \min\Big(y, \exp\big(y^{c/\log(u+1)}\big)\Big).$
We shall also make repeated use of the following nice ``localization'' bound of
La Bretèche and Tenenbaum \cite{BrTe05}. 
\begin{lemma}[Th\'eor\`eme 2.4 (i) of \cite{BrTe05}]\label{Lem:LocalSmoothBreTen}
Let $x\geq y\geq 2$ be real numbers, and $d$, $q$ be positive integers such that  $1\leq d\leq x/y$,  $P^+(q)\leq y$ and $\omega(q)\ll y^{1/2-\ep}$. Then we have 
$$
\Psi_q\left(\frac{x}{d}, y\right)\ll \frac{\prod_{p|q}(1-p^{-\lambda})}{d^{\lambda}}\Psi(x, y).
$$
Moreover, when $q=1$, this bound is valid in the whole range $x\geq y\geq 2$ and $1\leq d\leq x$.

\end{lemma}

Finally, we state some results on the distribution of smooth numbers in  arithmetic progressions. For real numbers $x\geq y\geq 2$ and positive coprime integers $a, q$ we define 
$$\Psi(x, y; q, a):= \sum_{\substack{n\leq x\\ P^+(n)\leq y\\ n\equiv a \bmod q}} 1. $$
Refining a result of Soundararajan \cite{So08}, Harper \cite{Har12a} proved that 
\begin{equation}\label{Eq:HarperSmoothAP}
  \Psi(x, y; q, a)\sim \frac{\Psi_q(x, y)}{\phi(q)}  
\end{equation}
in the range $x\geq y\geq 2$, and $q\leq y^{4\sqrt{e}-\ep}$ such that $\log x/\log q\to \infty$ as $x\to \infty$. We will also need a strong bound on the error term of this asymptotic formula.  The following lemma is an immediate consequence of the ``Bombieri-Vinogradov'' theorem for smooth numbers, proved by Harper \cite{Har12b}, which improves upon results of Fouvry and Tenenbaum \cite{FoTe91} and of Granville  \cite[Theorem 2]{Gr93}.
\begin{lemma}\label{Lem:BVHarperSmooth} There exists positive constants  $K$ and $c$ such that the following holds. For all real numbers $x, y$ such that $x$ is large enough and  $(\log x)^K\leq y\leq x$, and all positive integers $a, q$ such that $(a, q)=1$ and  $q\leq y^c$ we have  
\begin{align*}
\Psi(x, y; q, a)&=\frac{\Psi_q(x, y)}{\phi(q)} \\
& \quad +O_A\left(\Psi(x, y)\bigg( \frac{e^{-cu/(\log(u+1))^2}}{(\log x)^A} +y^{-c}\bigg)+\sqrt{\Psi(x, y)} q (\log x)^{7/2}\right), 
\end{align*}
for any positive constant $A>1$.
    
\end{lemma}
\begin{proof}
    This is immediate consequence of \cite[Theorem 1]{Har12b} by taking $Q=q$ and noting that $\Psi(x,y)\geq y$. 
\end{proof}
Finally, assuming the Generalized Riemann hypothesis, we have the better estimate (see \cite[Eq. (6.1)]{Gr93})
\begin{equation}\label{Eq:SmoothAPUnderGRH}
\Psi(x, y; q, a) = \frac{\Psi_q(x, y)}{\phi(q)}\left(1+O\left(\frac{q (\log y)^2}{\sqrt{y}}\right)\right),  
\end{equation}
which holds uniformly for $x\geq y\geq 2$ and $q\leq y^{1/2}/(\log y)^2$.

%%%%%%%%%%%%%%%%%%%%%%%%%%%%%%%

\subsection{Improving the upper bound \eqref{Eq:AndrewRegis}}
By optimizing one step in the proof of \eqref{Eq:AndrewRegis} we obtain the following proposition, whose proof we shall include for completeness. 
\begin{proposition}
\label{Pro:ImprovingRegis}
Let $x$ be large, $2\leq y\leq x$ and suppose that  $|\alpha-a/q|\leq 1/q^2$ where $(a, q)=1$ and $q\leq x$. Then  
 for all multiplicative functions $f:\mathbb{N}\to \mathbb{U}$ one has
\begin{align*}
\sum_{\substack{n\leq x\\ P^+(n)\leq y}}  f(n) e(n\alpha)&\ll 
\left(\frac{\sqrt{xy}}{\log y}+ \frac{x}{\sqrt{q}}+\sqrt{xq\log(2x/q)}\right) \sqrt{\log y}  + \frac{x}{\mathcal{L}(x)^{1/2+o(1)}}. 
\end{align*}   
\end{proposition}
Before proving this result we need the following basic lemma. Here and throughout, $\|\cdot\|$ denotes the distance to the nearest integer.
\begin{lemma}\label{Lem:StandardExpSum}
Let $q\geq 2$, and assume that $|\alpha-a/q|\leq 1/q^2$ where $(a, q)=1$. For all real numbers $H, M\geq 2$ we have 
$$
\sum_{h \leq H} \min\left(M, \frac{1}{\|h\alpha\|}\right)\ll M+\frac{HM}{q}+H\log\big(\min(q,M)\big)+ q\log\left(2+\frac{HM}{q}\right).
$$
\end{lemma}
\begin{proof}
Write $\alpha=a/q+\beta$ where $|\beta|\leq 1/q^2$. First assume that $H\leq q/2$. Then for all $h\leq H$ we have  $h\alpha=ha/q+h\beta$ and $|h\beta|\leq 1/(2q)$. Thus the set $\{ha/q+h\beta: h\leq H\}$ is $1/(2q)$-spaced modulo $1$ and hence 
\begin{align}\label{Eq:StandardExpSum1}
\sum_{h \leq H} \min\left(M, \frac{1}{\|h\alpha\|}\right)&\ll M+ \sum_{1\leq \ell\leq H} \min\left(M, \frac{q}{\ell}\right)
\nonumber\\&
\ll M+ M \sum_{1\leq \ell \leq \min(H, q/M)}1+ q\sum_{\min(H, q/M)< \ell \leq H} \frac{1}{\ell}
\nonumber\\
&\ll M+ q\log\left(2+\frac{HM}{q}\right).
\end{align}
We now assume that $H>q/2$. By \eqref{Eq:StandardExpSum1} it suffices to bound the contribution of the terms $q/2<h\leq H$. Write $h= mq+r$ where $m\geq 1$ and $|r|\leq q/2$. Then $\|h\alpha\|= \|ra/q+ r\beta+ mq\beta\|$. For a fixed $1\leq m\leq 2H/q$, the set $\{ra/q+r\beta+mq\beta: |r|\leq q/2\}$ is again $1/(2q)$-spaced modulo $1$, since $|r\beta|\leq 1/(2q)$. Thus, we get 
\begin{align}\label{Eq:StandardExpSum2}
\sum_{q/2<h \leq H} \min\left(M, \frac{1}{\|h\alpha\|}\right)&\ll \sum_{1\leq m\leq 2H/q}\left(M+ \sum_{\ell=1}^{q-1}\min\left(M, \frac{q}{\ell}\right)\right)\nonumber\\
&\ll \frac{HM}{q}+ H\log\big(\min(q, M)\big),
\end{align}
where we have used that  
$$ \sum_{\ell=1}^{q-1}\min\left(M, \frac{q}{\ell}\right)\leq q\sum_{\ell=1}^{q-1}\frac{1}{\ell}\ll q\log q,$$
if $M\geq q$ and 
$$ \sum_{\ell=1}^{q-1}\min\left(M, \frac{q}{\ell}\right)\leq  M \sum_{1\leq \ell \leq  q/M}1+ q\sum_{q/M<\ell\leq q-1}\frac{1}{\ell}\ll q\log M,$$
otherwise.
Combining  \eqref{Eq:StandardExpSum1} and \eqref{Eq:StandardExpSum2} completes the proof.
\end{proof}
\begin{proof}[Proof of Proposition \ref{Pro:ImprovingRegis}]
 It follows from \eqref{Eq:ChangeRegimSmoothCount} that
 $$\sum_{\substack{n\leq x\\ P^+(n)\leq y}}  f(n) e(n\alpha)= \sum_{\substack{n\leq x\\ \mathcal L(x)<P^+(n)\leq y}}  f(n) e(n\alpha)+ O\bigg(\frac x {\mathcal L(x)^{1/2+o(1)}}\bigg),
   $$
   so we may assume that $y>\mathcal L(x)$, otherwise the result follows readily.
   Let $\mathcal L(x)<p\leq y$ be a prime number. The contribution of the terms $n$ such that $p^2\mid n$ is 
   $$ \ll \sum_{\mathcal L(x)<p\leq \min(y, \sqrt{x})} \frac{x}{p^2}\ll \frac{x}{\mathcal L(x)}, $$
   which is acceptable. Therefore, we deduce that 
   \begin{align}
\label{Eq:ImprovingRegisSplit}
\sum_{\substack{n\leq x\\ P^+(n)\leq y}}  f(n) e(n\alpha)= \sum_{\mathcal L(x)<p\leq y} f(p)\sum_{\substack{ n\leq x/p \\ P^{+}(n)<p}} f(n) e(n\alpha p) + O\bigg(\frac x {\mathcal L(x)^{1/2+o(1)}}\bigg).
   \end{align}
   We now split the sum over primes in dyadic intervals. By the Cauchy-Schwarz inequality we get 
\begin{align}\label{Eq:ImprovingRegisDyadic}
   \sum_{M<p\leq 2M}\Bigg|\sum_{\substack{ n\leq x/p \\ P^{+}(n)<p}} f(n) e(n\alpha p)\Bigg|\ll \frac{\sqrt{M}}{\sqrt{\log M}}\Bigg(\sum_{M<p\leq 2M}\Bigg|\sum_{\substack{ n\leq x/p \\ P^{+}(n)<p}} f(n) e(n\alpha p)\Bigg|^2\Bigg)^{1/2}.
   \end{align}
  Next, we bound the sum over primes on the right hand side by the same sum over all integers $M\leq m\leq 2M$, by positivity\footnote{Our improvement comes from using Cauchy-Schwarz before this step, unlike in \cite{dlB98}.}. Opening up the square we obtain 
\begin{align}\label{Eq:ImprovingRegisIntegers}
  &\sum_{M<m\leq 2M}
\Bigg|\sum_{\substack{ n\leq x/m \\ P^{+}(n)<m}} f(n) e(mn\alpha )\Bigg|^2= \sum_{M<m\leq 2M}
   \sum_{\substack{ n_1, n_2\leq x/m \\ P^{+}(n_1n_2)<m}} f(n_1)\overline{f(n_2)} e(m(n_1-n_2)\alpha)\nonumber
   \\
   & \quad \quad \quad \quad =\sum_{n_1, n_2< \frac{x}{M}} f(n_1)\overline{f(n_2)}\sum_{\max(P^+(n_1n_2), M)<m \leq \min(\frac{x}{n_1}, \frac{x}{n_2}, 2M)}e(m(n_1-n_2)\alpha) \nonumber\\
   & \quad \quad \quad \quad \ll \sum_{n_1< n_2<x/M} \min\left(M, \frac{1}{\|(n_2-n_1)\alpha\|}\right) +x\\
   & \quad \quad \quad \quad \ll \frac{x}{M}\sum_{h<x/M}\min\left(M, \frac{1}{\|h\alpha\|}\right) +x. \nonumber  
  \end{align}
  Combining this estimate with Lemma \ref{Lem:StandardExpSum} and  \eqref{Eq:ImprovingRegisDyadic} implies that 
$$\sum_{M<p\leq 2M}\Bigg|\sum_{\substack{ n\leq x/p \\ P^{+}(n)<p}} f(n) e(n\alpha p)\Bigg|\ll \frac{\sqrt{xM}}{\sqrt{\log M}}+ \frac{x}{\sqrt{q\log M}}+ \frac{\sqrt{xq\log (x/q)}}{\sqrt{\log M}}+\frac{x}{\sqrt{M}}.$$ 
Summing this bound over the $M$'s of the form $M=2^j\mathcal L(x)$ with $0\leq j\leq \frac{1}{\log 2}\log (y/\log \mathcal L(x))$ and combining this with \eqref{Eq:ImprovingRegisSplit} completes the proof. 
\end{proof}

%%%%%%%%%%%%%%%%%%%%%%%%%%%%%%%%%%%%%%%%%%%%%%%%%%%%%%%%%%%

\subsection{Proof of Theorem \ref{Thm:Smooth}} Since the result holds trivially if $q$ is bounded, we assume throughout that $q$ is sufficiently large. 
 We also replace $\alpha$ by $a/q$ at the expense of an error term of size $O(\Psi(x, y)/q)$ since  $|\alpha-a/q|\leq 1/(qx)$; that is, we take $\alpha=a/q$ throughout the proof. 
Moreover, since we included the term $\Psi(x, q^2)$ in the upper bound, we only need to bound the sum 
$$S(x, y;\alpha):=\sum_{\substack{n\leq x\\ q^2<P^+(n)\leq y}}  f(n) e(n\alpha)= \sum_{q^2< p\leq y}\sum_{\substack{m\leq x/p\\ P^+(m)\leq p}}  f(mp) e(mp\alpha).$$
We start by handling the  contribution of the $m$'s such that $p|m$ in the above sum, which is  
$$ \ll \sum_{q^2<p\leq y} \ \Psi\left(\frac{x}{p^2}, p\right)\ll \sum_{q^2<p\leq y} \ \frac{1}{p^{\lambda}}\Psi\left(\frac{x}{p}, p\right) \ll \frac{1}{q^2}\sum_{q^2<p\leq y} \Psi\left(\frac{x}{p}, p\right)\ll \frac{\Psi(x, y)}{q^2}, $$
by \eqref{Eq:Buch}, Lemma  \ref{Lem:LocalSmoothBreTen}   and since $1-\lambda\ll 1/\log q$ by \eqref{Eq:estimateLambda} and our assumption $q^{\log(u+1)/A}\leq y$. Therefore, we obtain 
\begin{align}\label{Eq:SplitSumS}
|S(x, y;\alpha)|&\leq \Big|\sum_{q^2< p\leq y} f(p) \sum_{\substack{m\leq x/p\\ P^+(m)< p}}  f(m) e(mp\alpha)\Big| + O\left(\frac{\Psi(x, y)}{q^2}\right)\nonumber\\
&\leq 
\sum_{q^2< p\leq y} \Big| \sum_{\substack{m\leq x/p\\ P^+(m)< p}}  f(m) e(mp\alpha)\Big| + O\left(\frac{\Psi(x, y)}{q^2}\right).
\end{align}
We now split the above sum over the primes $q^2<p\leq y$ into short intervals $I_j=(t_j,t_{j+1}]$ for $0\leq j\leq J$,
where $t_0=q^2$, $t_J\sim y$, and $t_{j+1}=t_j+ t_j^{2/3}$ for each $j$. We also select the integer $J_0$ so that 
$t_{J_0}\sim \min\{ x/q,y\}$ so that $J_0=J$ if $y\leq x/q$.

Let $0\leq j\leq J$ and assume that $p\in I_j$. Then we observe that 
\begin{equation}\label{Eq:ApproxSumShortPrimesSmooth}
\sum_{\substack{m\leq x/p\\ P^+(m)< p}} f(m) e(mp\alpha)=\sum_{\substack{m\leq x/t_j\\ P^+(m)<  t_j}} f(m) e(mp\alpha) +E_1(j),
\end{equation}
where
\begin{align*}
 |E_1(j)|&\leq  \Psi\left(\frac{x}{t_j}, t_j\right)-\Psi\left(\frac{x}{t_{j+1}}, t_j\right)+ \sum_{\substack{m\leq x/t_j\\ t_j\leq P^+(m)<t_{j+1}}}1\\
 & \leq \Psi\left(\frac{x}{t_{j+1}}+ \frac{x}{t_j^{4/3}}, t_j\right)-\Psi\left(\frac{x}{t_{j+1}}, t_j\right)+ \sum_{t_j\leq p_1<t_{j+1}} \Psi\left(\frac{x}{p_1t_j}, p_1\right),
 \end{align*}
since $x/t_j\leq x/t_{j+1}+  x/t_j^{4/3}$.

First, we assume that $q^2\leq t_j\leq x^{1/10}.$  In this case we have  $x/t_j^{4/3}>t_j$ for all $0\leq j\leq J$. Therefore, using Lemmas \ref{Lem:HildebrandSubAdd} and \ref{Lem:LocalSmoothBreTen} we obtain
\begin{align*}
|E_1(j)| &\ll \Psi\left(\frac{x}{t_{j}^{4/3}}, t_j\right)+ \frac{1}{t_j^{2\lambda}}\sum_{t_j\leq p_1<t_{j+1}} \Psi\left(x, p_1\right)  
\\
&\ll \frac{\Psi(x, y)}{t_j^{4\lambda/3}} + \frac{\Psi\left(x, y\right)}{t_j^{2\lambda}}\sum_{t_j\leq p_1<t_{j+1}} 1  \ll \frac{\Psi(x, y)}{t_j^{2\lambda-2/3}}.
\end{align*}
 Therefore, we obtain
\begin{align}\label{Eq:ErrorSumSmooth3}
\sum_{\substack{0\leq j\leq J\\ t_j\leq x^{1/10}}}\sum_{t_j< p\leq t_{j+1}} |E_1(j)|&\ll  \Psi(x, y)\sum_{j=0}^J\sum_{t_j< p\leq t_{j+1}}\frac{1}{t_j^{2\lambda-2/3}} \ll \Psi(x, y)\sum_{j=0}^J\frac{t_{j+1}-t_j}{t_j^{2\lambda-2/3}} \nonumber\\
&\ll \Psi(x, y)\int_{q^2}^y \frac{dt}{t^{2\lambda-2/3}}\ll \frac{\Psi(x, y)}{q^{2/3}},
\end{align}
where we have used that $1-\lambda\ll  1/\log q$ and $q$ is sufficiently large. 

Now, we assume that $ x^{1/10}< t_j\leq y$. In this case we know that 
$\Psi(x, y)\asymp x$, and we use that 
\begin{align*}
|E_1(j)|&\leq \sum_{x/t_{j+1}<m\leq x/t_j}1+ \sum_{\substack{m\leq x/t_j\\ t_j\leq P^+(m)< t_{j+1}}}1  \\
&\ll \left(\frac{x}{t_{j}}- \frac{x}{t_{j+1}}+1\right)+ \sum_{t_j\leq  p_1<t_{j+1}} \ \sum_{\substack{m\leq x/t_j\\ p_1|m}}1 \nonumber\\
&\ll \frac{x}{t_{j}^{4/3}}+1 + \frac{x}{t_j^2}\sum_{t_j\leq p_1<t_{j+1}}1\ll \frac{x}{t_{j}^{4/3}}+1.
\end{align*}
Hence, we get  
\begin{align}
\label{Eq:SumErrorjSmooth}
\sum_{\substack{0\leq j\leq J\\  t_j>x^{1/10}}}\sum_{t_j< p\leq t_{j+1}} |E_1(j)|&\ll  x\sum_{j=0}^J\sum_{t_j< p\leq t_{j+1}}\frac{1}{t_j^{4/3}}+ \frac{y}{\log y} \nonumber \\
& \ll x\sum_{j=0}^J\frac{t_{j+1}-t_j}{t_j^{4/3}} +\frac{y}{\log y}\\
&\ll x\int_{q^2}^y \frac{dt}{t^{4/3}}+\frac{y}{\log y}\ll \frac{\Psi(x, y)}{q^{2/3}}+\frac{y}{\log y}, \nonumber
\end{align}
Inserting the estimates \eqref{Eq:ErrorSumSmooth3} and \eqref{Eq:SumErrorjSmooth} into \eqref{Eq:SplitSumS} gives 
\begin{equation}
\label{Eq:MainSumSmooth}
|S(x, y;\alpha)|\leq \sum_{j=0}^J\sum_{t_j< p\leq t_{j+1}} \Big| \sum_{\substack{m\leq x/t_j\\ P^+(m)< t_j}}  f(m) e(mp\alpha)\Big| + O\left(\frac{\Psi(x, y)}{q^{2/3}}+\frac{y}{\log y} \right).
\end{equation}

Next, we handle the main term on the right hand side of \eqref{Eq:MainSumSmooth}, which equals 
\begin{align}\label{Eq:MainResidueSmooth} \sum_{j=0}^J\sum_{t_j< p\leq t_{j+1}} \Big| \sum_{\substack{m\leq x/t_j\\ P^+(m)< t_j}}  f(m) e(mp\alpha)\Big|&=\sum_{j=0}^J\sum_{\substack{1\leq r\leq q-1\\ (r, q)=1}} \Big| \sum_{\substack{m\leq x/t_j\\ P^+(m)< t_j}}  f(m) e(mr\alpha)\Big|\sum_{\substack{t_j< p\leq t_{j+1}\\ p\equiv r \bmod q}} 1 \nonumber\\ 
&\ll \sum_{j=0}^J\frac{t_{j+1}-t_j}{\phi(q)\log t_j}\sum_{\substack{1\leq r\leq q-1\\ (r, q)=1}} \Big| \sum_{\substack{m\leq x/t_j\\ P^+(m)< t_j}}  f(m) e(mr\alpha)\Big|,
\end{align}
where the last bound follows from the Brun-Titchmarsh Theorem, since $t_{j+1}-t_j=t_j^{2/3}\geq q^{4/3}$ for all $0\leq j\leq J$. We now fix $0\leq j\leq J$  and use the Cauchy-Schwarz inequality which gives
\begin{align}\label{Eq:MainTermCSSmooth}
&\Bigg(\sum_{\substack{1\leq r\leq q-1\\ (r,q)=1}} \bigg|\sum_{\substack{m\leq x/t_j\\ P^+(m)< t_j}} f(m) e(m r \alpha)\bigg|\Bigg)^2
\leq \phi(q)\sum_{\substack{1\leq r\leq q-1\\ (r,q)=1}} \bigg|\sum_{\substack{m\leq x/t_j\\ P^+(m)< t_j}} f(m) e(m r \alpha)\bigg|^2\nonumber \\
& \quad \quad \quad \quad \quad \quad= \phi(q)\sum_{\substack{m_1, m_2\leq x/t_j\\ P^+(m_1m_2)< t_j}} f(m_1)\overline{f(m_2)}\sum_{\substack{1\leq r\leq q-1\\ (r,q)=1}}e\Big((m_1-m_2)\frac{ra}{q}\Big)\nonumber\\
& \quad \quad \quad \quad \quad \quad \leq  \phi(q)\sum_{\substack{m_1, m_2\leq x/t_j\\ P^+(m_1m_2)< t_j}} |c_q(m_1-m_2)|,
\end{align}
where $c_q(\ell)$ is Ramanujan's sum, defined by 
\begin{equation}\label{Eq:DefRamanujan}
c_q(\ell):= \sum_{\substack{1\leq b\leq q-1\\ (b,q)=1}} e\left(\frac{b\ell}{q}\right)= \sum_{d\mid (q, \ell)} d\mu\left(\frac{q}{d}\right),
\end{equation}
where the second identity follows by M\"obius inversion  (see for example   \cite[Ch. 26, Eq. (6)]{D80}. 
Therefore, by \eqref{Eq:MainResidueSmooth}  and \eqref{Eq:MainTermCSSmooth} we obtain 
\begin{align}
\label{Eq:MainSumSmooth2}
|S(x, y;\alpha)| &\ll \sum_{j=0}^J\frac{t_{j+1}-t_j}{\sqrt{\phi(q)}\log t_j}\Bigg(\sum_{\substack{m_2\leq m_1\leq x/t_j\\ P^+(m_1m_2)< t_j}}|c_q(m_1-m_2)|\Bigg)^{1/2} \nonumber\\
& \quad \quad \quad + O\left( \frac{\Psi(x, y)}{q^{2/3}}+ \frac{y}{\log y}\right).
\end{align} 
We now bound the right hand side of \eqref{Eq:MainTermCSSmooth} for those $t_j\leq x/q$ (that is for $j\leq J_0$). We note that 
\begin{equation}\label{Eq:OffDiagonalSmooth}
\begin{aligned}
 & \sum_{\substack{m_1,m_2\leq x/t_j\\ P^+(m_1m_2)< t_j}}|c_q(m_1-m_2)|
=  \sum_{b_1,b_2 \bmod q} |c_q(b_1-b_2)| 
\sum_{\substack{m\leq x/t_j\\ m \equiv b_1 \bmod q\\ P^+(m)< t_j}} 1 \cdot \sum_{\substack{n\leq x/t_j\\ n \equiv b_2 \bmod q\\ P^+(n)< t_j}} 1.
\end{aligned}
\end{equation}
First, assume that $t_j\geq x^{1/10}$. In this case we have for all $b\pmod q$
\begin{equation}
\label{Eq:BoundSmoothAPShort1}
\sum_{\substack{m\leq x/t_j\\ m \equiv b \bmod q\\ P^+(m)< t_j}} 1 \leq \sum_{\substack{m\leq x/t_j\\ m \equiv b \bmod q}} 1\ll \frac{x}{t_jq}\ll \frac{\Psi\left(x/t_j, t_j\right)}{q} ,
\end{equation} 
since $t_j\leq x/q$ and $\Psi(x/t_j, t_j)\asymp x/t_j$ by our assumption on $t_j$. 
 
We now assume that  $t_j< x^{1/10}$. Fix a residue class $b \pmod q$ and let $d= (b, q)$. Then by \eqref{Eq:HarperSmoothAP}  we have
\begin{equation}\label{Eq:BoundSmoothAPShort2}
\sum_{\substack{m\leq x/t_j\\ m \equiv b \bmod q\\ P^+(m)< t_j}} 1= \Psi\left(\frac{x}{t_j d}, t_j; \frac{q}{d}, \frac{b}{d}\right) \ll \frac{1}{\phi(q/d)} \Psi_{q/d}\left(\frac{x}{t_jd},t_j\right). 
\end{equation}
 Moreover, note that $d\leq q=x^{o(1)}\leq x/t_j^2$ and hence by Lemma \ref{Lem:LocalSmoothBreTen} we get 
\begin{align*}
\Psi_{q/d}\left(\frac{x}{t_jd},t_j\right) 
 &\ll \frac{\prod_{p|(q/d)}(1-p^{-\lambda})}{d^{\lambda}} \Psi\left(\frac{x}{t_j}, t_j\right) \\
 & \ll \frac {\phi(q/d)}{d^\lambda \cdot (q/d)}\Psi\left(\frac{x}{t_j}, t_j\right) \ll   \frac {\phi(q/d)} {q}\Psi\left(\frac{x}{t_j}, t_j\right),
 \end{align*}
since $1-\lambda \ll 1/\log q$, which implies that $d^{\lambda}\asymp d$ and $$\prod_{p|(q/d)}\big(1-p^{-\lambda}\big) \asymp \exp\Big(-\sum_{p|q/d} \frac{1}{p^{\lambda}}\Big)\asymp \exp\Big(-\sum_{p|(q/d)} \frac{1}{p}\Big)\asymp \frac{\phi(q/d)}{q/d}.$$
Inserting this bound in \eqref{Eq:BoundSmoothAPShort2} shows that \eqref{Eq:BoundSmoothAPShort1} holds in this case as well, for all residue classes $b \pmod q$.

Thus, combining the estimates \eqref{Eq:OffDiagonalSmooth} and \eqref{Eq:BoundSmoothAPShort1}   implies that for all $0\leq j\leq J_0$,
we obtain 
\begin{align}\label{Eq:OffDiagonalSmooth2}
\sum_{\substack{m_1,m_2\leq x/t_j\\ P^+(m_1m_2)< t_j}}|c_q(m_1-m_2)|& \ll   \frac {\Psi (x/t_j,t_j)^2}{q^2} \sum_{b_1,b_2 \bmod q} |c_q(b_1-b_2)| \nonumber\\
&
=\frac {\phi(q)2^{\omega(q)}}{q} \Psi\left(\frac{x}{t_j}, t_j\right)^2 ,
 \end{align}
since
\begin{align*}
  \sum_{b_1,b_2 \bmod q} |c_q(b_1-b_2)| &= q\sum_{h\bmod q}|c_q(h)|= q  \sum_{d|q} |c_q(d)| \sum_{\substack{h \bmod q \\ (h,q)=d}} 1
 \\
& = q \phi(q)\sum_{d|q} \mu(q/d)^2= q \phi(q) 2^{\omega(q)},
\end{align*}
where we have used that $c_q(h) =\mu(q/(q, h))\phi(q)/\phi(q/(q, h))$ for all positive integers $h$.

We now insert the estimate  \eqref{Eq:OffDiagonalSmooth2} in \eqref{Eq:MainTermCSSmooth} and then back into \eqref{Eq:MainResidueSmooth} to get 
\begin{align*}
\sum_{j=0}^{J_0}\sum_{t_j< p\leq t_{j+1}} \Big| \sum_{\substack{m\leq x/t_j\\ P^+(m)< t_j}}  f(m) e(mp\alpha)\Big|
\ll \frac{2^{\omega(q)/2}}{\sqrt{q}} \sum_{j=0}^{J_0} \frac{t_{j+1}-t_j}{\log t_j}\Psi\left(\frac{x}{t_j}, t_j\right).
\end{align*}
Finally, since $2/3>7/12$, we use Huxley's prime number theorem in short intervals \cite{Hu72}, together with Lemma \ref{Lem:HildTenebaum2x} and \eqref{Eq:Buch} to obtain
\begin{align}\label{Eq:Buch+Huxley}
\sum_{j=0}^{J_0} \frac{t_{j+1}-t_j}{\log t_j}\Psi\left(\frac{x}{t_j}, t_j\right)& \ll \sum_{j=0}^{J_0}\sum_{t_j<p\leq t_{j+1}}  \Psi\left(\frac{2x}{p}, p\right)\nonumber\\ &\ll \sum_{q^2<p\leq y}  \Psi\left(\frac{x}{p}, p\right) \leq \Psi(x, y),  
\end{align}
so that 
\begin{equation}\label{Eq:UpToJ0}
\sum_{j=0}^{J_0}\sum_{t_j< p\leq t_{j+1}} \Big| \sum_{\substack{m\leq x/t_j\\ P^+(m)< t_j}}  f(m) e(mp\alpha)\Big| \ll
\frac{2^{\omega(q)/2}}{\sqrt{q}} \Psi(x, y).
\end{equation}
Combining this estimate with \eqref{Eq:MainSumSmooth} completes the proof when $J_0=J$, that is, when $y\leq x/q$,
since the error term $y/\log y$ is $\ll \Psi(x,y)/\sqrt{q}$ in our range.

Now assume that $x/q<y\leq x$. It remains to handle the contribution of the $j$'s such that $x/q<t_j\leq y$.  Writing $h=m_1-m_2$ and dropping the condition $P^+(m_1m_2)<t_j$ we get 
\begin{equation}\label{Eq:DoubleSumSingleRamanujan}
\sum_{\substack{m_2\leq m_1\leq x/t_j\\ P^+(m_1m_2)< t_j}}|c_q(m_1-m_2)|\leq \frac{x}{t_j}\sum_{0\leq h\leq x/t_j }|c_q(h)|.
\end{equation}
We now bound the sum over $h$ on the right hand side. The contribution of the term $h=0$ (which corresponds to the diagonal terms $m_1=m_2$) is $\phi(q)$. Moreover, since $c_q(h) =\mu(q/(q, h))\phi(q)/\phi(q/(q, h))$,  the contribution of the terms $0<h\leq x/t_j$ equals 
\begin{align*} \phi(q)\sum_{d\mid q}\frac{|\mu(q/d)|}{\phi(q/d)}\sum_{\substack{0< h\leq x/t_j\\ (h, q)=d}} 1 &\leq \frac{x\phi(q)}{t_j}  \sum_{d\mid q}\frac{|\mu(q/d)|}{d\phi(q/d)}\\
&=\frac{x \phi(q)}{qt_j}  \sum_{d\mid q}\frac{|\mu(d)|d}{\phi(d)} \ll \frac{x \sqrt{\phi(q)}2^{\omega(q)}}{\sqrt{q}t_j},
\end{align*}
since 
$|\{0< h\leq x/t_j : \  (h, q)=d\}|\leq |\{0< h\leq x/t_j : \   d\mid h\}|\leq x/(dt_j),$
and 
$$\sum_{d\mid q}\frac{|\mu(d)|d}{\phi(d)}= \prod_{p\mid q}\left(1+\frac{p}{p-1}\right)= 2^{\omega(q)}\prod_{p\mid q}\left(1+\frac{1}{2(p-1)}\right)\ll 2^{\omega(q)}\sqrt{\frac{q}{\phi(q)}}.  $$
Therefore, we obtain
\begin{equation}\label{Eq:BoundShortSumRamanujan}\sum_{0\leq h\leq x/t_j }|c_q(h)|\ll  \phi(q)+ \frac{x \sqrt{\phi(q)}2^{\omega(q)}}{\sqrt{q}t_j}. \end{equation}
Combining this estimate with \eqref{Eq:DoubleSumSingleRamanujan}, we deduce that the contribution to the right hand side of \eqref{Eq:MainSumSmooth2} of the terms $j$ such that $x/q< t_j\leq y$ (that is, $J_0<j\leq J$) is
\begin{align*}
&\sum_{\substack{j \leq J\\ x/q< t_j\leq y}}\frac{t_{j+1}-t_j}{\sqrt{\phi(q)}\log t_j}
\bigg(\frac{\sqrt{x}{\sqrt{\phi(q)}}}{\sqrt{t_j}}+ \frac{x}{t_j} 2^{\omega(q)/2}\frac{\phi(q)^{1/4}}{q^{1/4}}\bigg)\\
&\ll \sqrt{x}\int_{x/q}^y\frac{dt}{\sqrt{t}\log t}
+ x \frac{2^{\omega(q)/2}}{(q\phi(q))^{1/4}}\int_{x/q}^{y}\frac{dt}{t\log t}\\
&\ll \frac{\sqrt{xy}}{\log x}+ x \frac{2^{\omega(q)/2}}{(q\phi(q))^{1/4}} \log\bigg(\frac{\log y}{\log (x/q)}\bigg)\\
&\ll \frac{\sqrt{xy}}{\log x}+ \frac{x}{\log x} \frac{2^{\omega(q)/2} \log q}{(q\phi(q))^{1/4}} \ll  \frac{2^{\omega(q)/2} }{\sqrt{q}} \Psi(x,y) 
\end{align*}
as $\log x/\log q \geq (\log\log q)^{1/4} \gg (q/\phi(q))^{1/4}$ and $\Psi(x,y)\gg x$.
Combining this estimate with \eqref{Eq:MainSumSmooth2} and  \eqref{Eq:UpToJ0} completes the proof.

\begin{proof}[Proof of Corollary \ref{Cor:Smooth}]
Assume that $q\leq \mathcal L(x)^{c_0}$ is sufficiently large, where $c_0=\frac 1{\sqrt{2}}-\ep$.
Theorem \ref{Thm:Smooth} implies that 
\[
\sum_{\substack{n\leq x\\ P^+(n)\leq y}}  f(n) e(n\alpha)\ll  \frac{\sqrt{xy}}{\log x}  1_{y>x/q} + 
\frac{2^{\omega(q)/2}}{\sqrt{q}}\Psi(x, y) +\Psi(x, q^2).
 \]
So our goal is to show that 
\begin{equation}\label{Eq:ThirdTermSmoothSmall}
\frac{\Psi(x, y)}{\sqrt{q}}\gg\Psi(x, q^2),
\end{equation} for all real numbers $y$ such that $ (2q)^{\log(u+1)/A}\leq y\leq x$, where $A>1$ is a fixed constant.

Suppose that $\mathcal L(x)^{c_1}\leq y\leq x,$ where $c_1=10/\ep$. Then it follows from \eqref{Eq:ChangeRegimSmoothCount} that 
$$ \frac{\Psi(x,y)}{\sqrt{q}}\geq  \frac{\Psi(x, \mathcal L(x)^{c_1})}{\mathcal L(x)^{c_0/2}}= \frac{x}{\mathcal L(x)^{1/(2c_1)+c_0/2+o(1)}}\geq \Psi(x, \mathcal L(x)^{2c_0})\geq \Psi(x, q^2), $$
since $1/(2c_1)+c_0/2<1/(4c_0)$. 

Now, suppose that $y\leq \mathcal L(x)^{c_1}$. In this case we have $u=\log x/\log y\gg \sqrt{\log x/\log\log x}$ so that $\log(u)\geq (\log\log x)/3.$ In particular, this implies $
\log q\ll \log y/ (\log u)\ll u$. Moreover $y\geq (2q)^{\log\log x/3A}=(\log x)^{(\log 2q)/3A}$, which is
$\geq (\log x)^{10}$ if  $q$ is large enough (say $q>q_0$). Therefore, by \eqref{Eq:CEPSmooth} we have 
$$ \frac{\Psi(x, y)}{\sqrt{q}}= x\exp\big(-(1+o(1))u\log  u+O(\log q)\big)=x\exp\big(-(1+o(1))u\log  u\big).$$
Let $v=\log x/(2\log q)\gg u\log u$, since $\log y\gg \log u\log q$. Thus,  if $q\geq (\log x)^2$ then another application of \eqref{Eq:CEPSmooth} gives
$ \Psi(x, q^2) = xv^{-(1+o(1))v}\leq x\exp(-c_2u(\log u)^2),$ for some positive constant $c_2$, which implies \eqref{Eq:ThirdTermSmoothSmall} in this case.

If $q_0<q\leq (\log x)^2$, then it follows from \eqref{Eq:SmoothPowersLog} that $\Psi(x,q^2)\leq \Psi\big(x, (\log x)^4\big)= x^{3/4+o(1)}\leq \Psi(x, (\log x)^{10})/\sqrt{q} \leq \Psi(x, y)/\sqrt{q}$, as desired.  If $q\leq q_0$ then $y\geq (\log x)^c$ where $c=(\log 2)/3A$, so
that $\sqrt{q} \Psi(x,q^2)\ll (\log x)^{\pi(q_0^2)} \ll \Psi(x, y)$.
\end{proof}

%%%%%%%%%%%%%%%%%%%%%%%%%%%%%%%%%%%%%%%%%%%%%%%
  
\subsection{Optimality of the bounds in Theorem \ref{Thm:Smooth}}

%%%%%%%%%%%%%%%%%%%%%%%%%%%%%%%%%%%%%%%%%%%%%%%%
\subsection{The major arcs}\label{SubSec:MajorArcsSmooth}
 Here we let $\alpha=a/q$, and choose $\chi$ to be a primitive Dirichlet character modulo $q$. In this case we have 
\begin{equation}\label{Eq:CharMajorArcsSmooth}\sum_{\substack{n\leq x\\ P^+(n)\leq y}} \chi(n) e(n\alpha)= \sum_{\substack{b\bmod q\\ (b, q)=1}} \chi(b) e(ba/q) \Psi(x, y; q, b).  
 \end{equation}
 
 First suppose that $q$ and $y$ satisfy 
 \begin{equation}\label{Eq:RangeSmallqSmoothChar}
 2\leq q\leq (\log x)^A, \  \text{ and } \  (\log x)^{B\log(1+\omega(q))}\leq y\leq x,
\end{equation}
where $B=B(A)$ is a positive large constant.  In particular, we have $y\geq (\log x)^{B\log 2}$.
%In this range with $y<x/q$,  Corollary \ref{Cor:Smooth}   implies that
%\[
%\sum_{\substack{n\leq x\\ P^+(n)\leq y}}  f(n) e(n\alpha)\ll   \frac{2^{\omega(q)/2}}{\sqrt{q}}\Psi(x, y) .
%\]
Therefore, if $B$ is suitably large then  Lemma \ref{Lem:BVHarperSmooth} implies that 
 \begin{equation}\label{Eq:AsymptoticSmoothAPStrong}
 \Psi(x, y;q, b)= \frac{\Psi_q(x, y)}{\phi(q)}+ O\left(\frac{\Psi(x, y)}{q^2}\right), 
 \end{equation}
 for all $b$ such that $(b, q)=1.$ Inserting this estimate in \eqref{Eq:CharMajorArcsSmooth} implies that 
 
 %Hence, we need strong estimates for $\Psi(x, y; b, q)$ in a wide range of the parameters. When $q\leq (\log x)^A$ (where $A$ is a positive constant), Fouvry and Tenenbaum \cite{FoTe90} showed that there exist positive constants $c_1, c_2$ such that uniformly for $x$ and $y$ in the range 
 %$$ \exp(c_1(\log\log x)^2)\leq y\leq x,$$
 %we have 
 %$$ \Psi(x, y; b, q)= \frac{\Psi_q(x, y)}{\phi(q)}\left(1+O\left(\exp\left(-c_2\sqrt{\log y}\right)\right)\right).$$
%Inserting this asymptotic formula in \eqref{Eq:CharMajorArcsSmooth} gives 
\begin{align}\label{Eq:ExponentialCharacterSmooth}
\sum_{\substack{n\leq x\\ P^+(n)\leq y}} \chi(n) e(n\alpha)&=  \frac{\Psi_q(x,y)}{\phi(q)} \sum_{\substack{b\bmod q\\ (b, q)=1}} \chi(b) e(ba/q)+ O\left(\frac{\Psi(x, y)}{q}\right)\nonumber\\
& =\frac{\overline{\chi(a)}\tau(\chi)}{\phi(q)}\Psi_q(x, y)+ O\left(\frac{\Psi(x, y)}{q}\right).
\end{align}
Furthermore, note that $y\geq (u+1)^{\frac{B}{2}\log (1+\omega(q))}$, and hence if $B$ is suitably large we have by \eqref{Eq:TenSmoothCrible}
$$ \Psi_q(x, y)\gg \frac{\phi(q)\Psi(x, y)}{q}.$$
Combining this estimate with \eqref{Eq:ExponentialCharacterSmooth} we get 
\begin{equation}\label{Eq:LargeCharacterMajorSmooth} \sum_{\substack{n\leq x\\ P^+(n)\leq y}} \chi(n) e(n\alpha)\gg \frac{\Psi(x, y)}{\sqrt{q}}, 
\end{equation}
showing, as desired, that  Corollary \ref{Cor:Smooth} is best possible in this range with $y\leq x/q$, up to the factor $2^{\omega(q)/2}$, which we guess should not be there.

Now assume that $q$ and $y$ satisfy 
\begin{align}\label{Eq:RangeLargeqSmoothChar}
 &  (\log x)^A \leq  q\leq \exp\left(c_1\frac{\sqrt{\log x}}{\log\log x}\right) ,\nonumber\\
   \text{ and } \quad \quad    & \max\left(q^{c_2}, (\log x)^{B\log(1+\omega(q))}\right)\leq y\leq \exp\left(c_3\frac{\log x}{\log q(\log\log q)^2}\right),
\end{align}
 Then $q\leq y^{1/c_2}$ and $\log y\leq c_3 \log x /(\log q(\log\log q)^2$, which implies that $q\leq \exp\big((c_3+o(1))u/(\log(u+1)^2\big)$. 
Thus, if $c_2$ is suitably large and $c_3$ is suitably small, it follows from Lemma \ref{Lem:BVHarperSmooth} that \eqref{Eq:AsymptoticSmoothAPStrong} is valid in this case as well. Moreover, since 
$y\geq (\log x)^{B\log(1+\omega(q))}\geq (u+1)^{B\log(1+\omega(q))}$ we see that \eqref{Eq:TenSmoothCrible} holds and hence \eqref{Eq:LargeCharacterMajorSmooth} is also valid in this range of the parameters $y$ and $q$.   %This condition implies that $$ \log\log x \ll \log q\ll \frac{u}{(\log u)^2}

Finally, assuming the Generalized Riemann Hypothesis in the range 
\begin{equation}\label{Eq:RangeGRHqSmoothChar}
  \max\left( q^3(\log q)^{5}, (\log x)^{B\log(1+\omega(q))}\right)\leq y\leq x
\end{equation}
and using \eqref{Eq:SmoothAPUnderGRH}, we derive similarly to \eqref{Eq:ExponentialCharacterSmooth} 
$$ \sum_{\substack{n\leq x\\ P^+(n)\leq y}} \chi(n) e(n\alpha)=\frac{\overline{\chi(a)}\tau(\chi)}{\phi(q)}\Psi_q(x, y)+ O\left(\Psi_q(x, y)\frac{q(\log y)^2}{\sqrt{y}}\right)\gg \frac{\sqrt{q}}{\phi(q)}\Psi_q(x, y),$$
  Thus we again have \eqref{Eq:LargeCharacterMajorSmooth} since $y\geq (\log x)^{B\log(1+\omega(q))}$. %We conclude that $2\leq 

%%%%%%%%%%%%%%%%%%%%%%%%%%%%%%%%%%%%%%%%%%%%%%%%%%%%%%%

\subsection{The minor arcs: Proof of Proposition \ref{Pro:OptimalMinorArcSmooth}}

We shall construct a multiplicative function $f:\mathbb{N}\to \mathbb{U}$ (in fact many such functions) which satisfy \eqref{Eq:ProSmoothLargeMinor}. First, we choose $f$ such that $f(p^k)=0$ for all primes $2x/y<p\leq y/2$ and all positive integers $k$. This gives
\begin{align*} 
\sum_{\substack{n\leq x\\ P^+(n)\leq y}}  f(n) e(n\alpha)
 &= \sum_{\substack{n\leq x\\ 
 \frac y2<P^+(n)\leq y}}  f(n) e(n\alpha)+ \sum_{\substack{n\leq x\\ P^+(n)\leq 2x/y}}  f(n) e(n\alpha)\\
 & =\sum_{ 
 \frac y2< p\leq y} f(p)\sum_{n\leq x/p}  f(n) e(np\alpha)+ \sum_{\substack{n\leq x\\ P^+(n)\leq 2x/y}}  f(n) e(n\alpha).
\end{align*}
Next, to optimize the first sum, we let $f(p)=e^{-i\theta_p}$ for each $y/2<p\leq y$, where $\theta_p$ is the argument of the inner sum $\sum_{n\leq x/p} f(n)e(n\alpha)$. This yields 
$$ \sum_{\substack{n\leq x\\ P^+(n)\leq y}}  f(n) e(n\alpha)
=\sum_{ 
 \frac y2< p\leq y} \bigg|\sum_{n\leq x/p}  f(n) e(np\alpha)\bigg|+ \sum_{\substack{n\leq x\\ P^+(n)\leq 2x/y}}  f(n) e(n\alpha).$$
Finally, on the remaining primes $p
\leq 2x/y$, we choose $f$ to be a Steinhaus random multiplicative function\footnote{This means that $\{f(p)\}_{p\leq 2x/y}$ is a sequence of independent random variables, uniformly distributed on the unit circle, and for all $n>1$ such that $P^+(n)\leq 2x/y$ we have $f(n)=f(p_1)^{a_1}\cdots f(p_k)^{a_k}$, if $n=p_1^{a_1}\cdots p_k^{a_k}$.} satisfying the following events 
\begin{equation}\label{Eq:Event1SmoothMinor}
\mc{A}_1: \bigg|\sum_{\substack{n\leq x\\ P^+(n)\leq 2x/y}}  f(n) e(n\alpha)\bigg|\leq 2\sqrt{x},
\end{equation}
and 
\begin{equation}\label{Eq:Event2SmoothMinor}
   \mc{A}_2: \sum_{ 
 \frac y2< p\leq y} \bigg|\sum_{n\leq x/p}  f(n) e(np\alpha)\bigg|\geq c_0 \frac{\sqrt{xy}}{\log x}, 
\end{equation}
for some suitably small constant $c_0>0$.
We will show that $\pr(\mc{A}_1)\geq 3/4$ and $\pr(\mc{A}_2)\geq 3/4$, so that $\pr(\mc{A}_2\cap \mc{A}_2)\geq 1/2$. This implies that for at least  half of all such multiplicative functions $f$, we have 
$$\bigg|\sum_{\substack{n\leq x\\ P^+(n)\leq y}}  f(n) e(n\alpha)\bigg|\ge \frac{c_0}{2}\frac{\sqrt{xy}}{\log x}, $$
as desired. 

We now establish these claims. We first consider the event $\mc{A}_1$, since it is much easier to handle. We have 
\begin{align*}
\ex\bigg(\bigg|\sum_{\substack{n\leq x\\ P^+(n)\leq 2x/y}}  f(n) e(n\alpha)\bigg|^{2}\bigg)&= \sum_{\substack{n_1, n_2\leq x\\ P^+(n_1n_2)\leq 2x/y}} e\big((n_1-n_2)\alpha\big)\ex\big(f(n_1)\overline{f(n_2)}\big)\\
&= \Psi(x, 2x/y)\leq x.
\end{align*}
 Thus, by Markov's inequality we have 
 $$ \pr(\mc{A}_1^c)\leq \frac{1}{4x}\ex\bigg(\bigg|\sum_{\substack{n\leq x\\ P^+(n)\leq 2x/y}}  f(n) e(n\alpha)\bigg|^{2}\bigg)\leq \frac14, 
 $$
 which implies that $\pr(\mc{A}_1)\geq 3/4.$

We now consider the event $\mc{A}_2$. For a positive integer $k$ we define 
 $$ M_k(f):= \sum_{\frac y2<p\leq y} \bigg| \sum_{m\leq x/p} f(m) e(mp\alpha)\bigg|^k.$$
 By H\"older's inequality we have 
\begin{equation}\label{Eq.HolderSmooth}
M_2(f) \leq M_1(f)^{2/3} M_4(f)^{1/3}.
\end{equation}
Thus, to show that $M_1(f)$ is large,  we need to obtain a lower bound for $M_2(f)$ and an upper bound for $M_4(f)$. To this end, we will prove that   $M_2(f)$ is large uniformly for all $f$, and then bound  $\ex(M_4(f)).$   

We start by handling $M_2(f)$.
Opening up the square gives
$$ M_2(f)= \sum_{\frac y2<p \leq y} \ \sum_{n_1, n_2
\leq x/p} f(n_1)\overline{f(n_2)} e((n_1-n_2) p\alpha).$$
The contribution of the diagonal terms $n_1=n_2$ is 
\begin{equation}\label{Eq.DiagonalM2Smooth}
\sum_{\frac y2<p \leq y} \left\lfloor\frac{x}{p}\right\rfloor\asymp \frac{x}{\log x}.   
\end{equation}
Next, we bound the contribution of the off-diagonal terms, which is  
$$
 \leq 2 \sum_{n_1< n_2<2x/y} \bigg|\sum_{\frac y2<p\leq \min(y, x/n_1)} e((n_2-n_1)p\alpha)\bigg|, 
$$
uniformly in $f$. Putting $h=n_2-n_1$, this becomes 
\begin{equation}
\label{Eq:OffDiagonalMinorSmooth}
 \ll \sum_{n_1<2x/y} \ \sum_{0<h<2x/y}\bigg|\sum_{\frac y2<p\leq \min(y, x/n_1)} e(h p\alpha)\bigg|. 
\end{equation}
To bound the inner exponential sum, we use \eqref{Eq:Vinogradov} (with $\ep/2$ instead of $\ep$), which by partial summation gives
\begin{equation}\label{Eq:Vinogradov2}
\sum_{p\leq N} e\left(\frac{bp}{r}\right)\ll \left(\frac{N}{\sqrt{r}}+\sqrt{Nr\log r} \right)(\log N)^{-1/4+o(1)}+Ne^{-(\frac12-\frac{\ep}{2})\sqrt{\log N}} .
\end{equation}
% \begin{equation}\label{Eq:Vinogradov2}
% \sum_{p\leq N} e\left(\frac{bp}{r}\right)\ll \left(\frac{N}{\sqrt{r}}+\sqrt{Nr\log r} +Ne^{-(\frac12-\frac{\ep}{2})\sqrt{\log N}}\right)(\log N)^{-1/4}\sqrt{\log\log N}, \end{equation}
for all positive integers $ r\geq 2$ and $1\leq b\leq r$ such that $(b, r)=1$.
Write $ha/q= h_1a/q_1$ where $q_1= q/d$ and $h_1= h/d$ with $d=(q,h)$. By \eqref{Eq:Vinogradov2} we get 
\begin{equation}\label{Eq:Vinogradov3}
 \sum_{\frac y2<p\leq \min(y, x/n_1)} e(h p\alpha) \ll \Big(\frac{y}{\sqrt{q/d}}+ \sqrt{yq(\log q)/d}\Big)(\log y)^{-1/4+o(1)}
 +ye^{-(\frac12-\frac{\ep}{2})\sqrt{\log y}}.
 \end{equation} 
% \begin{equation}\label{Eq:Vinogradov3}
%  \sum_{\frac y2<p\leq \min(y, x/n_1)} e(h p\alpha) \ll \Big(\frac{y}{\sqrt{q_1}}+ \sqrt{yq_1\log(2q_1)}+ye^{-(\frac12-\frac{\ep}{2})\sqrt{\log y}}\Big)% (\log y)^{-1/4}\sqrt{\log\log y}.   \end{equation}
For $d\mid q$ we have $|\{1\leq h<2x/y :  (h, q)=d\}|\leq |\{1\leq h<2x/y :  d\mid h\}|\ll x/(yd).$ Thus, by  \eqref{Eq:OffDiagonalMinorSmooth} and \eqref{Eq:Vinogradov3} we deduce that the contribution of the off-diagonal terms to $M_2(f)$ is 
\begin{align*}
&\ll \frac{x}{y} \sum_{d\mid q} \ \sum_{\substack{1\leq h<2x/y\\ (h, q)=d}} \bigg(\bigg(\frac{y}{\sqrt{q/d}}+ \sqrt{yq(\log q)/d}\bigg)(\log x)^{-1/4+o(1)} + ye^{-(\frac12-\frac{\ep}{2})\sqrt{\log y}}\bigg)\\
& \ll \Bigg(\frac{x^2}{y\sqrt{q}}  \sum_{d\mid q} \frac{1}{\sqrt{d}} +  \frac{x^2\sqrt{q\log q}}{y^{3/2}} \sum_{d\mid q}\frac{1}{d^{3/2}} \Bigg) 
(\log x)^{-1/4+o(1)}+ \frac{x^2}{y}e^{-(\frac12-\frac{\ep}{2})\sqrt{\log y}} \sum_{d\mid q} \frac1d \\
&\ll \Bigg(\frac{x^2}{y\sqrt{q}} \prod_{p\mid q} \left(1-\frac{1}{\sqrt{p}}\right)^{-1} +  \frac{x^2\sqrt{q\log q}}{y^{3/2}} \Bigg)(\log x)^{-1/4+o(1)} + \frac{x^2 q}{y\phi(q)}e^{-(\frac12-\frac{\ep}{2})\sqrt{\log y}}\\
&\ll \frac{x^2}{yq^{1/2+o(1)}}   (\log x)^{-1/4} +  \frac{x^2\sqrt{q}}{y^{3/2}}  (\log x)^{O(1)} + \frac{x^2  }{y }e^{-(\frac12-\frac{2\ep}{3})\sqrt{\log y}}
=o\bigg( \frac x{\log x} \bigg)
\end{align*}
as  $\prod_{p\mid q} (1-1/\sqrt{p})^{-1}=q^{o(1)}$, carefully using the ranges for $q$ and $y$ in the hypothesis of Proposition \ref{Pro:OptimalMinorArcSmooth}. 
Comparing this estimate with  \eqref{Eq.DiagonalM2Smooth}, we deduce that
\begin{equation}\label{Eq:LBoundSecondMomentMinorSmooth}
M_2(f) \geq c_1 \frac{x}{\log x},
\end{equation}
uniformly over all multiplicative functions $f$ such that $|f(p)|=1$ for $p\leq 2x/y,$ if $x$ is large enough, where $c_1$ is a positive constant. 

We now find an upper bound for the fourth moment $M_4(f)$, averaged over  Steinhaus random multiplicative functions $f$ (on the primes $p\leq 2x/y$). Thus we consider
\begin{align*}
    & \ex \big(M_4(f)\big)= \sum_{\frac y2<p\leq y} \ex\bigg(\bigg| \sum_{m\leq x/p} f(m) e(mp\alpha)\bigg|^4\bigg)\\
    & \quad =\sum_{\frac y2<p\leq y} \ \sum_{m_1, m_2, n_1, n_2\leq x/p} e\Big(\big((m_1+m_2)-(n_1+n_2)\big)p\alpha \Big) \ex\Big(f(m_1m_2) \overline{f(n_1n_2)}\Big)\\
    & \quad= \sum_{\frac y2<p\leq y} \ \sum_{\substack{m_1, m_2, n_1, n_2\leq x/p\\ m_1m_2=n_1n_2}} e\Big(\big((m_1+m_2)-(n_1+n_2)\big)p\alpha \Big). 
\end{align*}
The contribution of the diagonal terms $m_1+m_2=n_1+n_2$ equals 
\begin{equation}\label{Eq:DiagonalFourthMomentSmooth}2 \sum_{\frac y2<p\leq y} \bigg(\frac{x}{p}+O(1)\bigg)^2 \leq c_2\frac{x^2}{y\log x}, \end{equation}
for some positive constant $c_2$, since $m_1m_2=n_1n_2$ and $m_1+m_2=n_1+n_2$ imply $\{m_1, m_2\}= \{n_1, n_2\}$. On the other hand, the contribution of the off-diagonal terms $m_1+m_2\neq n_1+n_2$ is 
\begin{equation}\label{Eq:OffDiagonalFourthMomentSmooth}
\sum_{\substack{m_1, m_2, n_1, n_2\leq 2x/y\\ m_1m_2=n_1n_2\\ m_1+m_2\neq n_1+n_2}} \ \sum_{\frac y2<p\leq \min\big(y, \frac{x}{m_1}, \frac{x}{m_2}, \frac{x}{n_1}, \frac{x}{n_2}\big)} \  e\Big(\big((m_1+m_2)-(n_1+n_2)\big)p\alpha \Big). 
\end{equation}
Fix $m_1, m_2, n_1, n_2\leq 2x/y$ such that $m_1m_2=n_1n_2$, and let $h= (m_1+m_2)-(n_1+n_2)\neq 0.$ Write $h\alpha= ha/q= h_1a/q_1$, where $q_1= q/(q, h)$ and $h_1= h/(q, h)$.  Therefore, by \eqref{Eq:Vinogradov2} we have 
 \begin{align*}
& \sum_{\frac y2<p\leq \min\big(y, \frac{x}{m_1}, \frac{x}{m_2}, \frac{x}{n_1}, \frac{x}{n_2}\big)} \  e\Big(\big((m_1+m_2)-(n_1+n_2)\big)p\alpha \Big)\\
&\ll \Big(\frac{y}{\sqrt{q_1}}+ \sqrt{yq_1\log(2q_1)}\Big)(\log y)^{-1/4+o(1)}+ye^{-(\frac12-\frac{\ep}{2})\sqrt{\log y}}\\
&\ll \Big(\frac{\sqrt{xy}}{\sqrt{q}}+ \sqrt{yq\log q} \Big)(\log y)^{-1/4+o(1)}+ye^{-(\frac12-\frac{\ep}{2})\sqrt{\log y}}\\
 \end{align*} 
since  $q_1\geq q/|h|\gg qy/x.$ Now Theorem 3 of \cite{HNR15} gives
\[
\sum_{\substack{m_1, m_2, n_1, n_2\leq T\\ m_1m_2=n_1n_2}} 1= \ex\bigg(\bigg|\sum_{n\leq T} f(n)\bigg|^4\bigg)\asymp T^2 \log T,
\]
and taking $T=2x/y$ we deduce, from  \eqref{Eq:OffDiagonalFourthMomentSmooth},  that the contribution of the off-diagonal terms to $\ex(M_4(f))$ is 
\begin{align}\label{Eq:OffDiagonalFourthMomentSmooth2}
 &\ll \frac{x^{5/2}}{y^{3/2}\sqrt{q}}\log\left(\frac{2x}{y}\right)(\log x)^{-1/4+o(1)} +  \frac{x^2\sqrt{q\log q}}{y^{3/2}} (\log x)^{3/4+o(1)} +\frac{x^2}{y}e^{-(\frac12-\frac{2\ep}{3})\sqrt{\log y}} \\
 & \quad = o\bigg( \frac{x^2}{y\log x}\bigg), \nonumber
\end{align} 
using the ranges  for $q$ and $y$ in the hypothesis of Proposition \ref{Pro:OptimalMinorArcSmooth}. 
 Combining \eqref{Eq:DiagonalFourthMomentSmooth} and \eqref{Eq:OffDiagonalFourthMomentSmooth2} we obtain 
\begin{equation}\label{Eq:UBoundFourthMomentSmoothMinor} \ex\big(M_4(f)\big)\leq 2c_2\frac{x^2}{y\log x}.
\end{equation}
Choosing $c_0=\sqrt{c_1^3/(8c_2)}$, and using \eqref{Eq.HolderSmooth} and \eqref{Eq:LBoundSecondMomentMinorSmooth} we deduce that 
$$ \pr(\mc{A}_2^c)\leq \pr\left(M_4(f)\geq \frac{c_1^3}{c_0^2}\frac{x^2}{y\log x}\right)\leq \frac{c_0^2y \log x}{c_1^3x^2} \ex(M_4(f))\leq \frac{1}{4}, $$
by Markov's inequality and \eqref{Eq:UBoundFourthMomentSmoothMinor}. This completes the proof.

%%%%%%%%%%%%%%%%%%%%%%%%%%%%%%%%%%%%%%%%%%%%%%%%%%%%%%%%%%%%%%%%%%%%%%%%%%%%%%%%%%%%%%%%%%%%%%%%%%%%%%%%%%%%%%%%%%%%%%%%%%%%%%%%%%%%%%%%%%%%%%%%%%%%%%%%%%%%%%%%%%%%%%%%%%%%%%%%%%%%%%%%%%%%%%

\section{Large unweighted exponential sums in the minor arcs}

In this section we prove Theorem \ref{Thm:Unweighted} and Corollary \ref{Cor:Unweighted}, which show that if the exponential sum \eqref{Eq:UnweightedSum} is large in the minor arcs, then the main contribution comes from terms involving the  primes $p\asymp x$. 
\begin{proof}[Proof of Theorem \ref{Thm:Unweighted}]Let $y=x/Q$
where $Q:=\min \{ q,(x/q)^{1/2}\}\leq x^{1/3}$. We write 
$$ 
\sum_{n\leq x}f(n) e(n\alpha)= \sum_{\substack{n\leq x\\ P^+(n)\leq y}}  f(n) e(n\alpha)+\sum_{\substack{n\leq x\\ P^+(n)> y}} f(n) e(n\alpha).
$$
Recall that $\mathcal{L}(x)= \exp(\sqrt{\log x\log\log x})$. If $(\log x)^{2+\ep}\leq q\leq \mathcal{L}(x)$, then $y=x/q$ and \eqref{Eq:ySmallTheoremSmooth} implies that 
 \[
\sum_{\substack{n\leq x \\ P^+(n)\leq y}}  f(n) e(\alpha n) \ll 
   \frac{2^{\omega(q)/2}}{\sqrt{q}} x+ \Psi\big(x, \mathcal{L}(x)^2\big)    \ll\frac{x}{\sqrt{M}\log x},
\]
where the last estimate follows from \eqref{Eq:ChangeRegimSmoothCount} and our assumption on $M$.
On the other hand, if $\mathcal{L}(x)\leq q\leq x/(\log x)^{3+\ep}$ we use Proposition \ref{Pro:ImprovingRegis} which gives 
\begin{align*}
\sum_{\substack{n\leq x \\ P^+(n)\leq y}}  f(n) e(\alpha n) &\ll 
\frac{x}{\sqrt{Q\log x}}+\frac{x\sqrt{\log x} }{\sqrt q} + \sqrt{xq\log(2x/q)\log x}+ \frac x {\mathcal{L}(x)^{1/2+o(1)}} \\
& \ll \frac{x}{\sqrt{M}\log x},
\end{align*}
since
$M \leq  \min\Big( \mathcal{L}(x)^{1-\ep}, Q/\log x,  q/(\log x)^3,  x/\big(q(\log x)^3\log(2x/q)\big)\Big) $ which follows from the hypotheses.

Thus, in all cases we have
 $$
 \sum_{n\leq x}f(n) e(n\alpha)=  \sum_{y< p\leq x}f(p) \sum_{m\leq x/p} f(m) e(m p \alpha)+ O\bigg(\frac{x}{\sqrt{M}\log x} \bigg).
 $$
 Therefore, to complete the proof it suffices to show that 
\begin{equation}\label{Eq:TheSumLargePrimes}
 \sum_{y< p\leq x/M}f(p) \sum_{m\leq x/p} f(m) e(m p \alpha)\ll \begin{cases} \displaystyle{\frac{x}{\sqrt{M}\log x}} & \text{ if } q\leq x^{1-\ep},\\
 \\
\displaystyle{\frac{x\sqrt{\log M}}{\sqrt{M}\log x}} & \text{ if } x^{1-\ep} \leq q \leq x/(\log x)^{3+\ep}. \end{cases} 
\end{equation}
Since $|\alpha-a/q|\leq 1/(qx)$, replacing $\alpha$ by $a/q$ introduces an error term of size $O(x/q)$ which is acceptable. Thus for the remaining part of the proof we take $\alpha=a/q$.  

{\bf Case 1: $q\leq x^{1-\ep}$}. Divide the first sum over the primes $y=x/Q<p\leq x/M$ into short intervals $I_j=(x/t_{j+1},x/t_{j}]$,
where $t_0=M$, $t_J\sim Q$, and $t_{j+1}=t_j+ t_j^{1/2}$ for each $j$.
Now the absolute value of our sum is 
\begin{align*}
& \leq \sum_{j=0}^J \ \sum_{x/t_{j+1}\leq p\leq x/t_j}\bigg|\sum_{m\leq x/p} f(m) e(m p \alpha)\bigg|\\
&  = \sum_{j=0}^J \ \sum_{x/t_{j+1}\leq p\leq x/t_j}\bigg|\sum_{m\leq t_j} f(m) e(m p \alpha)\bigg| +O(E),
\end{align*}
where 
\begin{align*}
E&\ll\sum_{j=0}^J (t_{j+1}-t_j) \big(\pi(x/t_{j})-\pi(x/t_{j+1})\big)\\
&\ll\frac{x}{\log x} \sum_{j=0}^J \frac{t_{j+1}-t_j}{t_j^{3/2}}\ll \frac{x}{\log x}\int_M^{Q}\frac{dt}{t^{3/2}}\ll \frac{x}{\sqrt{M}\log x},
\end{align*}
by the Brun-Titchmarsh Theorem. Furthermore, we have 
\begin{align}\label{Eq:breackSumAP}
\sum_{x/t_{j+1}\leq p\leq x/t_j}\bigg|\sum_{m\leq t_j} f(m) e(m p \alpha)\bigg|
&= \sum_{\substack{1\leq r\leq q-1\\ (r,q)=1}}\ \bigg|\sum_{m\leq t_j} f(m) e(m r \alpha)\bigg|\sum_{\substack{x/t_{j+1}\leq p\leq x/t_j\\ p\equiv r\bmod q}} 1\nonumber\\
& \ll \frac{x(t_{j+1}-t_j)}{t_j^2\log x} \frac{1}{\phi(q)}\sum_{\substack{1\leq r\leq q-1\\ (r,q)=1}}\ \bigg|\sum_{m\leq t_j} f(m) e(m r \alpha)\bigg|,
\end{align}
by the Brun-Titchmarsh Theorem, since $x/t_j-x/t_{j+1}\asymp x/t_j^{3/2}\geq x/Q^{3/2}\geq x^{1/4}q^{3/4}\geq   q^{1+\ep/4}$ by our assumption on $q$ and $Q$. Now by the Cauchy-Schwarz inequality we have similarly to \eqref{Eq:MainTermCSSmooth} 
\begin{align}\label{Eq:CS_Ramanujan}
\Bigg(\sum_{\substack{1\leq r\leq q-1\\ (r,q)=1}} \bigg|\sum_{m\leq t_j} f(m) e(m r \alpha)\bigg|\Bigg)^2
&\leq \phi(q)\sum_{\substack{1\leq r\leq q-1\\ (r,q)=1}} \bigg|\sum_{m\leq t_j} f(m) e(m r \alpha)\bigg|^2\nonumber \\
&= \phi(q)\sum_{m_1, m_2\leq t_j} f(m_1)\overline{f(m_2)}c_q(m_1-m_2)\\
& \ll \phi(q)t_j\sum_{0\leq h\leq t_j} |c_q(h)|.
\end{align}
Since
$ t_j\leq Q\leq q$ for all $0\leq j\leq J$, it follows from \eqref{Eq:BoundShortSumRamanujan} that    
  
$$
\sum_{\substack{1\leq r\leq q-1\\ (r,q)=1}} \bigg|\sum_{m\leq t_j} f(m) e(m r \alpha)\bigg|\ll \phi(q) \sqrt{t_j}+\frac{\phi(q)^{3/4}}{q^{1/4}}2^{\omega(q)/2}t_j.$$
Using this bound in \eqref{Eq:breackSumAP} gives 
\begin{align*}
&\sum_{j=0}^J \ \sum_{x/t_{j+1}\leq p\leq x/t_j}\bigg|\sum_{m\leq t_j} f(m) e(m p \alpha)\bigg|\\
&\ll \frac{x}{\log x}\sum_{j=0}^J \frac{t_{j+1}-t_j}{t_j^{3/2}}+ \frac{2^{\omega(q)/2}}{(q\phi(q))^{1/4}}\frac{x}{\log x}\sum_{j=0}^J \frac{t_{j+1}-t_j}{t_j}\\
&\ll \frac{x}{\sqrt{M}\log x}+
\frac{2^{\omega(q)/2}}{(q\phi(q))^{1/4}} x\ll \frac{x}{\sqrt{M}\log x},
\end{align*}
since 
$$ \sum_{j=0}^J \frac{t_{j+1}-t_j}{t_j^{3/2}}\ll\int_M^{Q}\frac{dt}{t^{3/2}} \ll \frac{1}{\sqrt{M}} \ \ \text{ and } \ \ \sum_{j=0}^J \frac{t_{j+1}-t_j}{t_j}\ll\int_M^{Q}\frac{dt}{t} \leq \log x,$$
and the result follows in this case.

\medskip

{\bf Case 2: $x^{1-\ep}\leq q\leq x/(\log x)^{3+\ep}$}. In this case we split the first sum over the primes $y=x/Q<p\leq x/M$ in \eqref{Eq:TheSumLargePrimes} into dyadic intervals $x/L<p\leq 2x/L$ where $M\lesssim L\lesssim Q$. Now we have 
\begin{equation}\label{Eq:DyadicLargeq}
\begin{aligned}
\sum_{\frac{x}{L}<p\leq \frac{2x}{L}} f(p) \sum_{m\leq x/p} f(m) e(mp\alpha)
&= \sum_{m\leq L} f(m) \sum_{\frac{x}{L}<p \leq \min(\frac{2x}{L}, \frac{x}{m})} f(p) e(mp\alpha)\\
& \ll L^{1/2} \left(\sum_{m\leq L} \Big|\sum_{\frac{x}{L}<p \leq \min(\frac{2x}{L}, \frac{x}{m})} f(p) e(mp\alpha)\Big|^2\right)^{1/2},
\end{aligned}
\end{equation}
by the Cauchy-Schwarz inequality. Opening up the inner sum and changing the order of summation shows that it equals
\begin{equation}\label{Eq:DyadicLargeq2}
\begin{aligned}
&\sum_{\frac{x}{L}<p_1, p_2 \leq \frac{2x}{L}} f(p_1)\overline{f(p_2)} \sum_{m\leq \min\big(L, \frac{x}{p_1}, \frac{x}{p_2}\big)} e((p_2-p_1)\alpha m)\\
&\ll \frac{x}{\log x}+\sum_{\frac{x}{L}<p_1< p_2 \leq \frac{2x}{L}} \min\left(L, \frac{1}{||(p_2-p_1)\alpha||}\right)\\
&\leq  \frac{x}{\log x}+\sum_{0< h\leq x/L}\min\left(L, \frac{1}{||h\alpha||}\right) \#\{p_1, p_2\leq 2x/L : p_2=h+p_1\}\\
& \ll \frac{x}{\log x}+\frac{x}{L(\log x)^2} \sum_{0< h\leq x/L}\min\left(L, \frac{1}{||h\alpha||}\right)\frac{h}{\phi(h)},
\end{aligned}
\end{equation}
by a standard sieve estimate (see for example \cite[Theorem 3.11]{HR74}). To bound the inner sum over $h$ we use that 
$$\frac{h}{\phi(h)}\ll \sum_{m\mid h} \frac{\mu^2(m)}{m}\ll \sum_{\substack{m\mid h\\ m<\sqrt{h}}} \frac{\mu^2(m)}{m}, $$ 
since 
$$ 
\sum_{\substack{m\mid h\\ m>\sqrt{h}}} \frac{\mu^2(m)}{m}\leq \frac{1}{\sqrt{h}}\sum_{m\mid h} \mu^2(m)= \frac{2^{\omega(h)}}{\sqrt{h}}= h^{-1/2+o(1)}.
$$ 
Therefore, we deduce that\begin{equation}\label{Eq:BoundSumIntegerDistance}
\sum_{0< h\leq x/L}\min\left(L, \frac{1}{||h\alpha||}\right)\frac{h}{\phi(h)} \ll \sum_{m\leq (x/L)^{1/2}}\frac{\mu^2(m)}{m}  \sum_{0< n\leq x/(mL)}\min\left(L, \frac{1}{||mn\alpha||}\right).
\end{equation}
Note that $m\alpha= (m/(m,q))a/(q/(m,q)$. Hence, it follows from Lemma \ref{Lem:StandardExpSum} that
\begin{align*}
    \sum_{0< n\leq x/(mL)}\min\left(L, \frac{1}{||mn\alpha||}\right)
    &\ll
    \frac{x(m,q)}{m q}+L+\frac{x}{mL}\log L+ \frac{q}{(m, q)}\log\left(2+\frac{x(m, q)}{mq}\right)\\
    & \ll \frac{x}{mL}\log L+ \frac{q}{(m, q)}\log\left(2+\frac{x(m, q)}{mq}\right)
\end{align*}
since $L\ll x^{\ep}$, $q\geq x^{1-\ep}$ and $m\leq \sqrt{x}$. Combining this bound with \eqref{Eq:BoundSumIntegerDistance} gives
\begin{align*}
& \sum_{0< h\leq \frac{x}{L}}\min\left(L, \frac{1}{||h\alpha||}\right)\frac{h}{\phi(h)}\\
& \ll \sum_{m\leq \sqrt{\frac{x}{L}}}\frac{\mu^2(m)}{m}  \left(\frac{x}{mL}\log L+ \frac{q}{(m, q)}\log\left(2+\frac{x(m, q)}{mq}\right)\right)\\
& \ll \frac{x}{L}\log L+ q\log(x/q)\log x.
\end{align*}
Inserting this estimate in \eqref{Eq:DyadicLargeq2} and then back into \eqref{Eq:DyadicLargeq} we obtain 
$$
\sum_{x/L<p\leq 2x/L} f(p) \sum_{m\leq x/p} f(m) e(mp\alpha)
\ll \frac{x}{\log x} \frac{(\log L)^{1/2}}{\sqrt{L}}+ \frac{(xq\log(x/q))^{1/2}}{\sqrt{\log x}}.
$$
Summing over $M\lesssim L \lesssim Q= (x/q)^{1/2}$ in dyadic ranges implies that 
\begin{align*}
\sum_{y<p\leq x/M} f(p) \sum_{m\leq x/p} f(m) e(mp\alpha) &\ll \frac{x}{\sqrt{M}\log x}(\log M)^{1/2}+ \frac{\sqrt{xq}}{\sqrt{\log x}} \log (x/q)^{3/2}\\
&\ll \frac{x}{\sqrt{M}\log x}(\log M)^{1/2},
\end{align*}
as desired.

\end{proof}

\begin{proof}[Proof of Corollary \ref{Cor:Unweighted}]
 Let $M= A(1/c)^2\log(1/c)$ where $A$ is a suitably large constant. By our assumption and Theorem \ref{Thm:Unweighted} it follows that 
$$ \Big|\sum_{m\leq M}f(m) \sum_{x/M< p\leq x/m} f(p) e(m p \alpha) \big|\geq  \frac{cx}{2\log x}.$$
This shows that there exists $h\leq M$, such that 
$$\Big|\sum_{x/M< p\leq x/h} f(p) e(h p \alpha) \big|\gg  \frac{c^3}{\log(1/c)}\frac{x}{\log x}.$$
Splitting the interval $(x/M, x/h]$ into dyadic intervals $(z, 2z]$, we deduce that there exists $z\in  [x/M, x/h]$ such that 
$$\Big|\sum_{z< p\leq 2z} f(p) e(h p \alpha) \big|\gg  \frac{c^3}{\log(1/c)}\frac{z}{\log z}, $$
as desired.
\end{proof}

%%%%%%%%%%%%%%%%%%%%%%%%%%%%%%%%%%%%%%%%%%%%%%%%%%%%%%%%%%%%%%%%%%%%%%%%%%%%%%%%%%%%%%%%%%%%%%%%%%%%%%%%%%%%%%%%%%%%%%%%%%%%%%%%%%%%%%%%%%%%%%%%%%%%%%%%%%%%%%%%%%%%%%%%%%%%%%%%%%%%%%%%%%%%%%%%%%%%%%%%%%%%%%%%%%%%%%%%%%%%%%%%%%%%%%%%%%%%%%%%%%%%%%%%%%%%

 \section{Exponential sums in the logarithmic case: The minor arcs}

\subsection{Simplifying large sums in the logarithmic case}

 In this section we prove the following result which will imply Theorem \ref{cor: Cut sum2}.
 
\begin{theorem}\label{Thm:Main}
    Let  $ \alpha\in [0,1)$ and fix $\delta>0$. Let $x$ be large and suppose that  $|\alpha-a/q|\leq 1/(qx)$ with $(a, q)=1$. Let  $ y=q^{3/2+\delta}$.   Then
uniformly for all multiplicative functions $f:\mathbb{N}\to \mathbb{U}$ we have 
\begin{equation}\label{Eq:StatementThm1.1}
 \begin{aligned}
 \sum_{ n\leq x} \frac{f(n)}{n} e(n\alpha)&=  
 \sum_{n\leq y}  \frac{f(n)}{n} e(n\alpha) + \sum_{y< p\leq x} \frac{f(p)}{p}\sum_{m<\min\{y, x/p\}} \frac{f(m)}{m} e(mp\alpha)\\
 & \quad \quad
+O\left(1+\frac{2^{\omega(q)}\log q }{q^{1/2}} \cdot  \log x  \right).
\end{aligned}
\end{equation}
\end{theorem}

We will deduce Theorem \ref{Thm:Main} from the following key lemma, which shows that the contribution to the exponential sum $\sum_{n\leq x} f(n)e(n\alpha)/n$ of the terms $n=mp$ where $p$ is a large prime and $m>1$ is either $p$-rough or large and $y$-smooth,  is negligible.  Here and throughout we let  $P^{-}(n)$ denotes the smallest prime factor of $n$, with the standard convention $P^{-}(1)=\infty$.
\begin{lemma} \label{lem: Key} Let $x$ be large and fix $\delta>0$. Let $1\leq a\leq q\leq x$ be positive coprime integers.  
Let

\begin{itemize}
    \item[(i)]  $\mathcal M_p=\{ m>1:   P^-(m)>p\}$ and  $B=\sqrt{x}$, or 
\item[(ii)] $\mathcal M_p=\{ m\geq y: P^+(m)\leq y\}$ and $B=x/y$.  
\end{itemize}
If   $y\geq q^{3/2+\delta}$ then for all multiplicative functions $f:\mathbb{N}\to \mathbb{U}$ we have
\[
\sum_{\substack{n=pm\leq x \\ p \text{ prime} \\ y<p\leq B\\ m\in \mathcal M_p}}  \frac{f(n)}{n}e\bigg(\frac{na}{q}\bigg) \ll \frac{2^{\omega(q)/2}}{\phi(q)^{1/2}} \log x.
\]
\end{lemma}

%%%%%%%%%%%%%%%%%%%%%%%%%%%%%%%%%
\begin{proof} Note that the result is trivial if $y>\sqrt{x}$ (since the sum over primes is empty in both cases (i) and (ii)) so we might assume throughout the proof that $y\leq \sqrt{x}$. We  first divide the sum over the primes $t_0=y<p\leq  B\lesssim t_J$ into short intervals $I_j=(t_j,t_{j+1}]$,
taking $t_{j+1}=t_j(1+ 1/L(t_j))$ for each $j$, where $L:\mathbb{R}\to\mathbb{R}$ is a continuous positive increasing function to be chosen, and such that for all $u\geq y$ we have
\begin{equation}\label{Eq:ChoiceL}
1<L(u)<u^{1-\ep}/q,   
\end{equation}
 %for some $L(u)$  a positive increasing function for which
for some suitably small $\ep>0$, which depends only on $\delta$. This choice implies that $\log t_j>\log (t_{j+1}- t_j) >\log ( \frac{t_{j+1}- t_j}q) > \log t_j^\ep=\ep \log t_j$. 
Therefore, the absolute value of our exponential sum is
\begin{equation}\label{Eq:BoundS1S3}
S:=\Bigg| \sum_{\substack{n=pm\leq x \\ p \text{ prime} \\ y<p\leq B\\ m\in \mathcal M_p}}  \frac{f(p)f(m)}{mp}e\left(\frac{mpa}{q}\right)\Bigg|
\leq \sum_{0\leq j< J} \sum_{t_j<p\leq t_{j+1}}\frac{1}{p} \Bigg|\sum_{\substack{m\leq x/p\\ m\in \mathcal M_p}}\frac{f(m)}{m} e\left(\frac{mpa}{q}\right)\Bigg| 
 \end{equation}
and we now split the primes into different residue classes modulo $q$. First though we need to remove the dependence on $p$ in the index of the inner summation on $m$, noting that
\begin{equation}\label{Eq:Trimming}
\Bigg|\sum_{\substack{m\leq x/p\\ m\in \mathcal M_p}}\frac{f(m)}{m} e\left(\frac{mpa}{q}\right)
- \sum_{\substack{m\leq x/t_j\\ m\in \mathcal M_j}}\frac{f(m)}{m} e\left(\frac{mpa}{q}\right)\Bigg|
\leq \sum_{\substack{x/t_{j+1}\leq m\leq x/t_j\\ m\in \mathcal M_p}}\frac{1}{m}+ 
\sum_{\substack{m\leq x/t_j\\ m\in \mathcal M_j\setminus \mathcal M_p}}\frac{1}{m},
 \end{equation}
where $\mathcal M_j=\mathcal M_p$ for the smallest prime $p$ in $I_j$. 

(i)\ For the first sum we are sieving an interval and so this is
\begin{equation}
\label{Eq:FirstError_i}
\leq \frac{t_j}{x}\sum_{\substack{x/t_{j+1}\leq m\leq x/t_j\\ P^{-}(m)>t_j}}1
\ll    \frac{1}{\log t_j} (1-t_j/ t_{j+1}) \ll \frac{1}{L(t_j)\log t_j}.
\end{equation}
In the second sum let $p_1= P^-(m)$ and write $m=p_1^kn$ so that our sum is
\begin{align*} \leq\sum_{\substack{t_j<p_1\leq t_{j+1} \\ p_1 \text{ is a prime}}} \sum_{\substack{k\geq 1\\ p_1^k\leq x/t_j}}\frac{1}{p_1^k}\sum_{\substack{n\leq x/(p_1^kt_j) \\ P^-(n)> p_1}} \frac{1}{n} 
&\leq \sum_{\substack{t_j<p_1\leq t_{j+1} \\ p_1 \text{ is a prime}}} \sum_{\substack{k\geq 1\\ p_1^k\leq x/t_j}}\frac{1}{p_1^k}\prod_{\substack{p_1< p_2 \leq x\\ p_2 \text{ is a prime}}}\left(1-\frac{1}{p_2}\right)^{-1} \\ 
& \ll \log x\sum_{\substack{t_j<p_1\leq t_{j+1} \\ p_1 \text{ is a prime}}} \frac{1}{p_1\log p_1}
 \leq \log x \frac{\pi( t_{j+1})-\pi(t_j)}{t_j\log t_j} \\
 & \ll \frac{ t_{j+1}- t_j}{t_j\log t_j \log(t_{j+1}- t_j)} \log  x \ll  \frac{   \log x}{L(t_j)(\log t_j)^2} 
\end{align*}
by Mertens' Theorem and then the Brun-Titchmarsh theorem. This error bound is always larger than the bound in \eqref{Eq:FirstError_i}.
 Therefore the contribution of these error terms to the right-hand side of \eqref{Eq:BoundS1S3} is 
\[ \ll \log x\sum_{y<p\leq x} \frac{1}{pL(p)(\log p)^2}
  \leq \frac{\log x}{L(y)}\sum_{y<p\leq x} \frac{1}{p(\log p)^2}\ll \frac{\log x}{L(y)(\log y)^2}.
\]
 
 (ii)\ In this case we only have the first term from the right-hand side of \eqref{Eq:Trimming}. In all we sum over $m\in[y,x/y]$
 and each $m\in [x/t_{j+1}, x/t_j)$ has the weight
 \[
 \frac 1m \sum_{t_j<p\leq t_{j+1}}\frac{1}{p} \leq  \frac 1m \frac{\pi( t_{j+1})-\pi(t_j)}{t_j }\ll \frac 1m \frac 1 {L(t_j)\log t_j} 
 \leq \frac 1m \frac 1{L(y)\log y} ,
 \]
 by the Brun-Titchmarsh Theorem, so in total the contribution of such terms to the right hand side of \eqref{Eq:BoundS1S3} is 
  \[
 \ll  \sum_{\substack{y\leq m\leq x/y\\ P^+(m)\leq y}}\frac{1}{m} \frac 1{L(y)\log y} \ll  \frac 1{L(y)}.
 \]
 
 We deduce that, in both cases,  
\begin{align*} 
 S &\leq \sum_{0\leq j< J} \sum_{\substack{t_j<p\leq t_{j+1} }}\frac{1}{p} \Bigg|\sum_{\substack{1<m\leq x/t_j\\ m\in \mathcal M_j}}\frac{f(m)}{m} e\left(\frac{mpa}{q}\right)\Bigg|+O\left(\frac{\log x}{L(y)\log y}\right) \\
&= \sum_{0\leq j< J} \sum_{\substack{1\leq r\leq q-1\\ (r,q)=1}}\Bigg|\sum_{\substack{1<m\leq x/t_j\\ m\in \mathcal M_j}}\frac{f(m)}{m} e\left(\frac{mra}{q}\right)\Bigg|\sum_{\substack{t_j<p\leq t_{j+1}\\ p\equiv r\bmod q}}\frac{1}{p} +O\left(\frac{\log x}{L(y)\log y}\right)
\end{align*} 
splitting the primes $p$ into residue classes modulo $q$.
By the Brun-Titchmarsh Theorem 
$$ 
\sum_{\substack{t_j<p\leq t_{j+1}\\ p\equiv r\bmod q}}\frac{1}{p} \leq \frac{1}{t_j} \# \{  t_j<p\leq t_{j+1} : p \equiv r\bmod q\}
$$
$$
\ll  \frac{t_{j+1}/t_{j}-1}{\phi(q)\log\left( \frac{t_{j}}{qL(t_j)}\right)}\asymp \frac{t_{j+1}/t_{j}-1}{\phi(q) \log t_j} \ll 
\frac{1}{\phi(q)L(t_j) \log t_j}
$$
since   $q<t_j^{1-\ep}/L(t_j)$.
Inserting this estimate above and   changing variables $b=ar$ we obtain  the bound
\begin{equation}\label{EqBoundS4}
S  \ll   \sum_{0\leq j< J} \frac{1}{L(t_j) \log t_j} \cdot  \frac{1}{\phi(q)}\sum_{\substack{1\leq b\leq q-1\\ (b,q)=1}}\Bigg|\sum_{\substack{1<m\leq x/t_j\\ m\in \mathcal M_j}}\frac{f(m)}{m} e\left(\frac{mb}{q}\right)\Bigg| +\frac{\log x}{L(y)\log y}. 
\end{equation}
Using the Cauchy-Schwarz inequality, we get
\begin{equation}\label{Eq:NiceCorners}
\frac{1}{\phi(q)} \sum_{\substack{1\leq b\leq q-1\\ (b,q)=1}}\Bigg|\sum_{\substack{1<m\leq x/t_j\\ m\in \mathcal M_j}}\frac{f(m)}{m} e\left(\frac{mb}{q}\right)\Bigg|\leq \Bigg(\frac{1}{\phi(q)} \sum_{\substack{1\leq b\leq q-1\\ (b,q)=1}}\Bigg|\sum_{\substack{1<m\leq x/t_j\\ m\in \mathcal M_j}}\frac{f(m)}{m} e\left(\frac{mb}{q}\right)\Bigg|^2\Bigg)^{1/2}.
\end{equation}
Expanding the square and changing the order of summation gives
\begin{align}
&\frac{1}{\phi(q)} \sum_{\substack{1\leq b\leq q-1\\ (b,q)=1}}\Bigg|\sum_{\substack{1<m\leq x/t_j\\ m\in \mathcal M_j}}\frac{f(m)}{m} e\left(\frac{mb}{q}\right)\Bigg|^2 \nonumber \\
&= \sum_{\substack{1<m, n\leq x/t_j\\ m,n\in \mathcal M_j }}  \frac{f(m)\overline{f(n)}}{mn}\cdot \frac{1}{\phi(q)} \sum_{\substack{1\leq b\leq q-1\\ (b,q)=1}} e\left(\frac{(m-n)b}{q}\right) \notag \\
& \leq \frac{1}{\phi(q)} \sum_{\substack{1<m, n\leq x/t_j\\ m,n\in \mathcal M_j }} \frac{|c_q(m-n)|}{mn} ,\label{Eq:SecondMoment}
\end{align} 
 where $c_q(\ell)$ is defined in \eqref{Eq:DefRamanujan}.
The contribution of the diagonal terms $m=n$ to the right hand side of \eqref{Eq:SecondMoment} is 
\begin{equation}\label{Eq:Contrib_OffDiagonal}
 \ll   \sum_{\substack{1<m\leq x/t_j\\ m\in \mathcal M_j}} \frac{1}{m^2}. 
\end{equation}
 
(i)\ In this case, this sum is 
\[ 
\leq \prod_{\substack{p_1\geq t_j\\ p_1 \text{ is prime}}} \left(1+\frac{1}{p_1^2-1}\right)-1\leq \exp \bigg(\sum_{\substack{p_1\geq t_j\\\ p_1\textup{ is prime }}} \frac{1}{p_1^2-1}\bigg)-1 \ll \frac{1}{t_j\log t_j} \leq \frac{1}{y \log t_j} \
\]
since  $y\leq t_j$.

ii) \ In this case we have 
$$\sum_{\substack{1<m\leq x/t_j\\ m\in \mathcal M_j}} \frac{1}{m^2}\leq \sum_{m\geq y} \frac{1}{m^2}\ll \frac{1}{y},$$
so that in both cases, \eqref{Eq:Contrib_OffDiagonal} is  $\ll 1/y$.

Next, the absolute value of  the off-diagonal terms in the right hand side of \eqref{Eq:SecondMoment} is 
\[ \leq  \frac{1}{\phi(q)} \sum_{d\mid q}d\left|\mu\left(\frac{q}{d}\right)\right|\sum_{\substack{1<m\ne n\leq x/t_j\\ m,n\in \mathcal M_j \\ m\equiv n\bmod d}}  \frac{1}{mn}  \ll  \frac{2^{\omega(q)}}{\phi(q)}   \log x    \sum_{\substack{ 1<n\leq x/t_j\\
 n\in \mathcal M_j  }}  \frac{1}{n} 
\]
where we replace the pair $m,n$ by $\max\{m,n\}, \min\{m,n\}$ and since
\[
\sum_{\substack{n<m\leq x/t_j\\ m\in \mathcal M_j  \\ m\equiv n\bmod d}}  \frac{1}{m}\leq \sum_{\substack{ d<m\leq x/t_j\\m\equiv n\bmod d}} \frac{1}{m}\ll \frac{\log x}{d} .
\]

(i)\ In this case we have
$$  \sum_{\substack{ 1<n\leq x/t_j\\  n\in \mathcal M_j  }}  \frac{1}{n} \leq \prod_{\substack{t_j< p_1\leq x\\  p_1 \text{ is prime }}}\left(1-\frac{1}{p_1}\right)^{-1}\ll\frac{\log x}{\log t_j},
$$
leading to the error bound $\ll  \frac{2^{\omega(q)}}{\phi(q)}   \frac{(\log x)^2}{\log t_j} $ for the off-diagonal terms.

(ii)\ In this case we have
$$  \sum_{\substack{ 1<n\leq x/t_j\\  n\in \mathcal M_j  }}  \frac{1}{n} \leq \prod_{\substack{  p_1\leq y\\  p_1 \text{ is prime }}}\left(1-\frac{1}{p_1}\right)^{-1}\ll \log y
$$
leading to the error bound $\ll  \frac{2^{\omega(q)}}{\phi(q)}   (\log x)(\log y) $ for the off-diagonal terms. In both cases, these bounds exceed $1/y$.

 We now insert these estimates into \eqref{Eq:SecondMoment} and then into \eqref{Eq:NiceCorners}, and from there into our bound for $S$ in \eqref{EqBoundS4}.
% $$  \frac{1}{\phi(q)} \sum_{\substack{1\leq b\leq q-1\\ (b,q)=1}}\Bigg|\sum_{\substack{1<m\leq x/t_j\\ m \in \mathcal M_j  } \frac{f(m)}{m} e\left(\frac{mb}{q}\right)\Bigg|\ll \text{(i) } \frac{2^{\omega(q)/2}}{\phi(q)^{1/2} } \frac{\log x}{\sqrt{\log t_j}} \text{ or (ii) } \frac{2^{\omega(q)/2}}{\phi(q)^{1/2}} \sqrt{\log x \log y} $$
 
 (i) This leads to the bound
 \begin{align}\label{Eq:BoundS_i}
  S   &\ll \frac{2^{\omega(q)/2}\log x}{\phi(q)^{1/2}}\sum_{0\leq j< J} \frac{1}{L(t_j)(\log t_j)^{3/2}}  +\frac{\log x}{L(y)\log y} \nonumber\\
&\ll \frac{2^{\omega(q)/2}}{\phi(q)^{1/2}}\frac{\log x}{\sqrt{\log y}} +\frac{\log x}{L(y)\log y},
\end{align}
since 
\begin{equation}\label{Eq:CalculationIntegralE} \sum_{0\leq j< J} \frac{1}{L(t_j)(\log t_j)^{3/2}}  = \sum_{0\leq j< J} \frac{t_{j+1}-t_j}{t_j(\log t_j)^{3/2}}   \ll \int_y^{x} \frac{1}{t (\log t)^{3/2}} dt \ll \frac{1}{\sqrt{\log y }}.
\end{equation}
We now choose $L(u)=u^{1/3-\ep}$, which satisfies assumption \eqref{Eq:ChoiceL} since $y\geq q^{3/2+\delta}$. On the other hand, taking $\ep>0$ small enough insures that
$L(y)  >q^{1/2}$. Thus the second error term in \eqref{Eq:BoundS_i} is smaller than the first.
% and so
% \[ S \ll \frac{2^{\omega(q)/2}}{\phi(q)^{1/2}}\frac{\log x}{\sqrt{\log q}} . \]

 (ii) In this case we obtain the bound
 \begin{align*}
  S &\ll \frac{(2^{\omega(q) }\log x\log y)^{1/2}}{\phi(q)^{1/2}}\sum_{0\leq j< J} \frac{1}{L(t_j)(\log t_j) }  +\frac{\log x}{L(y)\log y}\\
&\ll  \frac{(2^{\omega(q) }\log x\log y)^{1/2}}{\phi(q)^{1/2}} \log\bigg(\frac{\log x}{\log y} \bigg)  +\frac{\log x}{L(y)\log y}
\ll \frac{2^{\omega(q)/2}}{\phi(q)^{1/2}} \log x,
\end{align*}
where the second inequality follows by a similar calculation to \eqref{Eq:CalculationIntegralE}. This completes the proof.
\end{proof}

\begin{proof} [Proof of Theorem \ref{Thm:Main}] We express our exponential sum over the range $y<n\leq x$ as
\begin{equation}
\label{Eq:SpliExponentialS}\sum_{ y<n\leq x} \frac{f(n)}{n} e(n\alpha)= \sum_{\substack{y<n\leq x  \\ P^+(n)\leq y}}  \frac{f(n)}{n} e(n\alpha)+ \sum_{\substack{n\leq x  \\ P^+(n)> y}}  \frac{f(n)}{n} e(n\alpha). 
\end{equation}
We first bound the contribution of the first sum.
To this end we split the range $y<n \leq x$ into two parts: $y<n\leq Y$, and $Y< n\leq x$, where $Y=\min(y^{\log\log y}, x).$ By \cite[Lemma 3.2] {BGGK18} the contribution of the second part is 
\[
\Bigg|\sum_{\substack{Y<n\leq x  \\ P^+(n)\leq y}}  \frac{f(n)}{n} e(n\alpha)\Bigg|\leq \sum_{\substack{n>Y  \\ P^+(n)\leq y}}  \frac{1}{n} \ll \frac 1{\log y}.
\]
On the other hand,  \eqref{Eq:AndrewRegis}  implies that if  $z\geq t\geq q^{3/2}$ then
\[
  \sum_{\substack{n\leq z \\ P^+(n)\leq t}}  f(n) e(\alpha n) \ll 
\sqrt{zt}  + \frac{z}{\sqrt q}  \log t  + \frac z {(\log z)^2},
\]
and so, by partial summation, if $t<z<Z$ then 
\[
\sum_{\substack{z<n\leq Z \\ P^+(n)\leq t}}  \frac{f(n)}n e(\alpha n)
  \ll \sqrt{t/z}  + \frac{\log (Z/z)}{\sqrt q}  \log t  + \frac 1 {\log z}. 
\]
In particular, this implies 
\begin{equation}\label{Eq:AndrewRegis2}
\sum_{\substack{y<n\leq Y \\ P^+(n)\leq y}}  \frac{f(n)}n e(\alpha n)\ll 1.
\end{equation}

We now turn our attention to the second sum on the right hand side of \eqref{Eq:SpliExponentialS}. First, if $ y\geq \sqrt{x}$, then the result holds trivially since 
$$  \sum_{\substack{n\leq x  \\ P^+(n)> y}}  \frac{f(n)}{n} e(n\alpha)= \sum_{y< p\leq x} \frac{f(p)}{p} \sum_{m\leq x/p} \frac{f(m)}{m} e(mp\alpha).$$
Therefore we might assume that $ y= q^{3/2+\delta}< \sqrt{x}$. Moreover, since  $|\alpha-a/q|\leq 1/(qx)$ we may replace $\alpha$ by $a/q$ at the expense of an error term of size $O(1)$.

Every integer $n\leq x$ with $P^{+}(n)>y$ can be written as $n=sb$  where $P^+(s)\leq y$ and  $b>1$ with $P^-(b)>y$. Let $p=P^-(b)$. The contribution of the $n$ for which $p^2$ divides $n$ is 
\[
\leq \sum_{p> y} \ \sum_{\substack{n\leq x \\ p^2|n}} \frac 1n = \sum_{p> y} \frac 1{p^2} \sum_{m\leq x/p^2} \frac 1m%\leq \sum_{p> y} \frac {1+\log^+(x/p)}{p^2} 
\ll \frac{\log x}{y\log y}.
\]
Therefore, we may assume that $p^2\nmid n$. We separate these integers into various classes:

(i) $b=pm$ with $m>1$ and $P^-(m)>p$. Now $m\geq P^-(m)>p$ and so $x\geq n= smp>sp^2$, which implies that $p<\sqrt{x/s}$.
These integers contribute
\[
\sum_{\substack{s\leq x \\ P^+(s)\leq y}} \frac{f(s)}s  \sum_{y<p\leq  \sqrt{x/s}}  \frac{f(p)} p \sum_{\substack{ 1<m\leq x/ps \\ P^-(m)>p}} 
\frac{f(m)}{m} e\left(\frac{mp\cdot sa}{q}\right)+O(1),
\]
and then by Lemma \ref{lem: Key}(i), this is 
\begin{align*}
& \ll \sum_{\substack{ s\leq x \\ P^+(s)\leq y \\  q_s=q/(q,s)}} \frac{1}s  \frac{2^{\omega(q_s)/2}}{\phi(q_s)^{1/2}} \log x 
\leq  \log x \cdot \sum_{d|q}  \sum_{\substack{   P^+(r)\leq y \\  (q/d,r)=1 }} \frac{1}{dr}  \frac{2^{\omega(q/d)/2}}{\phi(q/d)^{1/2}}\\
& \ll    \frac{\log x \log y}q \cdot \sum_{d|q}    2^{\omega(q/d)/2} \phi(q/d)^{1/2} 
 \ll   \frac{2^{\omega(q)}}{q^{1/2}}  \log x \log y,  
\end{align*}
upon writing $s=dr$, and since 
$$
\sum_{\substack{   P^+(r)\leq y \\  (q/d,r)=1 }} \frac{1}{r}\ll \frac{\phi(q/d)}{q/d}
\log y.$$
\medskip

(ii) $b=p$ and $s\geq y$.  These integers contribute $\displaystyle{\ll  \frac{2^{\omega(q)}}{\phi(q)^{1/2}}  \log x} $ by Lemma \ref{lem: Key}(ii).
\medskip

(iii) $b=p$ and $s< y$ yields the second term on the right hand side of \eqref{Eq:StatementThm1.1}. 
 This completes the proof.
\end{proof}
%%%%%%%%%%%%%%%%%%%%%%%%%%%%%%%%%%%%%%%%%%%%%%%%%%%%%%%%%%%%%%%%%%%%%%%%%%%%%%%%%%%%
\subsection{Reducing the number of terms in the sum: Proof of Theorem \ref{cor: Cut sum2}}\label{sec: Otherway}

 \begin{corollary} \label{cor: Cut sum} 
 Let $x$ be large and fix $ 0<\delta\leq 1/2$ and $0<\ep<1/10$, and suppose that  $\alpha$ is on a minor arc as in \eqref{eq: Minor Arc Range} with $  q \leq x^{1/2-\ep}$. If  $ y=q^{3/2+\delta}$   and $3\leq M <q^{1-\ep}$   then uniformly for all multiplicative functions $f:\mathbb{N}\to \mathbb{U}$ we have
 \[
 \sum_{y< n\leq x} \frac{f(n)}{n} e(n\alpha)=   
    \sum_{m\leq M} \frac{f(m)}{m} \sum_{y< p\leq x/m} \frac{f(p)}{p}   e(mp\alpha)  +O\bigg( 1 +   \bigg(    \frac { \log 2M  }{M}   \bigg)^{1/2}\log\log x\bigg). 
\]
\end{corollary}

Theorem \ref{cor: Cut sum2} follows from Corollary \ref{cor: Cut sum} by taking $y=q^2$ and $M=(\log\log x)^3$ (so that $M <q^{1-\ep}$ in the range \eqref{eq: Minor Arc Range}).

We deduce Corollary \ref{cor: Cut sum} from the following lemma inserted into \eqref{Eq:StatementThm1.1} of Theorem \ref{Thm:Main}, and with the error term there being $O(1)$ as $q\geq (\log x)^{2+\ep}$.

\begin{lemma} \label{lem: Cut sum} Let $x$ be large and fix $ 0<\delta\leq 1/2$ and $0<\ep<1/10$. 
Suppose that $|\alpha-a/q|\leq 1/(qx)$ with $(a, q)=1$ and $1\leq a\leq q\leq  x^{1/2-\ep}$. If  $ y=q^{3/2+\delta}$   and $3\leq M <q^{1-\ep}$   then uniformly for all multiplicative functions $f:\mathbb{N}\to \mathbb{U}$ we have
\begin{align*}
\sum_{\substack{y< p\leq x, m\leq y \\ mp\leq x}} \frac{f(mp)}{mp} e(mp\alpha) &=
\sum_{\substack{  m\leq M \\ y<p\leq  x/m}} \frac{f(mp)}{mp} e(mp\alpha)\\
& \quad \quad \quad \quad + O\bigg( \frac1q+ \frac{\log q}{\log x} +   \bigg(    \frac { \log 2M  }{M}   \bigg)^{1/2}\log\log x\bigg). 
\end{align*}
\end{lemma}

\begin{proof}   We prove the result for $a/q$ and note that the difference is bounded by 
\[
 \sum_{n\leq x} \bigg|\frac{f(n)}n (e(n\alpha)-e(n\tfrac aq)) \bigg| \leq \sum_{n\leq x} \frac{1}n| (e(n(\alpha-\tfrac aq))-1) |
\ll x |\alpha-\tfrac aq|\leq \frac 1q.
\]
 
We will prove that if $3\leq M\leq y/2$ then
\begin{equation}\label{Eq:TargetBoundDyadic}
  \sum_{\substack{M\leq m<2M \\ y<p\leq x/m}} \frac{f(mp)}{mp} e(mp\alpha)  \ll \frac 1{\log x} +
   \bigg(    \frac { \log M +   \delta_{q,M}  }{M}   \bigg)^{1/2}\log\log x,  
\end{equation}
where $\delta_{q,M}=1$ if $M<q^{1-\ep}$ and equals $\frac q{\phi(q)}$ otherwise.
Summing this bound over dyadic intervals with $M=y/2,y/4,\dots, M_0$, where $3\leq M_0\leq q^{1-\ep}$ gives the desired result.

We partition $y<p\leq x/m$ into $y<p\leq x/(2M)$ and $x/(2M)<p\leq x/m$. Since $M\leq y< x^{1-2\ep}$ we have $\log (x/M) \asymp \log x$ and so 
 \[
  \Bigg|\sum_{\substack{M\leq m<2M \\   x/(2M)<p\leq x/m}} \frac{f(mp)}{mp} e(mp\alpha)\Bigg| \leq
    \sum_{\substack{M\leq m<2M \\   x/(2M)<p\leq x/M}}  \frac{1}{mp} \ll \frac 1{\log(x/M)}\ll \frac 1{\log x}.
 \]
 For the remaining primes $p$ we observe that
\[
\Bigg|\sum_{\substack{M\leq m<2M \\ y<p\leq x/(2M)}} \frac{f(mp)}{mp} e(mp\alpha)\Bigg|^2  \leq \Bigg(\sum_{ M\leq m<2M }  \frac 1m \bigg| \sum_{y<p\leq x/(2M)} \frac{f(p)}{p} e(mp\alpha) \bigg|\Bigg)^2.
\]
By the Cauchy-Schwarz inequality this is 
\begin{align*}
 &\ll \sum_{ M\leq m<2M }  \frac 1m \bigg| \sum_{y<p\leq x/(2M)} \frac{f(p)}{p} e(mp\alpha) \bigg|^2\\
 &= \sum_{\substack{y<p,\ell\leq x/(2M)\\ p, \ell \text{ primes }}} \frac{f(p)\overline{f(\ell)}}{p\ell}  \sum_{ M\leq m<2M}  \frac {e(m(p-\ell)\alpha)}m\\
& \leq \sum_{b=0}^{q-1} \Big|\sum_{ M\leq m<2M }  \frac {e(mb\alpha)}m\Big| \sum_{\substack{r, s \bmod q\\ r-s \equiv b \bmod q\\ (rs, q)=1}} \ \ \sum_{\substack{y<p\leq x/(2M) \\ p\equiv r \bmod q}} \frac{1}{p} \sum_{\substack{y<\ell\leq x/(2M) \\ \ell\equiv s \bmod q}} \frac{1}{\ell} .
\end{align*}
Using the Brun-Titchmarsh theorem and partial summation we deduce that 
\begin{equation}\label{Eq:BoundBrunTitchmarsh}
 \Bigg|\sum_{\substack{M\leq m<2M \\ y<p\leq x/(2M)}} \frac{f(mp)}{mp} e(mp\alpha)\Bigg|^2\ll \bigg(\frac{\log (\tfrac{\log x}{\log y})}{\phi(q)}\bigg)^2\sum_{b=0}^{q-1} n_b \bigg|  \sum_{ M\leq m<2M }  \frac {e(mb\alpha)}m \bigg|,
 \end{equation}
where 
\begin{align*}
n_b:&=\# \{ r,s \pmod q:   r-s\equiv b \pmod q \ \& \ (rs,q)=1\} \\
& = \prod_{p^e\| q} \# \{ s \mod {p^e}:    (s(s+b),p)=1\} \\
& = q \prod_{p|(b,q)} \bigg(1-\frac 1p\bigg) \prod_{p\mid q, p\nmid b}\bigg(1-\frac 2p\bigg) \asymp \frac{\phi(q)^2}q \prod_{p|(b,q)}\frac p{p-1}.
\end{align*}
   If we write $ba/q=c/r$ where $(c,r)=1$ and $r=q/(q,b)$ then 
\begin{align*}
(e(c/r)-1) \sum_{   M\leq m<2M }  \frac {e(mb\alpha)}m &= (e(c/r)-1) \sum_{ M\leq m<2M }  \frac {e(mc/r)}m \\
& = \sum_{ M\leq m<2M}  \frac {e(mc/r)}{m(m-1)} + O\left(\frac{1}{M}\right) \ll \frac 1M.
\end{align*}
Therefore we obtain
\[
n_b\bigg|  \sum_{   M\leq m<2M }  \frac {e(mb\alpha)}m\bigg|  \ll  \frac{\phi(q)^2}q \prod_{p|q/r}\frac p{p-1} \cdot \min \bigg\{   \frac 1{M\| c/r\|} ,1   \bigg\},
\]
where $\|\cdot\|$ denotes the distance to the nearest integer. This gives
\begin{align}
\label{Eq:BoundSumDyadic} &\frac 1{\phi(q)}\sum_{\substack{0\leq b \leq q-1\\(b,q)=q/r}} n_b \bigg|  \sum_{ M\leq m<2M }  \frac {e(mb\alpha)}m \bigg| 
\ll   \frac{\phi(q)}q \prod_{p|q/r}\frac p{p-1}   \sum_{\substack{|c|\leq \frac r2 \\ (c,r)=1}}   \min \bigg\{   \frac r{M|c|} ,1   \bigg\}
\nonumber\\
& \quad \quad \quad \quad \quad \quad \quad \quad \ll  \frac{\phi(q)}q \prod_{p|q/r}\frac p{p-1}  \bigg( \sum_{\substack{ c\leq r/M \\ (c,r)=1}} 1 +  \sum_{\substack{r/M<c<r/2 \\ (c,r)=1} } \frac r{Mc} \bigg).
\end{align}
If $M\geq 3$ then the second sum here is $\displaystyle{\ll \frac rM\cdot \frac{\phi(r)}r \log M = \phi(r) (\log M)/M}$.
Therefore the total contribution of the second sum to the right hand side of \eqref{Eq:BoundSumDyadic}   is 
\begin{equation}\label{Eq:ContribSecondSum}
 \frac{\phi(q)\log M}{q M}  \sum_{r|q}  \prod_{p|q/r}\frac p{p-1}\cdot  \phi(r)  =\frac{\log M}{M}\frac{\phi(q)}q \cdot  q \prod_{p|q}  \bigg( 1+\frac{1}{p(p-1)}  \bigg)  \ll \phi(q)\frac{\log M}{M}.
\end{equation}
 We now focus on the first sum in \eqref{Eq:BoundSumDyadic}. This sum is $\ll r/M \leq (\phi(r)/M) \cdot(q/\phi(q)) $ and so its total 
contribution to the right hand side of \eqref{Eq:BoundSumDyadic} is  
\begin{equation}\label{Eq:ContribFirstSum}
 \frac{1}{M}  \sum_{r|q} \prod_{p|q/r}\frac p{p-1}\cdot  \phi(r)    \ll \frac {q}{M},   
\end{equation}
by the same calculation above.
However we can do better if $M<q^{1-\epsilon}$
since in the range $ r>q^{1-\varepsilon/2}$ we have $r/M >r^{\ep/2}$ and hence
$$
\sum_{\substack{ c\leq r/M \\ (c,r)=1}} 1 \ll  \frac rM\cdot \frac{\phi(r)}r =  \frac{\phi(r)}M.$$
Combining this with the weaker bound $r/M$ which we use in the range $r<q^{1-\ep/2}$, and appealing to the same calculation above shows that in the case $3\leq M< q^{1-\ep}$ the contribution of the first sum to the right hand side of \eqref{Eq:BoundSumDyadic} is  
\[
\frac{\phi(q)}{M}+\frac{\phi(q)}{qM}  \sum_{\substack{r|q\\ r<q^{1-\ep/2}}}  \prod_{p|q/r}\frac p{p-1}\cdot r
 \leq \frac{\phi(q)}{M}+\frac{1}{M}\sum_{\substack{r|q\\ r<q^{1-\ep/2}}} r \leq \frac{\phi(q)}{M}+\frac{q^{1-\ep/2}d(q)}{M}\ll \frac{\phi(q)}{M},
\]
 where $d(\cdot)$ is the divisor function. Therefore, using this estimate  together with \eqref{Eq:BoundSumDyadic},  \eqref{Eq:ContribSecondSum} and \eqref{Eq:ContribFirstSum}  implies that 
 $$\frac 1{\phi(q)}\sum_{0\leq b \leq q-1} n_b \bigg|  \sum_{ M\leq m<2M }  \frac {e(mb\alpha)}m \bigg|\ll \phi(q)\frac{\log M+\delta_{q, M}}{M}. 
 $$
 Inserting this bound in \eqref{Eq:BoundBrunTitchmarsh} completes the proof of \eqref{Eq:TargetBoundDyadic} and thus of the lemma.
 \end{proof}

%%%%%%%%%%%%%%%%%%%%%
\subsection{Bounding the logarithmically weighted sum over primes for almost all moduli $q$}\label{sec: LargePrimes}
In this section we prove the following proposition, which will be useful in the construction of the ``hybrid" multiplicative function  in Example \ref{Exa4}.

\begin{proposition} \label{Prop: small m} Let $x$ be large and $y\leq x^{1/3}$ be a real number. Let $ q\leq y$ be a positive integer such that 
\eqref{logBV} holds  for all $b$ with $(b,q)=1$.
Let $\mathcal M \subset [1, y] $ be a set of integers, and define $\mathcal L_1:=\sum_{m\in \mathcal M}\frac 1m$, and $\mathcal L_2=( \sum_{ m \in \mathcal M } \frac{1}{m^2} )^{1/2}$.   We also assume that
for each integer $d|q$ and each $(b,d)=1$ we have \begin{equation}\label{Eq:Hyp2}
\sum_{\substack{d<m\in \mathcal M\\   m\equiv b \bmod d}}\frac 1m\ll \frac {\mathcal L_1}d.
\end{equation}
Then for $\alpha=a/q$ with $(a, q)=1$ and all multiplicative functions $f:\mathbb{N}\to \mathbb{U}$ we have
\[
\bigg|\sum_{\substack{n=mp \\ m\in \mathcal M \\ y<p\leq x/y}} \frac{f(n)}{n} e(n\alpha)\bigg|
\leq \sum_{ y<p\leq x/y} \frac{1}{p} \cdot \bigg( \mathcal L_2   +O\bigg(\bigg(  \frac{1}{\log(\log x/\log y)}+ \frac {2^{\omega(q)}} {\phi(q) }  \bigg)   \frac{\mathcal L_1^2}{\mathcal L_2} \bigg)\bigg). 
\]
\end{proposition}

\begin{proof}  The absolute value of our sum can be written as 
\[
\bigg|\sum_{ y<p\leq x/y} \frac{f(p)}{p} \bigg(   \sum_{ m\in \mathcal M } \frac{f(m)}{m} e(m\cdot p\alpha) \bigg)\bigg|
    \leq S:= \sum_{ y<p\leq x/y} \frac{1}{p} \bigg|  \sum_{ m\in \mathcal M } \frac{f(m)}{m} e(m\cdot p\alpha) \bigg|
\]
and this  bound can be attained   by taking each $f(p)$ on the unit circle with  the conjugate angle to the given sum. We square and use the  Cauchy-Schwarz inequality so that 
\begin{equation}\label{Eq:CS_Sum}
S^2\leq S_2 \cdot  \sum_{ y<p\leq x/y} \frac{1}{p}
\end{equation} where
\[
S_2:=\sum_{ y<p\leq x/y} \frac{1}{p} \bigg|  \sum_{ m\in \mathcal M } \frac{f(m)}{m} e(m\cdot p\alpha) \bigg|^2 =
\sum_{ m,n\in \mathcal M } \frac{f(m)\overline{f(n)}}{mn}    \sum_{ y<p\leq x/y} \frac{e(p(m-n)\alpha)}{p} . 
\]
Taking $|f(m)\overline{f(n)}|\leq 1$ this is bounded by the sum of the diagonal terms (which yields our main term), plus the absolute value of the non-diagonal terms, so that 
\[
S_2\leq \mathcal{L}_2^2    \sum_{ y<p\leq x/y} \frac{1}{p}  +
\sum_{ m\ne n\in \mathcal M } \frac{1}{mn}  \bigg|  \sum_{ y<p\leq x/y} \frac{e(p(m-n)\alpha)}{p}  \bigg|.
\]
We now partition the latter sum into residue classes modulo $q$ to obtain
\[
  \bigg|   \sum_{(b,q)=1} e(b(m-n)a/q) \sum_{ \substack{y<p\leq x/y \\ p\equiv b \bmod q}} \frac{1}{p}\bigg|   
  = \frac{|c_q(m-n)|}{\phi(q)}  \sum_{ y<p\leq x/y} \frac{1}{p} +O(1)
 \]
 using our hypothesis \eqref{logBV} to evaluate the sums in arithmetic progressions. Summing over $m,n$ the total error term is $O(\mathcal L_1^2)$ and our main term is $\sum_{ y<p\leq x/y} \frac{1}{p}$ times
 
\begin{equation}\label{Eq:SumRamanujan2}
 \frac 1{\phi(q)} \sum_{ m\ne n\in \mathcal M } \frac{|c_q(m-n)|}{mn} \leq \frac 1{\phi(q) }  \sum_{d\mid q} d\mu\left(\frac{q}{d}\right)^2  \sum_{ \substack{m\ne n\in \mathcal M \\ m\equiv n \bmod d}} \frac{1}{mn}.
\end{equation}
If $1\leq b\leq d$ then by the hypothesis \eqref{Eq:Hyp2} we have
\[
\sum_{ \substack{m\ne n\in \mathcal M \\ m\equiv n\equiv b \bmod d}} \frac{1}{mn} \ll \frac {  \mathbf{1}_{b\in \mathcal M}}b\cdot  \frac{\mathcal L_1}d + \frac{\mathcal L_1^2}{d^2},
\]
so the left hand side of \eqref{Eq:SumRamanujan2} is 
\[
\ll \frac 1{\phi(q) }  \sum_{d\mid q}  \mu\left(\frac{q}{d}\right)^2 \sum_{b=1}^d \bigg( \frac {\mathbf{1}_{b\in \mathcal M}}b\cdot \mathcal L_1  + \frac{\mathcal L_1^2}{d}\bigg) \ll  \frac {2^{\omega(q)}} {\phi(q) }      \mathcal L_1^2 .
\]
Therefore
\begin{align*}
S_2 
&\leq \sum_{ y<p\leq x/y} \frac{1}{p} \bigg(  \mathcal L_2^2     +O\bigg(  \frac {2^{\omega(q)}} {\phi(q) }   \mathcal L_1^2 \bigg)\bigg)  +O( \mathcal L_1^2)\\
&=
\mathcal L_2^2\sum_{ y<p\leq x/y} \frac{1}{p} \bigg( 1 \mathcal      +O\bigg( \bigg( \frac 1{\log(\log x/\log y)}+ \frac {2^{\omega(q)}} {\phi(q) } \bigg)  \frac{\mathcal L_1^2}{\mathcal L_2^2} \bigg)\bigg). 
\end{align*}
Inserting this bound in \eqref{Eq:CS_Sum} completes the proof.
\end{proof}
 
\begin{remark}\label{Rem2}  We observe that
\[
 \sum_{ \substack{y<p\leq x/y \\ p\equiv b \bmod q}} \frac{1}{p} - \frac 1{\phi(q)}  \sum_{ \substack{y<p\leq x/y }} \frac{1}{p} =   \bigg[ \frac{\pi(t,q,b)-\frac{\pi(t)}{\phi(q)}}t\bigg]_y^{x/y} 
 + \int_y^{x/y}  \frac{\pi(t,q,b)-\frac{\pi(t)}{\phi(q)}}{t^2} dt
\]
and so if $E(x;Q):=\sum_{Q<q<2Q} \max_{(b,q)=1} |\pi(x,q,b)-\frac{\pi(x)}{\phi(q)}|$ then
\[
\sum_{Q<q<2Q} \max_{(b,q)=1} \bigg|  \sum_{ \substack{y<p\leq x/y \\ p\equiv b \bmod q}} \frac{1}{p} - \frac 1{\phi(q)}  \sum_{ \substack{y<p\leq x/y }} \frac{1}{p}  \bigg| \leq \frac{E(x/y;Q)}{x/y}+\frac{E(y;Q)}{y}
+  \int_y^{x/y}  \frac{E(t;Q)}{t^2} dt
\]
 which is $\ll  \frac 1{(\log y)^A}$ for any fixed constant $A>0$, provided $y>Q^{2+\ep}$ using the Bombieri-Vinogradov Theorem in the form $E(t;Q)\ll \frac t{(\log t)^{A+1}}$ for all $t>Q^{2+\ep}$.
 
 Now $\sum_{Q<q<2Q}\frac 1{\phi(q)} \asymp 1$ and so the hypothesis  \eqref{logBV} holds for almost all integers $q$ (in fact one can deduce that the number of exceptions in $[Q,2Q]$ is $\ll _A Q/(\log Q)^A$
 for any given positive constant $A$).
 \end{remark}

%%%%%%%%%%%%%%%%%%%%%%%%%%%%%%%%%%%%%%%%%%%%%%%%%%%%%%%%%%%%%%%%%%%%%%%%%%%%%%%%%%%%%%%%%%%%%%%%%%%%%%%%%%%%%%%%%%%%%%%%%%%%%%%%%%%%%%%%%%%%%%%%%%%%%%%%%%%%%%%%%%%%%%%%%%%%%%%%%%%%%%%%%%%%%%%%%%%%%%%%%%%%%%%%%%%%%%%%%%%%%%%%%%%%%%%%%%%%%%%%%%%%%%%%%%%%%

 \section{Logarithmic pretentious large sieve}

 In this section we will prove Theorems  \ref{Thm:PretentiousLSieve} and \ref{Thm:LogHala}.

\subsection{Preliminary results}
We will need the following auxiliary lemmas.
 \begin{lemma} 
\label{lem:Shiu}  Let $\ep>0$ be small and fixed. Let $q\geq 2$ be a positive integer and $1\leq a\leq q-1$ be such that  $(a, q)=1$. Let $x, y$ be real numbers such that $q^{1+\ep}<y <x/10$.  Then uniformly for all multiplicative functions $f:\mathbb{N}\to \mathbb{U}$ we have 
 $$ 
\sum_{\substack{ y \leq n \le x \\ n\equiv a \pmod q }} \frac{|f(n)|}{n} \ll \frac{\log(\log x/\log y)}{\phi(q)} \exp\Big( \sum_{\substack{ p\le x \\ p\nmid q}} \frac{|f(p)|}{p}\Big). 
 $$  
 \end{lemma} 
 \begin{proof} Splitting the range of summation in dyadic intervals and using Theorem 1 of Shiu \cite{Shiu80} we derive 
\begin{align*}\sum_{\substack{ y \leq n \le x \\ n\equiv a \bmod q }} \frac{|f(n)|}{n} 
&\leq \sum_{\log y/\log 2-1\leq j\leq \log x/\log 2+1} \frac{1}{2^j}\sum_{\substack{ 2^j < n \le 2^{j+1} \\ n\equiv a \bmod q }} |f(n)|\\
&\ll \frac{1}{\phi(q)}\sum_{\log y/\log 2-1\leq j\leq \log x/\log 2+1} \frac{1}{j} \exp\Big( \sum_{\substack{ p\le x \\ p\nmid q}} \frac{|f(p)|}{p}\Big)\\
& \ll \frac{\log(\log x/\log y)}{\phi(q)} \exp\Big( \sum_{\substack{ p\le x \\ p\nmid q}} \frac{|f(p)|}{p}\Big),
 \end{align*}
as desired.   
 \end{proof} 
\begin{lemma}\label{Lem:SecondMomentPrimes}
     Let $0<\lambda<1/4$ and $\{a_n\}_{n\geq 1}$ be arbitrary complex numbers with $|a_n|\leq 1$. Let $q\geq 2$ be an integer and $z>q^{1+\ep}$ be a real number. For any $t\in \mathbb{R}$ we have $$\sum_{\chi\bmod q} \left|\sum_{P^{-}(n)>z} \frac{a_n\Lambda(n) \chi(n)}{n^{1+\lambda+it}}\right|^2 \ll \frac{1}{z^{2\lambda} \lambda^2}. 
     $$ 
 \end{lemma}
 \begin{proof}
      By the orthogonality relations for characters we have
      
\begin{equation}\label{Eq:OrthogonalitySecondMoment}
\sum_{\chi\bmod q} \left|\sum_{P^{-}(n)>z} \frac{a_n\Lambda(n) \chi(n)}{n^{1+\lambda+it}}\right|^2 \leq \phi(q)\sum_{P^{-}(n)>z} \frac{\Lambda(n)}{n^{1+\lambda}}\sum_{\substack{P^{-}(m)>z\\ m\equiv n\bmod q}} \frac{\Lambda(m)}{m^{1+\lambda}}.
\end{equation}
Fix  $n>z$. By the Brun-Tichmarsh Theorem the inner sum over $m$ equals 
\begin{align*}
\sum_{\substack{p>z\\ p\equiv n\bmod q}} \frac{\log p}{p^{1+\lambda}} +O\left(\frac{1}{z}\right)
 &\ll \sum_{j\geq \log z/\log 2-1 } \frac{j}{2^{j(1+\lambda)}} \sum_{\substack{2^j< p\leq  2^{j+1}\\  p\equiv n\bmod q}}1 +\frac{1}{z}\\
 & \ll \frac{1}{\phi(q)} \sum_{j\geq \log z/\log 2-1 } \frac{1}{2^{\lambda j}}+\frac{1}{z}\ll \frac{1}{\phi(q) z^{\lambda} \lambda}.
\end{align*}
 Inserting this bound in \eqref{Eq:OrthogonalitySecondMoment} completes the proof. 
 \end{proof}

\subsection{Proof of Theorem \ref{Thm:LogHala}}

  %We start by the orthogonality relation for characters which gives 
%$$ \sum_{\substack{n\leq x\\ n\equiv a \bmod q}}\frac{f(n)}{n}=\frac{1}{\phi(q)}\sum_{\chi\bmod q} \chi(a) \sum_{n\leq x} \frac{f(n) \overline{\chi}(n)}{n}.$$
%Since the result is trivial if $T<1/\log x$, we assume throughout that $T\geq 1/\log x$. 
 To shorten our notation we let
\begin{align*}
E_{f,{\mathcal X}}(x;q;a):&=\sum_{\substack{y\leq n\leq x\\ n\equiv a \bmod q}}  \frac{f(n)}{n} - \frac 1{\phi(q)} \sum_{\substack{\chi \bmod q\\ \chi\in {\mathcal X}}} \chi(a) \sum_{y\leq n\leq x} \frac{f(n)\overline{\chi}(n)}{n}\\
& =\frac 1{\phi(q)} \sum_{\substack{\chi\bmod q\\ \chi \notin \mathcal{X}}} \chi(a) \sum_{y\leq n\leq x} \frac{f(n)\overline{\chi}(n)}{n},    
\end{align*}
where the last equality follows by the orthogonality relations of characters.
Let $z> y$ be a parameter to be chosen. We define the multiplicative functions 
$s(\cdot)$  and $\ell(\cdot)$   by
\[
 s(p^k)=\begin{cases} f(p^k) \\  0   \end{cases}  
 \ \ \text{and} \ \ 
\ell(p^k)=\begin{cases} 0 & \text{if} \ p\leq z,\ k\geq 1 \\  f(p^k) & \text{if} \ p> z,\ k\geq 1;
      \end{cases}
\]
so that $f$ is the convolution of $s$ and $\ell$. For a Dirichlet character  $\chi \bmod q$ we define for Re$(s)>1$
\[
 {\mathcal S}_{\chi}(s) = \sum_{n\geq 1}\frac{s(n)\overline{\chi}(n)}{n^{s}},  \ \   \  \
{\mathcal L}_{\chi}(s) 
=\sum_{n\geq 1}\frac{\ell(n)\overline{\chi}(n)}{n^{s}}, \ \ \textup{ and } \ \ - \frac{\mathcal {L}^{\prime}_{\chi}(s)}{\mathcal {L}_{\chi}(s)} = \sum_{n=2}^{\infty}  \frac{\Lambda_{\ell}(n)\overline{\chi}(n)}{n^{s}}. \ \  
\]
Then for $\re(s)>1$ we have 
$$ F_{\chi}(s):= \sum_{n\geq 1} \frac{f(n)\overline{\chi}(n)}{n^s}= \eS_{\chi}(s) \eL_{\chi}(s). $$
Let  $c_0:=1/\log x$ and $\eta:=1/\log z.$ By a slight variation of Lemma 2.2 of \cite{GHS19} we obtain 
\begin{equation}
\label{Eq:SumPerronAp1}
\int_{\alpha=0}^{\eta}\int_{\beta=0}^{2\eta} \frac{1}{2\pi i } \int_{c_0-i\infty}^{c_0+i\infty} 
\eS_{\chi}(s+1)\eL_{\chi}(s+1+\alpha+\beta) \frac{\eL^{\prime}_{\chi}}{\eL_{\chi}}(s+1+\alpha)\frac{\eL^{\prime}_{\chi}}{\eL_{\chi}}(s+1+\alpha+\beta) \frac{x^s}{s} \ ds \ d\beta \ d\alpha
\end{equation}
\begin{align}
\label{Eq:SumPerronAp2} 
&= \sum_{n\le x} \frac{f(n)\overline{\chi}(n)}{n}  -\sum_{mn\le x} \frac{s(m)\overline{\chi}(m)}{m} \frac{\ell(n)\overline{\chi}(n)}{n^{1+\eta}} \nonumber \\
&  \quad \quad \quad \quad \quad \quad \quad \quad - \int_0^{\eta} \sum_{mnk \le x} \frac{s(m)\overline{\chi}(m)}{m} \frac{\Lambda_{\ell}(k)\overline{\chi}(k)}{k^{1+\alpha}} \frac{\ell(n)\overline{\chi}(n)}{n^{1+2\eta+\alpha}} d\alpha, \nonumber \\
&= \sum_{y\leq n\le x} \frac{f(n)\overline{\chi}(n)}{n}  -\sum_{y\leq mn\le x} \frac{s(m)\overline{\chi}(m)}{m} \frac{\ell(n)\overline{\chi}(n)}{n^{1+\eta}} 
\nonumber \\
& \quad \quad \quad \quad \quad \quad \quad \quad- \int_0^{\eta} \sum_{mnk \le x} \frac{s(m)\overline{\chi}(m)}{m} \frac{\Lambda_{\ell}(k)\overline{\chi}(k)}{k^{1+\alpha}} \frac{\ell(n)\overline{\chi}(n)}{n^{1+2\eta+\alpha}} d\alpha,
\end{align}
where the last equality follows since $s(n)=f(n)$ and  $\ell(n)=0$ for all $1<n\leq y$.
 We multiply \eqref{Eq:SumPerronAp1} and \eqref{Eq:SumPerronAp2} by $\chi(a)/\phi(q)$ and sum over all characters  $\chi$ modulo $q$ such that $\chi\not\in\mathcal{X}$.
 The contribution of the second term in \eqref{Eq:SumPerronAp2} is  
\[
\leq \sum_{\substack{y\leq mn \le x \\ mn\equiv a \bmod q}} \frac{|s(m)|}{m} \frac{|\ell(n)|}{n^{1+\eta}}+ \frac{ |{\mathcal X}|}{\phi(q)} \sum_{y\leq mn \le x} \frac{|s(m)|}{m} \frac{|\ell(n)|}{n^{1+\eta}}
\]
and then applying Lemma \ref{lem:Shiu} to the multiplicative function which equals $f(p)/p$ on the primes $p\leq z$, and $f(p)/p^{1+\eta}$ on the primes $p> z$, this is
\begin{equation} \label{Eq:ErrSecondAp} 
\begin{aligned}
& \ll \frac{\log(\log x/\log y) +|\mathcal{X}|}{\phi(q) } \exp\Big(\sum_{p\le z} \frac{|f(p)|}{p} + \sum_{z< p\le x} \frac{|f(p)|}{p^{1+\eta}} \Big)  \\
&
\ll  \frac{(\log(\log x/\log y)+|\mathcal{X}|)\log z}{\phi(q)}.
\end{aligned}
\end{equation}
 
By swapping the sum and integral and performing the integration over $\alpha$ we find that the contribution of the third term in \eqref{Eq:SumPerronAp2}  is
\begin{align*}
& \ll\sum_{\substack{mnk \le x \\ mnk\equiv a \bmod q}} \frac{|s(m)|}{m} \frac{|\ell(n)|}{n^{1+2\eta}} \frac{|\Lambda_{\ell}(k)|}{k\log k} + \frac{|\mathcal{X}|}{\phi(q)}\sum_{mnk \le x } \frac{|s(m)|}{m} \frac{|\ell(n)|}{n^{1+2\eta}} \frac{|\Lambda_{\ell}(k)|}{k\log k} \\
& \leq \sum_{\substack{mn \le x/z \\ (mn, q)=1}} \frac{|s(m)|}{m} \frac{|\ell(n)|}{n^{1+2\eta}} \sum_{\substack{z<k\leq x/(mn)\\ k\equiv (mn)^{-1}a\bmod q}} \frac{\Lambda(k)}{k\log k} \\
& \quad \quad \quad + \frac{|\mathcal{X}|}{\phi(q)}\sum_{mn \le x/z } \frac{|s(m)|}{m} \frac{|\ell(n)|}{n^{1+2\eta}}\sum_{z<k\leq x/(mn)}\frac{\Lambda(k)}{k\log k} \\
& \ll  \frac{(1+|\mathcal{X}|)\log(\log x/\log y)\log z}{\phi(q)},
\end{align*}
applying Lemma \ref{lem:Shiu} exactly as in \eqref{Eq:ErrSecondAp}, and
since  
$$\sum_{\substack{y<k\leq x\\ k\equiv b\bmod q}} \frac{\Lambda(k)}{k\log k} \ll \frac{\log(\log x/\log y)}{\phi(q)}$$
by the Brun-Titchmarsh Theorem. 

We now turn to the task of bounding the contribution of the integral on \eqref{Eq:SumPerronAp1}. 
We first fix $\alpha$ and $\beta$ and consider the inner integral over $s$ multiplied by $\chi(a)/\phi(q)$ and summed over  $\chi\not\in\mathcal{X}$, which equals 
\begin{equation}\label{Eq:PerronIntegral}
\begin{aligned}
&\frac{1}{\phi(q)}\sum_{\substack{\chi\bmod q\\ \chi \notin \mathcal{X}}}\chi(a)  \frac{1}{2\pi i } \int_{c_0-i \infty}^{c_0+i \infty} 
\eS_{\chi}(s+1)\\
& \quad \quad \quad \quad \times \eL_{\chi}(s+1+\alpha+\beta) \sum_{P^{-}(n)>z} \frac{\Lambda_\ell(n)\overline{\chi}(n)}{n^{s+1+\alpha}}\sum_{P^{-}(n)>z} \frac{\Lambda_\ell(n)\overline{\chi}(n)}{n^{s+1+\alpha+\beta}} \frac{x^s}{s} ds. 
\end{aligned}
\end{equation}
By the proof of Lemma 2.5 of \cite{GHS19} it follows from the quantitative Perron's formula that the part of the integral in \eqref{Eq:PerronIntegral} with $|t|\geq U$, integrated over $0\leq \alpha, \beta\leq 2/\log z$ is $\ll (\log x)^2/U.$ We choose $U=(q\log x)^2$ and  $z=U^{50}.$ Therefore, we might replace the integral \eqref{Eq:PerronIntegral} by the following 
\begin{equation}\label{Eq:IntegralPerron2}
\begin{aligned}
&\frac{1}{\phi(q)}\sum_{\substack{\chi\bmod q\\ \chi \notin \mathcal{X}}} \chi(a)  \frac{1}{2\pi i } \int_{c_0-i U}^{c_0+i U} 
\eS_{\chi}(s+1)\\
& \quad \quad \quad \quad \times \eL_{\chi}(s+1+\alpha+\beta) \sum_{P^{-}(n)>z} \frac{\Lambda_\ell(n)\overline{\chi(n)}}{n^{s+1+\alpha}}\sum_{P^{-}(n)>z} \frac{\Lambda_\ell(n)\overline{\chi(n)}}{n^{s+1+\alpha+\beta}} \frac{x^s}{s} ds  
\end{aligned}
\end{equation}
at the cost of an error term of size $O(1/q^2)$ in \eqref{Eq:InequalityHALASZ}, which is acceptable. 
Moreover, for $s=c_0+it$ we have 
\begin{align*}
 \left|\frac{\eS_{\chi}(s+1)}{\eS_{\chi}(s+1+\alpha+\beta)}\right| & \ll \exp\left(\sum_{p\leq z} \frac{1}{p^{1+c_0}}- \frac{1}{p^{1+c_0+\alpha+\beta}}\right)\\
 & \ll \exp\left((\alpha+\beta)\sum_{p\leq z} \frac{\log p}{p}\right)\ll  1.
 \end{align*} Therefore, the expression in \eqref{Eq:IntegralPerron2} is 
\begin{equation}\label{Eq:IntegralPerron3}
\begin{aligned}
 & \ll\frac{1}{\phi(q)}\int_{-U}^{U}\sum_{\substack{\chi\bmod q\\ \chi \notin \mathcal{X}}} \left|F_{\chi}(1+c_0+\alpha+\beta+it)\right|\\
 & \quad \quad \quad \quad \quad \times\bigg|\sum_{P^{-}(n)>z} \frac{\Lambda_\ell(n)\overline{\chi}(n)}{n^{1+c_0+\alpha+it}}\bigg|\bigg|\sum_{P^{-}(n)>z} \frac{\Lambda_\ell(n)\overline{\chi}(n)}{n^{1+c_0+\alpha+\beta+it}} \bigg| \frac{dt}{|c_0+it|}.
\end{aligned}
\end{equation}
We split the integral over $t$ into two parts $|t|\leq T$ and $T< |t|\leq U$. In the second part we use that 
$\left|F_{\chi}(1+c_0+\alpha+\beta+it)\right|\leq \zeta(1+c_0+\alpha+\beta)\ll (c_0+\alpha+\beta)^{-1}$ uniformly in $t$ and $\chi$. Therefore we deduce that the contribution of this part is 
\begin{align*}
&\ll \frac{1}{(c_0+\alpha+\beta)T\phi(q)}\int_{-U}^{U} \sum_{\chi\bmod q}\bigg|\sum_{P^{-}(n)>z}\frac{\Lambda_\ell(n)\overline{\chi}(n)}{n^{1+c_0+\alpha+it}}\bigg|\bigg|\sum_{P^{-}(n)>z} \frac{\Lambda_\ell(n)\overline{\chi}(n)}{n^{1+c_0+\alpha+\beta+it}} \bigg| dt\\
& \ll \frac{1}{(c_0+\alpha+\beta)T\phi(q)} \\
& \quad \times\bigg(\int_{-U}^{U} \sum_{\chi\bmod q}\bigg|\sum_{P^{-}(n)>z}\frac{\Lambda_\ell(n)\overline{\chi}(n)}{n^{1+c_0+\alpha+it}}\bigg|^2 dt\cdot \int_{-U}^{U} \sum_{\chi\bmod q}\bigg|\sum_{P^{-}(n)>z}\frac{\Lambda_\ell(n)\overline{\chi}(n)}{n^{1+c_0+\alpha+\beta+it}}\bigg|^2 dt\bigg)^{1/2}
\end{align*}
by using the Cauchy-Schwarz inequality twice (first over the sum over $\chi$ and then  over the integral over $t$). By Lemma 3.2 of \cite{GHS19} as $z>q^{3}$ this is 
\begin{equation}\label{Eq:ContributionIntegralLargeT}
\begin{aligned}
 &\ll \frac{1}{(c_0+\alpha+\beta)T\phi(q)} \left(\sum_{n>z}\frac{\Lambda(n)}{n^{1+2c_0+2\alpha}}\right)^{1/2} \left(\sum_{n>z}\frac{\Lambda(n)}{n^{1+2c_0+2\alpha+2\beta}}\right)^{1/2}\\
 & \ll \frac{1}{(c_0+\alpha)^{1/2}(c_0+\alpha+\beta)^{3/2} T\phi(q)}. 
\end{aligned}
\end{equation}
We now handle the contribution of the part of the integral in \eqref{Eq:IntegralPerron3} over $|t|\leq T$, which is
\begin{equation}\label{Eq:IntegralPerron4}
\begin{aligned}&\leq \frac{1}{\phi(q)}\max_{\substack{\chi \notin \mathcal{X}\\ |t|\leq T}} \left|F_{\chi}(1+c_0+\alpha+\beta+it)\right|   \\
 & \quad \quad \quad \quad \quad \times\int_{-T}^{T}\sum_{\chi\bmod q}\bigg|\sum_{P^{-}(n)>z} \frac{\Lambda_\ell(n)\overline{\chi}(n)}{n^{1+c_0+\alpha+it}}\bigg|\bigg|\sum_{P^{-}(n)>z} \frac{\Lambda_\ell(n)\overline{\chi}(n)}{n^{1+c_0+\alpha+\beta+it}} \bigg| \frac{dt}{|c_0+it|}.
\end{aligned}
\end{equation}
By Lemma \ref{Lem:SecondMomentPrimes} and the Cauchy-Schwarz inequality  the inner sum over $\chi$ is bounded by 
\begin{equation}\label{Eq:CauchySchwarzCHI}
\begin{aligned}
&\bigg(\sum_{\chi\bmod q}  \bigg|\sum_{P^{-}(n)> z} \frac{\Lambda_\ell(n\overline{\chi}(n)}{n^{1+c_0+\alpha+it}}\bigg|^2\bigg)^{1/2}\bigg(\sum_{\chi\bmod q}  \bigg|\sum_{P^{-}(n)>z} \frac{\Lambda_\ell(n)\overline{\chi}(n)}{n^{1+c_0+\alpha+\beta+it}}\bigg|^2\bigg)^{1/2}\\
& \ll \frac{1}{(c_0+\alpha) (c_0+\alpha+\beta)}.
\end{aligned}
\end{equation}
%since $y^{c_0+\alpha+\beta}\asymp 1.$
This implies that the integral over $t$ in \eqref{Eq:IntegralPerron4}  is
$$ \ll \frac{1}{(c_0+\alpha) (c_0+\alpha+\beta)}  \int_{-T}^T \frac{dt}{\sqrt{c_0^2+t^2}}\ll \frac{\log (T\log x)}{(c_0+\alpha) (c_0+\alpha+\beta)},$$
by a change of variable $u=t/c_0=t\log x.$
Combining this estimate with \eqref{Eq:ContributionIntegralLargeT} and \eqref{Eq:IntegralPerron4}, we deduce that the quantity in \eqref{Eq:IntegralPerron2} is 
\begin{align*} &\ll \frac{\log (T\log x)}{(c_0+\alpha) (c_0+\alpha+\beta)\phi(q)}\max_{\substack{\chi \notin \mathcal{X}\\ |t|\leq T}} \left|F_{\chi}(1+c_0+\alpha+\beta+it)\right|\\
& \quad \quad \quad \quad + \frac{1}{(c_0+\alpha)^{1/2}(c_0+\alpha+\beta)^{3/2} T\phi(q)}. \end{align*}
Integrating this over $\alpha$ and $\beta$, and making the change of variables $\sigma=c_0+\alpha+\beta$ for the integral over $\beta$, we deduce that the integral in \eqref{Eq:SumPerronAp1} multiplied by $\chi(a)/\phi(q)$ and summed over $\chi\notin \mathcal{X}$ is   
\begin{align}\label{Eq:FinalBoundIntegralPerron} &\ll \frac{\log(T\log x)}{\phi(q)}\int_{0}^{\eta}\frac{1}{c_0+\alpha}\left(\int_{c_0+\alpha}^{2\eta+c_0+\alpha}\max_{\substack{\chi \notin \mathcal{X}\\ |t|\leq T}} \left|F_{\chi}(1+\sigma+it)\right| \frac{d\sigma}{\sigma} \right)d\alpha \nonumber\\
& \quad \quad \quad \quad + \frac{1}{T\phi(q)}\int_{0}^{\eta} \frac{1}{(c_0+\alpha)^{1/2}}\left(\int_{c_0+\alpha}^{2\eta+c_0+\alpha}\frac{d\sigma}{\sigma^{3/2}} \right) d\alpha +\frac1{q^2} \nonumber\\
& \ll \frac{(\log\log  x)\log (T\log x)}{\phi(q)}\int_{1/\log x}^{1} \max_{\substack{\chi \notin\mathcal{X}\\ |t|\leq T}} \left|F_{\chi}(1+\sigma+it)\right|\frac{d\sigma}{\sigma} +\frac{\log\log x}{T\phi(q)},
\end{align}
since 
$$\int_{0}^{\eta} \frac{1}{(c_0+\alpha)^{1/2}}\left(\int_{c_0+\alpha}^{2\eta+c_0+\alpha}\frac{d\sigma}{\sigma^{3/2}}\right) d\alpha \ll \int_{0}^{\eta} \frac{1}{c_0+\alpha}d\alpha \leq  \int_{1/\log x}^{1}\frac{1}{\alpha}d\alpha=  \log\log x.$$
Using that $|F_{\chi}(1+\sigma+it)|\leq\zeta(1+\sigma)\ll 1/\sigma$ uniformly in $\chi$ and $t$, we observe that
$$ \int_{T}^{1}  \max_{\substack{\chi \notin \mathcal{X}\\ |t|\leq T}} \left|F_{\chi}(1+\sigma+it)\right|\frac{d\sigma}{\sigma} \ll \int_{T}^{1}\frac{d\sigma}{\sigma^2} \ll \frac{1}{T}.$$
Furthermore, for any character $\chi \bmod q$,  Lemma 3.3 of \cite{LaMa22} implies that 
$$ \max_{|t|\leq T} |F_{\chi}(1+\sigma+it)|\ll (\log x)e^{-M(f\overline{\chi}; x, T)},$$
 uniformly in the range $1/\log x\leq \sigma\leq T$. Inserting these estimates in \eqref{Eq:FinalBoundIntegralPerron} yields
 \begin{equation}\label{Eq:InequalityHALASZ}
\begin{aligned}
E_{f,{\mathcal X}}(x;q;a) & \ll \frac {(\log x)(\log\log x)(\log(T\log x))^2} {\phi(q)}   \exp\left(-M_{\mathcal X, f}(x, T)\right) \\
&  \quad \quad \quad \quad  + \frac{|\mathcal X|}{\phi(q)}(\log\log x)^2+ \frac{(\log\log x) \log(T\log x)}{\phi(q)T},
\end{aligned}
\end{equation}
 and this can be simplified since $T\leq 1$ to yield \eqref{Eq:InequalityHALASZ2}. \hfil \qed

\subsection{Proof of Theorem  \ref{Thm:PretentiousLSieve}}
We observe that
\begin{align*} 
\sum_{(a,q)=1}   
& \Big| E_{f,{\mathcal X}}(x;q,a) \Big|^2
= \sum_{(a,q)=1} \Big|\frac{1}{\phi(q)} \sum_{\substack{\chi \bmod q\\ \chi\not\in {\mathcal X}}} \chi(a) \sum_{y\leq n\leq x}\frac{f(n)\overline{\chi}(n)}{n}\Big|^2\\
& \quad = \frac{1}{\phi(q)^2} \sum_{\substack{\chi, \psi \bmod q\\ \chi, \psi \not\in {\mathcal X}}} \ \sum_{y\leq n\leq x}\frac{f(n)\overline{\chi}(n)}{n} \sum_{y\leq m\leq x}\frac{\overline{f(m)}\psi(m)}{m} \sum_{a \bmod q} \chi(a)\overline{\psi(a)} 
\\
& \quad = \frac{1}{\phi(q)} \sum_{\substack{\chi \bmod q\\ \chi\not\in {\mathcal X}}} \Big|\sum_{y\leq n\leq x}\frac{f(n)\overline{\chi}(n)}{n}\Big|^2. 
\end{align*}
Bounding each   $E_{f,{\mathcal X}}(x;q,a)$ by  \eqref{Eq:InequalityHALASZ}   and summing over $(a, q)=1$ yields
\begin{equation}\label{eq: PLS log2}
\begin{aligned} 
\sum_{\substack{\chi \bmod q\\ \chi\not\in {\mathcal X}}} \Big|\sum_{y  \leq n\leq x}\frac{f(n)\overline{\chi}(n)}{n}\Big|^2
& \ll (\log x)^2(\log\log x)^2\log(T\log x)^4   \exp\left(-2M_{\mathcal X, f}(x, T)\right) \\
&    + |\mathcal X|^2(\log\log x)^4+ \frac{(\log\log x)^2\log(T\log x)^2}{T^2}
\end{aligned}
\end{equation}
and \eqref{eq: PLS log} follows as $T\leq 1$. \hfil \qed

%%%%%%%%%%%%%%%%%%%%%%%%%%%%%%%%%%%%%%%%%%%%%%%%%%%%%%%%%%%%%%%%%%%%%%%%%%%%%%%%%%%%%%%%%%%%%%%%%%%%%%%%%%%%%%%%%%%%%%%%%%%%%%%%%%%%%%%%%%%%%%%%%%%%%%%%%%%%%%%%%%%%%%%%%%%%%%%%%%

 \section{ Exponential sums in the logarithmic case: The major arcs}
   
   %Logarithmically weighted exponential sums}

  %%%%%%%%%%%%%%%%%%%%%%%%%%%%%%%%%%%%%%%%%%%%%%%%%%%%%%%%%%%%%%%%%%
  
%%%%%%%%%%%%%%%%%%%%%%%%%%%%%
%\section{The major arcs case} \label{sec: Small $q$}
 In this section we prove Theorem \ref{Thm:Smallq} and Corollary \ref{Cor:Smallq3}, and deduce a Hal\'asz-type upper bound for logarithmically weighted exponential sums with multiplicative  coefficients.  We will also classify the multiplicative functions for which  \eqref{Eq:MVThree} is attained and prove Theorem \ref{Thm:Classify_Large}.
 
 \subsection{Proof of Theorem \ref{Thm:Smallq}}
Since   $|\alpha-a/q|\leq 1/(qx)$ we may replace $\alpha$ by $a/q$ at the expense of an error term of size $O(1)$. We start by recording the following identity (see for example   \cite[Proposition 2.3]{Go12}) 
\begin{equation}\label{Eq:IdentityCharacters}
\sum_{q^{3}\leq n\leq x} \frac{f(n)}{n} e\bigg(\frac{an}{q}\bigg)= \sum_{r\mid q}\frac{f(r)}{r}\frac{1}{\phi(q/r)}\sum_{\chi \bmod q/r}%\mathcal{G}
\tau(\chi) \overline{\chi}(a)\sum_{q^3/r\leq n\leq x/r} \frac{f(n) \overline{\chi}(n)}{n}.
\end{equation}
Let $2/\log x\leq T\leq 1$ and order the primitive characters $\chi\bmod r$ for $r\mid q$ (including the trivial character $\chi_0$ which equals $1$ for all integers)  as $\{\chi_j \bmod \ell_j\}_{j\geq 1}$, where
\begin{equation*}
M(f\overline{\chi_j}, x; T) \leq M(f\overline{\chi_{j+1}};x, T), 
\end{equation*}
for all $j \geq 1$. Then it follows from Lemma 3.3 of \cite{BGS13} that 
\begin{equation}\label{Eq:BaGrSo}
    M(f\overline{\chi_{2}};x, T)\geq \left(\frac{1}{3}+o(1)\right)\log\log x. 
\end{equation}Moreover, by a slight variation of Lemma 3.1 of \cite{BGS13} we have  
\begin{equation}\label{Eq:BoundRepulsive}
M\left(f\overline{\chi_j};x, T\right) \geq 
\left(1-\frac{1}{\sqrt{j}}\right)\log\log x+O\left(\sqrt{\log\log x}\right).
\end{equation}
Therefore, if $\chi \bmod r$ is induced by $\chi_{j}$, then
\begin{equation}\label{Eq:DistanceImprimitive}
\begin{aligned}
M\left(f\overline {\chi};x, T\right) \geq M\left(f\overline{\chi_j};x, T\right) +O\Big(\sum_{p \mid r} \frac{1}{p}\Big) 
\geq M\left(f\overline{\chi_j};x, T\right) +O(\log_3 x),
\end{aligned}
\end{equation}
since $\sum_{p\mid r}1/p\ll \log_2 r\ll \log_3 x$. 
 For each $r\mid q$, let $\mathcal{X}(r)$ be the set of characters $\chi\bmod q/r$ such that $\chi$ is induced by one of the $\chi_j$ for $1\leq j\leq J-1$. Note that $|\mathcal{X}(r)|\leq J-1$. Using the Cauchy-Schwarz inequality, together with Theorem \ref{Thm:PretentiousLSieve} and the estimates   \eqref{Eq:BoundRepulsive} and \eqref{Eq:DistanceImprimitive},   we deduce that the contribution to \eqref{Eq:IdentityCharacters}  of all characters $\chi \notin \cup_{r\mid q} \mathcal{X}(r) $ is
\begin{align}\label{Eq:BadCharPretentious}
&\leq\sum_{r\mid q} \frac{1}{r\phi(q/r)} \sum_{\substack{\chi \bmod q/r\\ \chi\notin \mathcal{X}(r)}}|\tau(\chi)| \Bigg|\sum_{q^3/r\leq n\leq x/r} \frac{f(n) \overline{\chi}(n)}{n}\Bigg|\nonumber\\
& \leq \sum_{r\mid q} \frac{\sqrt{q/r}}{r\sqrt{\phi(q/r)}}\Bigg(\sum_{\substack{\chi \bmod q/r\\ \chi\notin \mathcal{X}(r)}} \Bigg|\sum_{q^3/r\leq n\leq x/r} \frac{f(n) \overline{\chi}(n)}{n}\Bigg|^2\Bigg)^{1/2}\nonumber\\
& \ll 
  (\log x)^{1/\sqrt{J}}\exp\left(O\left(\sqrt{\log\log x}\right)\right)+ \frac{(\log\log x)^4}{T}
\end{align}
since $|\tau(\chi)|\leq \sqrt{q/r}$ and 
\begin{equation}
\label{Eq:BoundFirstCoefficientSumSmall}\sum_{r\mid q} \frac{\sqrt{q/r}}{r\sqrt{\phi(q/r)}}\leq \left(\frac{q}{\phi(q)}\right)^{3/2}\ll (\log\log q)^{3/2}\ll (\log\log \log x)^{3/2}.
\end{equation}
If $J=2$, we can replace the term $(\log x)^{1/\sqrt{J}}\exp(O(\sqrt{\log\log x}))$ on the right hand side of \eqref{Eq:BadCharPretentious} by $(\log x)^{2/3+o(1)}$ by \eqref{Eq:BaGrSo}.

Thus, it now remains to estimate the contribution of the characters $\chi\bmod q/r$ for $r\mid q$, that are induced by $\chi_j \bmod \ell_{j}$ for some $1\leq j\leq J-1$. Fix such a $j$. If $\chi\bmod q/r$ is induced by $\chi_j$ then we must have $r \mid (q/\ell_j)$. In this case Lemma 4.1 of \cite{GS07} gives
$$\tau(\chi)= \mu\left(\frac{q}{r\ell_{j}}\right)\chi_j\left(\frac{q}{r\ell_j}\right)\tau(\chi_j).$$
Therefore, the contribution of such characters to \eqref{Eq:IdentityCharacters} is
\begin{equation}\label{Eq:ContributionInduced}
\overline{\chi_j}(a)\tau(\chi_j)\sum_{r\mid (q/\ell_j)} \frac{f(r)}{r}\cdot \frac{1}{\phi(q/r)}  \mu\left(\frac{q}{r\ell_j}\right)\chi_j\left(\frac{q}{r\ell_j}\right)
\sum_{\substack{q^3/r\leq n \leq x/r\\ (n, q/r)=1}} \frac{f(n)\overline{\chi_j}(n)}{n}.
\end{equation}
Furthermore, it follows from Lemma 4.4 of \cite{GS07} that
\begin{align*}
\sum_{\substack{q^3/r\leq n \leq x/r\\ (n, q/r)=1}} \frac{f(n)\overline{\chi_j}(n)}{n} &= \sum_{\substack{n \leq x\\ (n, q/r)=1}} \frac{f(n)\overline{\chi_j}(n)}{n} +O(\log q)\\
&= \prod_{p \mid  \frac qr} \left(1-\frac{f(p)\overline{\chi_j}(p)}{p}\right) \sum_{n \leq x}  \frac{f(n)\overline{\chi_j}(n)}{n} + O\left(\log q\right).
\end{align*}
Hence, we deduce that the expression in \eqref{Eq:ContributionInduced} becomes
\begin{align*}
&\overline{\chi_j}(a)\tau(\chi_j)\sum_{r\mid (q/\ell_j)} \frac{f(r)}{r}\cdot \frac{1}{\phi(q/r)}  \mu\left(\frac{q}{r\ell_j}\right)\chi_j\left(\frac{q}{r\ell_j}\right)
\prod_{p \mid  \frac qr} \left(1-\frac{f(p)\overline{\chi_j}(p)}{p}\right) \sum_{n \leq x}  \frac{f(n)\overline{\chi_j}(n)}{n}\\
& \quad \quad \quad \quad  \quad \quad  \quad \quad + E_5, 
\end{align*}
where 
\begin{align*}
E_5 &\ll  \sqrt{\ell_j}  \log q\sum_{\substack{r\mid (q/\ell_j)\\ (q/(r\ell_j), \ell_j)=1}} \frac{1}{r\phi(q/r)}  \mu^2\left(\frac{q}{r\ell_j}\right) \\
& =  \sqrt{\ell_j}  \log q  \frac{\ell_j}{q \phi( \ell_j)}   \prod_{p|q/\ell_j, p\nmid \ell_j} 2 \bigg( 1 + \frac {1}{2p-2}\bigg)  \ll \frac{ \sqrt{\ell_j} \log q}{\phi(\ell_j)}. 
\end{align*}
Finally, by \cite[Eq. (6.6)]{GS07} we have 
\begin{align*} &\sum_{r\mid (q/\ell_j)} \frac{f(r)}{r}\cdot \frac{1}{\phi(q/r)}  \mu\left(\frac{q}{r\ell_j}\right)\chi_j\left(\frac{q}{r\ell_j}\right)
\prod_{p \mid  \frac qr} \left(1-\frac{f(p)\overline{\chi_j}(p)}{p}\right)\\
&= \frac{1}{\phi(q)}\prod_{\substack{p^{\alpha}\parallel q/\ell_j\\ \alpha\geq 1 }} \left(f(p^{\alpha})-\chi_j(p) f(p^{\alpha-1})\right)= \frac{\kappa_j(q/\ell_j)}{\phi(q)}. 
\end{align*}
Collecting the above estimates  and noting that $\sum_{n\leq q^3} f(n) e(na/q)/n\ll \log q$ completes the proof. \hfill \qed

%%%%%%%%%%%%%%
\subsection{Bounding the contribution of various characters} \label{sec: VariousCharacters}

We now prove  Corollary \ref{Cor:Smallq3} for completely multiplicative $f$:

\begin{corollary} \label{Cor:Smallq2} Let $x$ be large. Suppose that  $\alpha$ lies in a major arc given by 
 \eqref{eq: Major Arc Range}. 
For any completely multiplicative function $f:\mathbb N\to \mathbb U$ there exists a primitive character $\chi \pmod \ell$, for some integer $\ell \mid q$, such that 
\[
\sum_{n\leq x} \frac{f(n)e(\alpha n) }n=  \frac{\overline{\chi}(a)\kappa(q/\ell)  \tau(\chi)}{\phi(q)}\sum_{n\leq x} \frac{f(n) \overline{\chi} (n)}{n} +O\left(  \frac{(\log x)^{2/3+o(1)}}{ \sqrt{q}}+ e^{O(\sqrt{\log\log x})} \right),
\]
 where  $\kappa$ is  defined by the convolution  $f(n) = (\kappa*\chi)(n)$. %The main term here is $\ll \frac{\log x}{\sqrt{q}}$.
\end{corollary}

\begin{proof}[Proof of Corollary \ref{Cor:Smallq2}]
Take $J=\log\log x$ in Theorem \ref{Thm:Smallq}  (so that $T=(\log x)^{-\frac{1}{\sqrt{J}}}= \exp(-\sqrt{\log\log x})$) and let $\chi=\chi_1$  and $\ell=\ell_1$.

Since
 $\kappa_j(p^{\alpha})= f(p)^{\alpha}-\chi_j(p) f(p)^{\alpha-1}$ for all $\alpha\geq 1$, we have 
$|\kappa_j(q/\ell_j)|\leq 2^{\omega(Q_j)}$ where $Q_j$ is the largest squarefree divisor of $q$ that is coprime to $\ell_j$. 
Therefore
\begin{equation}\label{Eq:MainTermAsymptoticDev}
\begin{aligned}
\bigg| \frac{\overline{\chi}_j(a)\kappa_j(q/\ell_j)   \tau(\chi_j)}{\phi(q)}\sum_{n\leq x} \frac{f(n) \overline{\chi}_j (n)}{n}\bigg| & \leq \frac{2^{\omega(Q_j)}  \sqrt{\ell_j} }{\phi(q)}\bigg| \sum_{n\leq x} \frac{f(n) \overline{\chi}_j (n)}{n}\bigg|
\\
&   \ll q^{-1/2+o(1)}\bigg| \sum_{n\leq x} \frac{f(n) \overline{\chi}_j (n)}{n}\bigg|.
\end{aligned}
\end{equation}
Moreover \eqref{Eq:BaGrSo}  
and \eqref{Eq:BoundRepulsive}  imply that
\[
M\left(f\overline{\chi_j};x, T\right) \geq  \left(\frac{1}{3}+o(1)\right)\log\log x,
\]
for all $j\geq 2$, and so by
 Theorem 1.6 of \cite{LaMa22} we obtain
$$  \sum_{n\leq x} \frac{f(n) \overline{\chi}_j (n)}{n}  \ll (\log x) e^{-M\left(f\overline{\chi_j};x, T\right)} +\frac1T  \ll (\log x)^{2/3+o(1)}. $$    
We use this last result to bound the terms with $1<j<J$, which gives the desired bound.

 %Finally, to bound the main term we use that
%\[
%\bigg| \sum_{n\leq x} \frac{f(n) \overline{\chi} (n)}{n}\bigg|  \leq \sum_{\substack{ n\leq x\\ (n,\ell)=1}} \frac{1}{n}
%\lesssim   \frac{\phi(\ell)}{\ell} \log x.
%\]
%Inserting this estimate in the first inequality of \eqref{Eq:MainTermAsymptoticDev} and letting $Q=Q_1$ gives 
%\begin{equation}\label{Eq:correctbd}
%\bigg| \frac{\overline{\chi}(a)\kappa(q/\ell)   \tau(\chi)}{\phi(q)}\sum_{n\leq x} \frac{f(n) \overline{\chi} (n)}{n}\bigg| \leq  \frac{\log x}{\sqrt{q}}    \frac{ 1}  { \sqrt{q/(Q\ell)}}
%\prod_{p|Q} \frac {2p^{1/2}}{p-1} \ll \frac{\log x}{\sqrt{q}},
%\end{equation}
%which completes the proof.
\end{proof}

\subsection{Proof of Corollary \ref{Cor:Smallq3}}
We 
order the primitive characters $\chi\bmod \ell$ for all $\ell\mid q$ as $\{\chi_j \bmod \ell_j\}_{j\geq 1}$, so that 
\begin{equation}\label{Eq:OrderCharCor1.3}
M(f\overline{\chi_j}, x; T) \leq M(f\overline{\chi_{j+1}};x, T) \text{ for all } j\geq 1,
\end{equation}
where  $T=\exp(-\sqrt{\log\log x})$. We also let $\chi=\chi_1$ and $\ell=\ell_1$. Let $g$ be the completely multiplicative function (that is, $g\in \mathcal F$) for which $g(p)=f(p)$ for all primes $p$ and write $f=g*h$ so that 
$h(p^j)=f(p^j)-f(p)f(p^{j-1})$ for all $j\geq 1$, which implies $h(p)=0$ and $|h(p^j)|\leq 2$ for all $j\geq 2$. In particular, we have 
\begin{equation}
\label{Eq:BoundSumHConvolution} 
\sum_{k\geq 1}\frac{|h(k)|}{k^{u}}\leq  \prod_{p} \left(1+\frac{2}{p^{u} (p^{u}-1)}\right)\ll_{u} 1,
\end{equation}
for every $u>1/2.$
Now
%Let $B= (\log x)^3$. By \eqref{Eq:BoundSumHConvolution} we have
%\[
%\bigg|  \sum_{k\geq B} \frac{h(k)}{k} \sum_{ m\leq x/k} \frac{g(m)}{m}  e\left(\frac{mka}{q}\right) \bigg| \ll 
 %\log x \sum_{k\geq B} \frac{|h(k)|}{k} \leq B^{-1/3} \log x \sum_{k\geq 1} \frac{|h(k)|}{k^{2/3}} \ll 1.
%\]
%Inserting this bound in \eqref{Eq:FirstExpressionConvolution} and using \eqref{Eq:BoundSumHConvolution} we get
\begin{align}\label{Eq:ExpressionMultiplicativef}
\sum_{n\leq x}  \frac{f(n)}{n} e(n\alpha) &= \sum_{n\leq x}  \frac{f(n)}{n} e\left(\frac{na}{q}\right) + O(1) \nonumber\\
& = \sum_{k\leq x} \frac{h(k)}{k} \sum_{m\leq x/k} \frac{g(m)}{m}  e\left(\frac{mka}{q}\right)+ O(1)\nonumber\\
&= \sum_{r\mid q} \sum_{\substack{b\leq x/r\\ (b, q/r)=1}} \frac{h(br)}{br}\sum_{m\leq x/(br)} \frac{g(m)}{m}  e\left(\frac{mba}{q/r}\right)+O(1)\\
&= \sum_{r\mid q} \sum_{\substack{b\geq 1\\ (b, q/r)=1}} \frac{h(br)}{br}\sum_{m\leq x} \frac{g(m)}{m}  e\left(\frac{mba}{q/r}\right)+O(1),\nonumber
\end{align}
upon writing $r=(k, q)$ and $b= k/r$ and using that 
\begin{equation}\label{Eq:BoundSumHConvolution2}
\sum_{r|q}\sum_{\substack{b\geq 1\\ (b, q/r)=1}} \frac{|h(br)|}{br}\log(br)=\sum_{k\geq 1} \frac{|h(k)|}{k}\log k\ll 1,
\end{equation} 
by \eqref{Eq:BoundSumHConvolution}.  
Furthermore, by  Corollary \ref{Cor:Smallq2} for $g\in  \mathcal F$ there exists a primitive character $\psi_k\pmod{q_k}$ such that $q_k\mid q/r$ and
\begin{align}\label{eq:ContribCharCor1.3}
   \sum_{m\leq x} \frac{g(m)}{m}  e\left(\frac{mba}{q/r}\right)&= \frac{\overline{\psi_{k}}(ba)\kappa_g(q/rq_k) \tau(\psi_k)}{\phi(q/r)}\sum_{m\leq x} \frac{g(m)\overline{\psi_k}(m)}{m} \nonumber\\& \quad \quad +O\left(\frac{ (\log x)^{2/3+o(1)}}{ \sqrt{q/r}}+e^{O(\sqrt{\log\log x})} \right),
\end{align}
 where  $\kappa_g=\kappa_{g,\psi_k}$ is  defined by the convolution  $g(n) = (\kappa_g*\psi_k)(n)$. Moreover, it follows from \eqref{Eq:OrderCharCor1.3} and the proof of Corollary \ref{Cor:Smallq2} that $\psi_k=\chi$ if $\ell\mid q/r$. On the other hand, if $\ell\nmid q/r$ then $\psi_k\neq \chi$, and hence by \eqref{Eq:BaGrSo} and \eqref{Eq:BoundRepulsive} we obtain 
$$M(g\overline{\psi_k}, x;T)= M(f\overline{\psi_k}, x;T)\geq \left(\frac13+o(1)\right)\log\log x.$$
Therefore, Theorem 1.6 of \cite{LaMa22} implies that in this case
$$ 
\sum_{m\leq x}\frac{g(m){\overline{\psi_k}}(m)}{m}\ll (\log x)e^{-M(g\overline{\psi_k}, x; T)} +\frac{1}{T}\ll (\log x)^{2/3+o(1)}.$$
Inserting this estimate in \eqref{eq:ContribCharCor1.3} and using \eqref{Eq:MainTermAsymptoticDev} we deduce that if $\ell\nmid q/r$ then 
$$\sum_{m\leq x} \frac{g(m)}{m}  e\left(\frac{mba}{q/r}\right) \ll  \frac{ (\log x)^{2/3+o(1)}}{ \sqrt{q/r}}+e^{O(\sqrt{\log\log x})}.$$
Combining this bound with \eqref{eq:ContribCharCor1.3} when $\ell\mid q/r$ (and hence $\psi_k=\chi$ in this case), and inserting both in \eqref{Eq:ExpressionMultiplicativef} gives 
\begin{equation}\label{Eq:MainExpressionMultiplicative}
\sum_{n\leq x}  \frac{f(n)}{n} e(n\alpha)= \tau(\chi)\overline{\chi}(a)\sum_{r\mid q/\ell} \sum_{\substack{b\geq 1\\ (b, q/r)=1}} \frac{h(br)}{br}\frac{\overline{\chi}(b)\kappa_g(q/r\ell) }{\phi(q/r)}\sum_{m\leq x/br} \frac{g(m)\overline{\chi}(m)}{m}+E_6,
\end{equation}
where 
\begin{align}\label{Eq:BoundErrorE2}
E_6 &\ll \sum_{r\mid q} \sum_{\substack{b\geq 1\\ (b, q/r)=1}} \frac{|h(br)|}{br}\left(\frac{ (\log x)^{2/3+o(1)}}{ \sqrt{q/r}}+\log(br)+e^{O(\sqrt{\log\log x})}\right)\nonumber\\
&\ll\frac{ (\log x)^{2/3+o(1)}}{ \sqrt{q}}\sum_{r\mid q} \sum_{b\geq 1} \frac{|h(br)|}{b\sqrt{r}} +e^{O(\sqrt{\log\log x})},
\end{align}
by \eqref{Eq:BoundSumHConvolution2}. Note that $h(n)\neq 0$ only when $p| n\implies p^2 | n$ and in this case we have $|h(n)|\leq d(n)$. This implies
\begin{align}
 \sum_{r\mid q} \sum_{b\geq 1} \frac{|h(br)|}{b\sqrt{r}} &\leq \mathop{\sum_{r\mid q} \sum_{b\geq 1}}_{p|(br)\implies p^2| (br)} \frac{d(br)}{b\sqrt{r}} \notag \\
&=\prod_{p\nmid q} \bigg( 1+\sum_{i\geq 2} \frac{i+1}{p^i}\bigg)\cdot
\prod_{p^e\| q, e\geq 1} \bigg( 1+\sum_{\substack{ i\leq e, j\geq 0\\ i+j\geq 2}} \frac{i+j+1}{p^{j+i/2}}\bigg)
\ll (q/\phi(q))^3, \label{eq: b root r}
\end{align}
the main term coming from the case $i=2,j=0$ for each prime $p$ dividing $q$.  
 Inserting this bound in \eqref{Eq:BoundErrorE2} yields
$$
E_6 \ll\frac{ (\log x)^{2/3+o(1)}}{ \sqrt{q}} +e^{O(\sqrt{\log\log x})}.
$$
We now estimate the main term on the right hand side of \eqref{Eq:MainExpressionMultiplicative}, which equals
%\begin{equation}\label{Eq:MainTermSumFMultiplicative}
%\tau(\chi)\overline{\chi}(a)\sum_{r\mid q/\ell}   \frac{\kappa_g(q/r\ell)}{r\phi(q/r)} \sum_{\substack{b\geq 1\\ (b, q/r)=1}} \frac{\overline{\chi}(b)h(br)}{b}\sum_{m\leq x/br} \frac{g(m)\overline{\chi}(m)}{m}.
%\end{equation} 
\begin{equation}\label{Eq:MainTermSumFMultiplicative}
\frac{\tau(\chi)\overline{\chi}(a)}{\phi(q)} \sum_{r\mid q/\ell} \kappa_g(q/r\ell) \prod_{\substack{p\mid q\\ p\nmid (q/r)}}\left(1-\frac{1}{p}\right) \sum_{\substack{b\geq 1\\ (b, q/r)=1}} \frac{\overline{\chi}(b)h(br) }{  b} \sum_{m\leq x/br} \frac{g(m)\overline{\chi}(m)}{m},
\end{equation} 
since
$$ \phi(q/r)r= \phi(q)\prod_{\substack{p\mid q\\ p\nmid (q/r)}}\left(1-\frac{1}{p}\right)^{-1}.$$
We now write $m=m_1m_2$ and $b=b_1b_2$ where $(m_1, q)=(b_1, q)=1$ and $p\mid m_2b_2 \implies p\mid q$. Therefore, the inner double sum over $b$ and $m$ in \eqref{Eq:MainTermSumFMultiplicative} equals:
\begin{align*}
&\sum_{\substack{b_2\geq 1\\ (b_2, q/r)=1 \\ p|b_2\implies p|q}} \frac{\overline{\chi}(b_2)h(b_2r) }{b_2} \sum_{\substack{m_2 \geq 1\\ p|m_2 \implies p|q}} \frac{g(m_2)\overline{\chi}(m_2)}{m_2} \sum_{\substack{b_1\geq 1\\ (b_1, q)=1}} \frac{\overline{\chi}(b_1)h(b_1) }{b_1}\sum_{\substack{m_1\leq \frac{x}{b_1b_2m_2r}\\(m_1, q)=1}} \frac{g(m_1)\overline{\chi}(m_1)}{m_1}\\
&= \prod_{p|q}\left(1-\frac{f(p)\overline{\chi}(p)}{p}\right)^{-1}\sum_{\substack{n\leq x\\ (n, q)=1}} \frac{f(n)\overline{\chi}(n)}{n}\sum_{\substack{b_2\geq 1\\ (b_2, q/r)=1 \\ p|b_2\implies p|q}} \frac{\overline{\chi}(b_2)h(b_2r) }{b_2}  \\
& \quad \quad \quad + O\Bigg(\log(2r)\sum_{\substack{b_2\geq 1\\ (b_2, q/r)=1 \\ p|b_2\implies p|q}} \frac{d(b_2r)\log(2b_2)}{b_2} \sum_{\substack{m_2\geq 1\\ p|m_2 \implies p|q}} \frac{\log(2m_2)}{m_2} \Bigg),
\end{align*}
since $\sum_{p|m_2 \implies p|q} g(m_2)\overline{\chi}(m_2)/m_2 = \prod_{p|q}\left(1-f(p)\overline{\chi}(p)/p\right)^{-1}$ and
\begin{align*}
\sum_{\substack{b_1\geq 1\\ (b_1, q)=1}} \frac{\overline{\chi}(b_1)h(b_1) }{b_1}\sum_{\substack{m_1\leq x/(b_1b_2m_2r)\\ (m_1, q)=1}} \frac{g(m_1)\overline{\chi}(m_1)}{m_1}
&= \sum_{\substack{n\leq x/(b_2m_2 r)\\ (n, q)=1}} \frac{f(n)\overline{\chi}(n)}{n}\\
&= \sum_{\substack{n\leq x\\ (n, q)=1}} \frac{f(n)\overline{\chi}(n)}{n} +O(\log(b_2m_2r)),
\end{align*}
upon writing $n=b_1m_1$ and using that $f=g*h$. Using this estimate we see that \eqref{Eq:MainTermSumFMultiplicative} becomes 
\begin{equation}\label{Eq:MainTermSumFMultiplicative2}
\frac{\tau(\chi)\overline{\chi}(a)}{\phi(q)} F_f(q, \ell) \sum_{\substack{n\leq x\\ (n, q)=1}} \frac{f(n)\overline{\chi}(n)}{n}  +E_7
\end{equation}
where 
$$ F_f(q, \ell)= \prod_{p|q}\left(1-\frac{f(p)\overline{\chi}(p)}{p}\right)^{-1} \sum_{r\mid q/\ell} \kappa_g(q/r\ell) \prod_{\substack{p\mid q\\ p\nmid (q/r)}}\left(1-\frac{1}{p}\right) \sum_{\substack{b_2\geq 1\\ (b_2, q/r)=1 \\ p|b_2\implies p|q}} \frac{\overline{\chi}(b_2)h(b_2r) }{b_2} ,
$$
is an Euler product that we shall evaluate, and where the error term satisfies
\begin{align*}
E_7&\ll \frac{\sqrt{\ell}}{\phi(q)}\sum_{\substack{m_2\geq 1\\ p|m_2 \implies p|q}} \frac{\log(2m_2)}{m_2}\sum_{r\mid q/\ell} d(r)|\kappa_g(q/r\ell)| \log(2 r)  \sum_{\substack{b_2\geq 1\\ p|b_2\implies p|q}} \frac{d(b_2)\log(2b_2)}{b_2 }\nonumber \\
& \ll q^{-1/2+o(1)} \Bigg(\sum_{\substack{b\geq 1\\ p|b\implies p|q}} \frac{d(b)\log(2b)}{b}\Bigg)^2=q^{-1/2+o(1)},
\end{align*}
since $d(r)|\kappa_g(q/r\ell)| \log (2r)= q^{o(1)}$ for all $r|q$, 
and  
$$ \sum_{\substack{b\geq 1\\ p|b\implies p|q}} \frac{d(b)\log (2b)}{b} \ll \Big(\sum_{p|q} \frac{\log p}{p}\Big)\prod_{p|q}\left(1-\frac1p\right)^{-2} \ll (\log\log q)^3.$$
%$|h(b_2r_)|\leq \mathbb_{1}_{p|b_2\implies p^2|b_2} d(b_2)$  
%We first bound $E_3$ before evaluating the term $F_f(q, \ell)$.
%We are going to split the summands into up to three parts as follows:
To calculate $F_f(q, \ell)$ we write $q/\ell = q_1q_2$ where $p|q_1\implies p|\ell$ and $p|q_2\implies p\nmid\ell$. Note that for all $r| q/\ell$, we have $p\nmid q/r\implies p|q_2$. Moreover, since $\chi(n)=0$ if $(n, \ell)>1$, we can impose the condition $p|b_2 \implies p|q_2$ for all $b_2$ in the inner sum. Therefore, we get
\begin{align}\label{Eq:CalculEulerProduct}
F_f(q, \ell)&= \prod_{p|q_2}\left(1-\frac{f(p)\overline{\chi}(p)}{p}\right)^{-1}\sum_{r_1\mid q_1} h(r_1)\kappa_g(q_1/r_1) 
\nonumber \\
&\quad \quad \quad \times\sum_{r_2\mid q_2} \kappa_g(q_2/r_2) \prod_{\substack{p\mid q_2\\ p\nmid q_2/r_2}}\left(1-\frac{1}{p}\right) \sum_{\substack{b_2\geq 1\\ (b_2, q_2/r_2)=1 \\ p|b_2\implies p|q_2}} \frac{\overline{\chi}(b_2)h(b_2r_2) }{b_2}.
\end{align}
Now 
\begin{equation}\label{Eq:ProductPrimesq1}
 \sum_{r_1|q_1 } h(r_1 ) \kappa_g(q_1/r_1 )=  \sum_{r_1|q_1 } h( r_1 )  g(q_1/r_1 )=f(q_1).
\end{equation}
On the other hand, the inner sum over $r_2$ is multiplicative, so we might break it down into the prime powers dividing $q_2$. Hence if $p^e\| q_2$ then 
$p^a\| r_2$ for some $0\leq a\leq e$, and $p\nmid b_2$ unless $a=e$. Thus, we obtain
\begin{align} \label{Eq:ProductPrimesq2}
&\sum_{r_2\mid q_2} \kappa_g(q_2/r_2) \prod_{\substack{p\mid q_2\\ p\nmid q_2/r_2}}\left(1-\frac{1}{p}\right) \sum_{\substack{b_2\geq 1\\ (b_2, q_2/r_2)=1 \\ p|b_2\implies p|q_2}} \frac{\overline{\chi}(b_2)h(b_2r_2) }{b_2}\nonumber\\
&= \prod_{p^e\|q_2}\left(\sum_{a=0}^{e-1}h(p^a)  \kappa_g(p^{e-a}) +\left(1-\frac{1}{p}\right)
 \sum_{j=0}^{\infty}  \frac{\overline{\chi}(p^j)h(p^{e+j})  }{p^j}\right).
\end{align}
 Now
\begin{align*}
\sum_{a=0}^{e-1}h(p^a)  \kappa_g(p^{e-a})&= h* \kappa_g(p^e)-h(p^e)\\
&= f(p^e)-\chi(p)f(p^{e-1})-h(p^e)= (f(p)-\chi(p))f(p^{e-1}),%-\chi(p)f(p^{e-1}),
\end{align*}
 since $f= \chi*(h*\kappa_g)$. Moreover, we have
\begin{align*}
 \sum_{j=0}^{\infty}  \frac{\overline{\chi}(p^j)h(p^{e+j})}{p^j}&= h(p^e)+\sum_{j=1}^{\infty}  \frac{\overline{\chi}(p^j)f(p^{e+j})}{p^j}- f(p)\sum_{j=1}^{\infty}  \frac{\overline{\chi}(p^j)f(p^{e+j-1})}{p^j}\\
 & = h(p^e)- f(p^e)+\sum_{j=0}^{\infty}  \frac{\overline{\chi}(p^j)f(p^{e+j})}{p^j}-\frac{f(p)\overline{\chi}(p)}{p}\sum_{j=0}^{\infty}  \frac{\overline{\chi}(p^j)f(p^{e+j})}{p^j} \\
 & = -f(p)f(p^{e-1})+ \left(1-\frac{f(p)\overline{\chi}(p)}{p}\right)\sum_{j=0}^{\infty}  \frac{\overline{\chi}(p^j)f(p^{e+j})}{p^j}. 
 \end{align*}
 Thus, we deduce that
\begin{align*}&\sum_{a=0}^{e-1}h(p^a)  \kappa_g(p^{e-a}) +\left(1-\frac{1}{p}\right)
 \sum_{j=0}^{\infty}  \frac{\overline{\chi}(p^j)h(p^{e+j})  }{p^j}\\
 &= \frac{f(p)f(p^{e-1})}{p} - \chi(p)f(p^{e-1}) +\left(1-\frac{1}{p}\right) \left(1-\frac{f(p)\overline{\chi}(p)}{p}\right)\sum_{j=0}^{\infty}  \frac{\overline{\chi}(p^j)f(p^{e+j})}{p^j}\\
 &= \left(1-\frac{f(p)\overline{\chi}(p)}{p}\right) \left(-\chi(p)f(p^{e-1})+ \left(1-\frac{1}{p}\right)\sum_{j=0}^{\infty}  \frac{\overline{\chi}(p^j)f(p^{e+j})}{p^j}\right)\\
 & =\left(1-\frac{f(p)\overline{\chi}(p)}{p}\right) \left(f(p^e)-\chi(p)f(p^{e-1})+\sum_{j=1}^{\infty}  \frac{\overline{\chi}(p^j)f(p^{e+j})}{p^j}- \sum_{j=1}^{\infty}  \frac{\overline{\chi}(p^{j-1})f(p^{e+j-1})}{p^j}\right)\\
 &= \left(1-\frac{f(p)\overline{\chi}(p)}{p}\right) \left(\sum_{j=0}^{\infty} \frac{\overline{\chi}(p^j)\kappa_f(p^{e+j})}{p^j}\right).
 \end{align*}
Combining this identity with \eqref{Eq:CalculEulerProduct}, \eqref{Eq:ProductPrimesq1} and \eqref{Eq:ProductPrimesq2} implies that 
\begin{align*} 
F_f(q, \ell)&= f(q_1)\prod_{p^e\|q_2}\left(\sum_{j=0}^{\infty} \frac{\overline{\chi}(p^j)\kappa_f(p^{e+j})}{p^j}\right)= \prod_{p^e\|q/\ell}\left(\sum_{j=0}^{\infty} \frac{\overline{\chi}(p^j)\kappa_f(p^{e+j})}{p^j}\right)\\
& =\sum_{\substack{m\geq 1 \\ p|m\implies p|q}}  \frac{\overline{\chi}(m)\kappa_f(m\cdot q/\ell)}{m} .
\end{align*}
Therefore we deduce that 
\begin{align*} 
 \sum_{n\leq x} \frac{f(n)e(\alpha n) }n&=  \frac{\overline{\chi}(a)\tau(\chi)}{\phi(q)} 
  \sum_{\substack{m\geq 1 \\ p|m\implies p|q}}  \frac{\kappa_f(m\frac q\ell)\overline{\chi}(m)}{m} 
   \sum_{\substack{n\leq x\\ (n, q)=1}} \frac{f(n)\overline{\chi}(n)}{n} \\
   &\quad \quad \quad \quad   +O\left(\frac{(\log x)^{2/3+o(1)}}{ \sqrt{q}}+ e^{O(\sqrt{\log\log x})} \right).
\end{align*}
Now for each $m$ we truncate the sum over $n$ at $x/m$ producing an error term $\ll \log m$ and so, in total an error of
\[
\ll  \frac{\sqrt{\ell}\, 2^{\omega(q_2)}}{\phi(q)} 
  \sum_{\substack{m\geq 1 \\ p|m\implies p|q_2}}  \frac{\log m}{m} 
 \ll  \frac{\sqrt{\ell}\, 2^{\omega(q_2)}}{ \phi(\ell q_1) \phi(q_2)} \frac{q_2}{\phi(q_2)} \sum_{ p|q_2} \frac{\log p}{p-1}  \ll 1
  \]
since $|\kappa_f(m\frac q\ell)|\leq 2^{\omega(q_2)}$. The first result follows since the sum in the main term is now over $mn\le x$.

Now $|\tau(\chi)|=\sqrt{\ell}$ and $|\sum_{n\leq x,\ (n, q)=1} \frac{f(n)\overline{\chi}(n)}{n}|\lesssim (\phi(q)/q)\log x$.  Furthermore we have 
\[
\sum_{j=0}^{\infty} \frac{\overline{\chi}(p^j)\kappa_f(p^{e+j})}{p^j} =x
\bigg( 1-\frac 1p\bigg) \sum_{j=0}^{\infty} \frac{\overline{\chi}(p^j) f(p^{e+j})}{p^j}   -\chi(p) f(p^{e -1}) 
\]
as $\kappa_f(p^{e+j})=f(p^{e+j})-\chi(p)f(p^{e+j-1})$, and  so the absolute value  of our sum is
$\leq 1+(1-\frac 1p) \sum_{j\geq 0}  \frac{1 }{p^j}=2$.
Therefore the main term is, in absolute value,
\begin{equation}\label{Eq:OptimalMainTermLog}
 \lesssim   \frac{\log x}{\sqrt{q}} \cdot   \frac {2^{\omega(q/\ell)}}{\sqrt{q/\ell}}.
\end{equation}
Now $2/p^{e/2}\leq1$ unless  $e=1$ and $p\in \{2,3\}$.   Hence, the term on \eqref{Eq:OptimalMainTermLog} %in the worst case we get 
is $ \lesssim  \frac 4{\sqrt{6}} \frac{\log x}{\sqrt{q}}$. Moreover, we obtain $\sim \frac 4{\sqrt{6}} \frac{\log x}{\sqrt{q}}$ only when $q=6\ell$ with $(6,\ell)=1$, and for $p\in \{2,3\}$ 
$$ 
\bigg|\bigg( 1-\frac 1p\bigg) \sum_{j=0}^{\infty} \frac{\overline{\chi}(p^j) f(p^{1+j})}{p^j}   -\chi(p)\bigg|= 2, 
$$ 
in which case we must have $  f(p^{ j})=-\chi(p)^j   $ for all $j\geq 1$,
so $\kappa_f(p^j)=0$ for $j\geq 2$ with $\kappa_f(p)=-2\chi(p)$. The final result follows.

%%%%%%%%%%%%%%
\subsection{A Hal\'asz-type upper bound}

We will use Theorem \ref{Thm:PretentiousLSieve} to deduce a bound on our logarithmic sums:
 
\begin{corollary}\label{Cor:Smallq}
 Let $x$ be large with $T\in [ \frac 2{\log x}, 1]$ and  $f\in \F$. If  $|\alpha-a/q|\leq 1/(qx)$ with $(a, q)=1$ and $q\leq (\log x)^{10}$  then
\[ \sum_{n\leq x} \frac{f(n)e(n\alpha)}{n} \ll 
\bigg( \log x    \exp\ (-M(f,q,x, T))   +  \frac{1}{T} \bigg) (\log\log x)^{3+o(1)},
\]
where we define
$$ 
M(f,q,x, T):= \min_{\substack{ |t|\leq T \\ \chi \bmod q  }}  \sum_{\substack{p\leq x \\ p\nmid q}}\frac{1-\re(f(p)\overline{\chi}(p)p^{-it})}{p}.
$$ 
\end{corollary}

\begin{proof}
We start by bounding $\sum_{n\leq q^3} f(n)e(n\alpha)/n$ trivially, replacing $\alpha$ by $a/q$ (at the expense of an error term of size $O(1)$) and using  \eqref{Eq:IdentityCharacters}  to obtain
\begin{equation}
\label{Eq:IdExpSumSmallq}   \sum_{n\leq x} \frac{f(n)e(n\alpha)}{n}= \sum_{r\mid q}\frac{f(r)}{r}\frac{1}{\phi(q/r)}\sum_{\chi \bmod q/r}%\mathcal{G}
\tau(\chi) \overline{\chi}(a)\sum_{q^3/r\leq n\leq x/r} \frac{f(n) \overline{\chi}(n)}{n}+O(\log q).
\end{equation}
Proceeding similarly to 
\eqref{Eq:BadCharPretentious} gives 
\begin{equation}\label{Eq:BoundExpSumSmall2}
\sum_{n\leq x} \frac{f(n)e(n\alpha)}{n}\ll  \sum_{r\mid q} \frac{\sqrt{q/r}}{r\sqrt{\phi(q/r)}}\Bigg(\sum_{\chi \bmod q/r} \Bigg|\sum_{q^3/r\leq n\leq x/r} \frac{f(n) \overline{\chi}(n)}{n}\Bigg|^2\Bigg)^{1/2}+\log q.
\end{equation}
We now use \eqref{eq: PLS log2}
% Theorem \ref{Thm:PretentiousLSieve}  
with $\mathcal X = \emptyset$ to get
\begin{align*}
&\bigg( \sum_{\chi \bmod q/r} \Big|\sum_{q^3/r \leq n\leq x/r}\frac{f(n)\overline{\chi}(n)}{n}\Big|^2\bigg)^{1/2} \\
& \ll  (\log x)(\log\log x)\log(T\log x)^2   \exp\ \left(-\min_{\chi \bmod q/r} M(f\overline{\chi}; x/r, T)\right) \\
&    \quad \quad \quad + \frac{(\log\log x)\log(T\log x)}{T}.
\end{align*}
Note that 
\begin{align*}
\min_{\chi \bmod q/r} M(f\overline{\chi}; x/r, T)&\geq \min_{\chi \bmod q} M(f\overline{\chi}; x/r, T)\\
&= M(f, q, x, T)+ \sum_{p\mid q}\frac1p+ O\left(\frac{\log r}{\log x}\right).
\end{align*} Inserting these estimates in \eqref{Eq:BoundExpSumSmall2} and using \eqref{Eq:BoundFirstCoefficientSumSmall} imply that 
\begin{align*}\sum_{n\leq x} \frac{f(n)e(n\alpha)}{n}& \ll \left(\frac {q}{\phi(q)}\right)^{1/2} (\log x)(\log\log x)\log(T\log x)^2   \exp\ (-M(f,q,x, T)) \\
   & \quad \quad \quad  + \left(\frac {q}{\phi(q)}\right)^{3/2}\frac{(\log\log x)\log(T\log x)}{T},
\end{align*}
which is slightly better (but more complicated) than the given result, which is obtained by using that $T\log x\leq \log x$ and $q\leq (\log x)^{10}$ so that $\frac {q}{\phi(q)}\ll \log\log\log x$.
\end{proof}
%%%%%%%%%%%%%%%%%%%%%%%%%%%%%%%%%%%%%%%%%%%%%%%%%%%%%%%%%%%%%%%%%%%%%%%%%%%%%%%%%%%%%%%%%%
\subsection{Classifying the functions $f\in \F$ for which  \eqref{Eq:MVThree} is attained: Proof of Theorem \ref{Thm:Classify_Large}}
We first record the following bound for logarithmic mean values of  multiplicative functions, which was proved in \cite{GM23}.

 \begin{lemma}[Proposition 1.2 of \cite{GM23}]\label{Lem:LogSumDistance}
Let  $f:\mathbb{N}\to \mathbb{U}$ be a multiplicative function and $x\geq 2$ be a real number. Then we have 
$$ \sum_{n\leq x}\frac{f(n)}{n}\ll \log\log x+ (\log x)(1+ \mathbb{D}(f, 1, x)^2)\exp\left(-\lambda \mathbb{D}(f, 1, x)^2\right)$$    
where $\lambda=0.8221\ldots$ is given by $\int_0^1|e(\theta)-\lambda|d\theta=2-\lambda$.
\end{lemma}

% \begin{lemma}[Lemma 4.3 of \cite{GS07}]\label{Lem:LogSumDistance}
% Let $f\in \F$ and $x\geq 2$ be a real number. Then we have 
% $$ \sum_{n\leq x}\frac{f(n)}{n}\ll 1+ (\log x)\exp\left(-\frac12 \mathbb{D}(f, 1, x)^2\right).$$    
% \end{lemma}

We will also need the following lemma from \cite{Go12}, which is a generalization of Lemma 6.2 of \cite{GS07}.
\begin{lemma}
\label{lem:Goldmakher}
Let $f \in \F$ and $\alpha \in [0,1]$. Let $x$ be large and $M$ be a real number such that $
(\log x)^5\leq M\leq x$. Suppose that $
\left|\alpha - b/r\right| \leq 1/(rM)
$
where $b, r$ are coprime integers such that $1\leq b\leq r\leq M$. 
Set $N = \min\left\{x, \frac{1}{|r\alpha - b|}\right\}$. Then we have
$$
\sum_{n \leq x} \frac{f(n)}{n} \, e(n\alpha) =
\sum_{n \leq N}\frac{f(n)}{n} \, e\left(\frac{bn}{r}\right) +
O\left(\log \log x \right),
$$
where the implied constant in the error term is absolute.
\end{lemma}
\begin{proof}
This is a consequence of Lemma 4.1 of \cite{Go12}.
\end{proof}

\subsection*{Proof of Theorem \ref{Thm:Classify_Large} for $f\in \F$}
% \begin{proof}[Proof of Theorem \ref{Thm:Classify_Large} for $f\in \F$]
Throughout let $c=1/B\in (0,1)$ so that $|\sum_{ n\leq x} \frac{f(n)}{n} e(n\alpha)|\geq c\log q$.

 Let $R= q^{c/3}$. We first approximate $\alpha$  by a rational $b/r$ with $(b, r)=1$ and $r\leq R$, such that $|\alpha-b/r|\leq 1/(rR)$. Note that $b/r\neq a/q$ since $R<q$. We will consider three cases depending on the size of $q$. The first two correspond to part (a) of the theorem, while the last one corresponds to part (b).

{\bf Part (a) } for  $(\log x)^A\leq q\leq  x^{1/3}$, where $A>0$ is a suitably large constant in terms of $c$. Using   Corollary \ref{cor: Cut sum} with $y=q^2$ and $M=3$ we obtain
\begin{align}\label{Eq:Case1Approx}
\sum_{n\leq x} \frac{f(n)}{n} e(n\alpha)&=  
 \sum_{n\leq q^2}  \frac{f(n)}{n} e(n\alpha) +  \sum_{m\leq 3} \frac{f(m)}{m}\sum_{q^2< p\leq x/m} \frac{f(p)}{p} e(mp\alpha) +
 O( \log\log x) \nonumber\\
&= \sum_{n\leq q^2}  \frac{f(n)}{n} e(n\alpha) +  O\left(\frac{\log q}{A}\right).
\end{align}
Since $\left|\sum_{ n\leq x} f(n) e(n\alpha)/n\right|\geq c \log q$ by our assumption,  we deduce that if $A$ is suitably large in terms of $c$, then
\begin{equation}
\label{Eq:LargeShortSum0}
 \left|\sum_{n \leq q^2}\frac{f(n)}{n} \, e(n\alpha)\right|\geq \frac{4c}{5} \log q. 
\end{equation}Note that $R=q^{c/3}>(2\log q)^5$, if $x$ is large enough. Hence by Lemma \ref{lem:Goldmakher} we obtain 
$$
\sum_{n \leq q^2} \frac{f(n)}{n} \, e(n\alpha) =
\sum_{n \leq N_1}\frac{f(n)}{n} \, e\left(\frac{bn}{r}\right) +
O\left(\log \log q \right),
$$
where $N_1 = \min\left\{q^2, \frac{1}{|r\alpha - b|}\right\}\geq R$. 
Inserting this estimate in \eqref{Eq:LargeShortSum0} we get  
\begin{equation}
\label{Eq:LargeShortSum1}
\frac{3c}{4} \log q\leq \left|\sum_{n \leq N_1}\frac{f(n)}{n} \, e\left(\frac{bn}{r}\right)\right|\leq \log N_1+O(1). 
\end{equation}
In particular, this implies $N_1\gg q^{3c/4}\geq R^2\geq r^2$. Hence,  by \eqref{Eq:MVLargeIntegers} we obtain
$$
\sum_{r^2\leq n \leq N_1}\frac{f(n)}{n} e\left(\frac{bn}{r}\right) \ll \frac{\log r}{\sqrt{r}}\log N_1+\log\log N_1\ll \frac{\log r}{\sqrt{r}}\log q+ \log\log q.
$$
Combining this bound with \eqref{Eq:LargeShortSum1} we deduce that 
\begin{align*}
\frac{3c}{4}\log q &\leq 2\log r + \bigg|\sum_{r^2\leq n \leq N_1}\frac{f(n)}{n} e\left(\frac{bn}{r}\right)\bigg| +O(1)\\
&\leq \frac{2c}{3}\log q + O\left(\frac{\log r}{\sqrt{r}}\log q+ \log\log q\right),
\end{align*}
since $r\leq R=q^{c/3}$. Thus we must have $r=O_c(1)$. 
We now use Corollary \ref{Cor:Smallq2}, which implies the existence of a primitive character $\chi\pmod \ell$ with $\ell\mid r$ such that 
$$ 
\sum_{n\leq N_1} \frac{f(n)}{n} e\left(\frac{bn}{r}\right)= \frac{\overline{\chi}(b)\kappa(r/\ell)  \tau(\chi)}{\phi(r)}\sum_{n\leq N_1} \frac{f(n) \overline{\chi} (n)}{n} +O\left((\log N_1)^{2/3+o(1)}\right).
$$
By \eqref{Eq:LargeShortSum1} and since $r=O_c(1)$ and $\log N_1\asymp_c \log q$ we deduce that 
$$ \sum_{n\leq N_1} \frac{f(n) \overline{\chi} (n)}{n} \gg_{c} \log q.$$
On the other hand,  by Lemma \ref{Lem:LogSumDistance} we obtain 
\begin{align*} \sum_{n\leq N_1} \frac{f(n) \overline{\chi} (n)}{n} 
&\ll \log\log N_1+ (\log N_1) \exp\left(-\frac45\, \mathbb{D}(f, \chi, N_1)^2\right)\\
&\ll_c  \log\log q+ (\log q) \exp\left(-\frac45\, \mathbb{D}(f, \chi, q)^2\right),
 \end{align*}
since $\log N_1 \asymp_c \log q.$
Combining these two bounds we get 
$$ \mathbb{D}(f, \chi, q) \ll_c 1, $$
as desired.

{\bf Part (a) } for $ x^{1/3}\leq q\leq x$. Since $R>(\log x)^5$ if $x$ is sufficiently large, we may use Lemma \ref{lem:Goldmakher} which gives 
$$
\sum_{n \leq x} \frac{f(n)}{n} \, e(n\alpha) =
\sum_{n \leq N_2}\frac{f(n)}{n} \, e\left(\frac{bn}{r}\right) +
O\left(\log \log x \right),
$$
where $N_2 = \min\left\{x, \frac{1}{|r\alpha - b|}\right\}$. By our assumption we deduce that 
$$\frac{3c}{4}\log q\leq  \Big|\sum_{n \leq N_2}\frac{f(n)}{n} e\left(\frac{bn}{r}\right)\Big| \leq \log N_2 +O(1).
$$
This implies that $N_2\geq R^2\geq r^2$. Therefore, by \eqref{Eq:MVLargeIntegers} we obtain
\begin{align*}
\frac{3c}{4}\log q &\leq 2\log r + \bigg|\sum_{r^2\leq n \leq N_2}\frac{f(n)}{n} e\left(\frac{bn}{r}\right)\bigg| +O(1)\\
&\leq \frac{2c}{3}\log q + O\left(\frac{\log r}{\sqrt{r}}\log x+ \log\log x\right).
\end{align*}
Since $\log q\geq (\log x)/3$ we deduce that $r=O_c(1)$.  Furthermore, by Corollary \ref{Cor:Smallq2} there exists a primitive character $\chi\pmod \ell$ with $\ell\mid r$ such that 
$$ 
\sum_{n\leq N_2} \frac{f(n)}{n} e\left(\frac{bn}{r}\right)= \frac{\overline{\chi}(b)\kappa(r/\ell)  \tau(\chi)}{\phi(r)}\sum_{n\leq N_2} \frac{f(n) \overline{\chi} (n)}{n} +O\left(\log x)^{2/3+o(1)}\right).
$$
Since $r=O_c(1)$ we get 
$$ \sum_{n\leq N_2} \frac{f(n) \overline{\chi} (n)}{n} \gg_{c} \log x.$$
On the other hand, by Lemma \ref{Lem:LogSumDistance} we obtain 
\begin{align*} \sum_{n\leq N_2} \frac{f(n) \overline{\chi} (n)}{n} 
&\ll \log \log N_2+ (\log N_2) \exp\left(-\frac45\, \mathbb{D}(f, \chi, N_2)^2\right)\\
&\ll_c  \log  \log x+ ( \log x) \exp\left(-\frac45\, \mathbb{D}(f, \chi, q)^2\right), 
\end{align*}
since $\log N_2 \asymp_c \log q.$
Thus, we deduce from these two bounds  that 
$$ \mathbb{D}(f, \chi, q) \ll_c 1.$$

{\bf Parts (a) and (b)} for $(\log x)^{2+\varepsilon}\leq q\leq (\log x)^A$. Let $K=K(c)\geq 3$ be a parameter to be chosen. By Corollary \ref{cor: Cut sum} we have 
\begin{align*}
\sum_{ n\leq x} \frac{f(n)}{n} e(n\alpha)&=   
 \sum_{n\leq q^2}  \frac{f(n)}{n} e(n\alpha) + \sum_{\substack{ m\leq K \\ q^2<p\leq x/m}} \frac{f(mp)}{mp} e(mp\alpha)\\
 & \quad \quad \quad +O\bigg(1 +   \bigg(    \frac {\log 2K }{K}   \bigg)^{1/2}\log\log x\bigg). 
\end{align*}
  Therefore, if $K=B/c^3$ for some suitably large constant $B$, then by our assumption we deduce that 
\begin{equation}
\label{Eq:LargeShortSum2}
 \bigg|\sum_{n\leq q^2}  \frac{f(n)}{n} e(n\alpha) + \sum_{\substack{ m\leq K \\ q^2<p\leq x/m}} \frac{f(mp)}{mp} e(mp\alpha)\bigg|\geq \frac{4}{5} c\log q.
 \end{equation}
  Let us first assume that 
\begin{equation}\label{Eq:AssumptionFirstSumLarge}
\bigg|\sum_{n\leq q^2}  \frac{f(n)}{n} e(n\alpha)\bigg|\geq \frac{c}{2}\log q.      
  \end{equation}
  In this case, we are in the same situation as \eqref{Eq:LargeShortSum0}, and a similar argument to the proof of part (a) gives the desired conclusion. Therefore, we might assume that 
  $$ \bigg|\sum_{n\leq q^2}  \frac{f(n)}{n} e(n\alpha)\bigg|\leq \frac{c}{2}\log q.$$
  In this case, \eqref{Eq:LargeShortSum2} implies that 
  $$ \bigg|\sum_{\substack{ m\leq K \\ q^2<p\leq x/m}} \frac{f(mp)}{mp} e(mp\alpha)\bigg|\geq \frac{c}{4} \log q\geq \frac{c}{2} \log\log x.
  $$
  On the other hand we observe that 
 \begin{align*}
\bigg|\sum_{\substack{m\leq K \\ q^2<p\leq x/m}} \frac{f(mp)}{mp} e(mp\alpha)\bigg|&
\leq \sum_{m\leq K} \frac{1}{m}\bigg|\sum_{q^2<p\leq x/m}\frac{f(p)}{p}e(mp\alpha)\bigg|\\
& \ll (\log K) \cdot \max_{m\leq K} \bigg|\sum_{q^2<p\leq x/m}\frac{f(p)}{p}e(mp\alpha)\bigg|,
\end{align*}
and 
$$ \sum_{q^2<p\leq x/m}\frac{f(p)}{p}e(mp\alpha)= \sum_{p\leq x}\frac{f(p)}{p}e(mp\alpha)+O(\log\log\log x). $$
Consequently, we obtain 
$$ \max_{m\leq K} \bigg|\sum_{p\leq x}\frac{f(p)}{p}e(mp\alpha)\bigg| \gg \frac{c}{\log K} \log\log x,$$
which completes the proof.  \hfill \qed
%\end{proof}

\begin{proof}[Proof of Theorem \ref{Thm:Classify_Large} for arbitrary multiplicative functions $f:\mathbb N\to \mathbb U$]
Let $g\in \mathcal F$ for which $g(p)=f(p)$ for all primes $p$ and write $f=g*h$ so that 
$h(p^j)=f(p^j)-f(p)f(p^{j-1})$ for all $j\geq 1$, which implies $h(p)=0$ and $|h(p^j)|\leq 2$ for all $j\geq 2$.
Then we have 
\begin{align}\label{Eq:SecondExpressionConvolution}
\sum_{n\leq x}  \frac{f(n)}{n} e(n\alpha) &= \sum_{n\leq x}  \frac{f(n)}{n} e\left(\frac{na}{q}\right) + O\left(\frac{1}{q}\right) \nonumber\\
& = \sum_{b\leq x} \frac{h(b)}{b} \sum_{m\leq x/b} \frac{g(m)}{m}  e\left(\frac{mba}{q}\right)+ O\left(\frac{1}{q}\right).
\end{align}
Let $B\geq 1$ be a parameter to be chosen. Writing $s=(b, q)$ and $k=b/
 s$, and using \eqref{Eq:MVLargeIntegers+} we obtain
\begin{align}\label{Eq:ConvolutionThm1.1}
 \sum_{B<b \leq x} \frac{h(b)}{b} \sum_{m\leq x/b} \frac{g(m)}{m} e\left(\frac{mba}{q}\right)  &= \sum_{s \mid q} 
 \sum_{\substack{B/s<k\leq x/s\\ (k, q/s)=1}}\frac{h(ks)}{ks} \sum_{m\leq x/(ks)} \frac{g(m)}{m} e\left(\frac{mka}{q/s}\right)\nonumber\\
 &\ll \sum_{s \mid q}\sum_{\substack{B/s<k\leq x/s\\ (k, q/s)=1}}\frac{|h(ks)|}{ks}  \left(\frac{\log x}{\sqrt{q/s}}\log q+ \log q\right).
 &   
\end{align}
By \eqref{eq: b root r} the contribution of the first sum is 
 $$ 
 \ll \frac{\log x}{\sqrt{q}}\log q\sum_{s\mid q} \sum_{k\geq 1} \frac{|h(ks)|}{k\sqrt{s}}\ll \frac{\log x}{\sqrt{q}}(\log q)^{1+o(1)}=o(1),
 $$
 since $q\geq (\log x)^{2+\ep}.$ 
 We now bound the second sum on the right hand side of \eqref{Eq:ConvolutionThm1.1} which is 
 $$ 
 \log q \sum_{s \mid q}\sum_{\substack{B/s<k\leq x/s\\ (k, q/s)=1}}\frac{|h(ks)|}{ks}= \log q\sum_{B<b\leq x}\frac{|h(b)|}{b}\ll B^{-1/3}\log q\sum_{b\geq 1}\frac{|h(b)|}{b^{2/3}}\ll B^{-1/3}\log q,
 $$
by \eqref{Eq:BoundSumHConvolution}. Combining these estimates with \eqref{Eq:SecondExpressionConvolution} and \eqref{Eq:ConvolutionThm1.1}  we get 
$$ 
\sum_{n\leq x}  \frac{f(n)}{n} e(n\alpha)= \sum_{b\leq B} \frac{h(b)}{b} \sum_{m\leq x/b} \frac{g(m)}{m}  e\left(\frac{mba}{q}\right)+ O\left(B^{-1/3}\log q\right)+o(1).
$$ 
We now let $B=\beta/c^3$ where $\beta$ is a suitably large absolute constant. 

\noindent Since $\left|\sum_{n\leq x} \frac{f(n)}{n} e(n\alpha)\right|\geq c\log q$ by our assumption, and $ \sum_{b\geq1} |h(b)|/b \ll 1$ by \eqref{Eq:BoundSumHConvolution}, there exists some integer $1\leq b\leq B$ such that 
\[
\sum_{m\leq x} \frac{g(m)}{m}  e\left(\frac{mba}{q}\right)=  \sum_{m\leq x/b} \frac{g(m)}{m}  e\left(\frac{mba}{q}\right)+ O_c\left(1\right) \gg c\log q,
\]
 where the implicit constant is absolute. Let $q_1=q/(b, q)\asymp_c q$,  $b_1=b/(q, b)\ll_c 1$ and put $\alpha_1= ba/q= b_1 a/q_1$.  Then by Theorem \ref{Thm:Classify_Large} for completely multiplicative functions applied to $g$ at least one of the following holds:

 \begin{itemize}
\item[(a)] There exist coprime integers $1\leq a_1\leq r_1\ll_c 1$ such that $|\alpha_1- a_1/r_1|\ll_c q_1^{-c/3}\ll_c q^{-c/3}$, and a primitive character $\chi \pmod \ell$ with $\ell | r_1$ such that $\mathbb{D}(g, \chi; q_1)\ll_c 1$. Now $\alpha= a/q+O(1/(qx))= \alpha_1/b+ O(1/(qx))$ and hence $|
\alpha- a_1/r|\ll_c q^{-c/3}$, where $r=br_1$. Writing $a_1/r=a_2/r_2$ with $(a_2, r_2)=1$ we get that $r_1| r_2$ and hence $\ell | r_2$.  Moreover, we have
$$\mathbb{D}(f, \chi; q)=\mathbb{D}(g, \chi; q_1)+ O_c\left(\frac{1}{\log q}\right)\ll_c 1.$$
Thus we have the same conclusion in this case. 
\item[(b)] $q_1$ is in the range $(\log x)^{2+\ep}\leq q_1 \leq (\log x)^{A_1}$ for some constant $A_1=A_1(c)$ and there exists an integer $h\ll 1/c^3$ such that 
$$\left|\sum_{p\leq x} \frac{g(p)}{p} e(hp\alpha_1)\right
    |\gg \frac{c}{\log (1/c)}\log\log x.$$
    Now $q_1\leq q\ll_c q_1$ and hence $(\log x)^{2+\ep} \leq q\leq (\log x)^{A}$ where $A=A_1+\ep$ (if $x$ is large enough). Finally we note that 
    $$\sum_{p\leq x} \frac{g(p)}{p} e(hp\alpha_1)=  \sum_{p\leq x} \frac{f(p)}{p} e(hbp\alpha)+ O_c\left(\frac{1}{q\log x}\right),
    $$
    since $\alpha_1= b\alpha +O_c(1/(qx))$. Hence we get the desired conclusion in part (b) with $h_1=bh \ll 1/c^6$. This completes the proof.
\end{itemize}

\end{proof}


\begin{thebibliography}{99}
 
 
 %\bibitem{Bach} G. Bachman, On exponential sums with multiplicative coefficients. II, 
%\emph{Acta Arith} 106, (2003), 41--57.

 
 \bibitem{BGS13} A. Balog, A. Granville and K. Soundararajan, Multiplicative functions in arithmetic progressions. \emph{Ann. Math. Qu\'ebec}, 37 (2013), no. 1, 3--30. 
 

%\bibitem{dlB99} R. de la Bret\`eche,  Sommes sans grand facteur premier, \emph{Acta
%Arith.}, {\bf 88}, no.~1, (1999),  1--14.

\bibitem {BGGK18}  J. Bober, L. Goldmakher, A. Granville, and D. Koukoulopoulos, 
The frequency and the structure of large character sums.
\emph{J. Eur. Math. Soc.} (JEMS) 20 (2018), no. 7, 1759--1818.

\bibitem{dlB98} R. de la Bret\`eche,  Sommes  d'exponentielles et entiers
sans grand facteur premier. \emph{Proc. London Math. Soc}, {\bf 77}, (1998),
39--78.

\bibitem{dlBG2} R. de la Bret\`eche and A. Granville,
Exponential sums with multiplicative coefficients and applications.
 \emph{Trans.  Amer. Math. Soc.} {\bf 375} (2022), 6875--6901.
 
 \bibitem{BrTe05} R. de la Bret\`eche and G. Tenenbaum, 
 Propri\'et\'es statistiques des entiers friables.
\emph{Ramanujan J.} 9 (2005), no. 1-2, 139--202.

\bibitem{BrTe07}R. de la Bret\`eche and G. Tenenbaum, 
Sommes d'exponentielles friables d'arguments rationnels.
\emph{Funct. Approx. Comment. Math.} 37 (2007), part 1, 31--38.

\bibitem{Bru51} N. G. de Bruijn, On the number of positive integers $\leq x$ and free of prime factors $>y$, 
\emph{Nederl. Akad. Wetensch. Proc. Ser. A} 54 (1951), 50--60.

 
% \bibitem{dlBT} R. de la Bret\`eche and G. Tenenbaum,
%Sommes d'exponentielles friables d'arguments rationnels,
 %\emph{Functiones et Approx.} {\bf 37} (2007), 31--38.
 
 \bibitem{CEP83}  E. R. Canfield, P. Erd\H{o}s and C. Pomerance,
 On a problem of Oppenheim concerning ``factorisatio numerorum''.
\emph{J. Number Theory} 17 (1983), no. 1, 1--28.


\bibitem{Dab96} H. Daboussi, 
Effective estimates of exponential sums over primes. \emph{Analytic number theory}, Vol. 1 (Allerton Park, IL, 1995), 231--244.
Progr. Math., 138
Birkhäuser Boston, Inc., Boston, MA, 1996.

\bibitem{D80} H. Davenport, {\em Multiplicative number theory}.
 Springer Verlag, New York, (1980).
 
 \bibitem{Dr16} S. Drappeau,
Sommes friables d'exponentielles et applications.
\emph{Canad. J. Math.} 67 (2015), no. 3, 597--638.
 
\bibitem{FoTe91} E. Fouvry and G. Tenenbaum, 
Entiers sans grand facteur premier en progressions arithm\'etiques.
\emph{Proc. London Math. Soc.} (3) 63 (1991), no. 3, 449--494.

\bibitem{Gr93} A. Granville, Integers, without large prime factors, in arithmetic progressions. I
\emph{Acta Math.} 170 (1993), no. 2, 255--273.
 
\bibitem{GS07} A. Granville and K. Soundararajan,  Large character sums: Pretentious
characters and the P{\' o}lya-Vinogradov theorem. \emph{J. Amer.
Math. Soc}, {\bf  20}, (2007),  357-384.

%\acom{
%\bibitem{GS18} A. Granville and K. Soundararajan,  Multiplicative number theory:
%The pretentious approach (preprint, 2018)}

\bibitem{GHS19} A. Granville, A. Harper and K. Soundararajan,   A new proof of Hal\'asz's Theorem, and some consequences.  \emph{Compos. Math.} {\bf 155} (2019), no. 1, 126--163. 

\bibitem{GM23} A. Granville and S. Mangerel, Three conjectures about character sums.  
\emph{Math. Zeit.} {\bf 305}, 49 (2023). 
 
%\bibitem{H68}	G. Hal\'asz, \"Uber die Mittelwerte multiplikativer zahlentheoretischer Funktionen, {\it  Acad. Math. Acad. Sci. 
%Hungar.} {\bf  19} (1968), 365--403. 

\bibitem{Go12} L. Goldmakher, Multiplicative mimicry and improvements to the P\'olya-Vinogradov inequality.
\emph{Algebra Number Theory} 6 (2012), no. 1, 123--163.

\bibitem{GoLa12} L. Goldmakher and Y. Lamzouri, Lower bounds on odd order character sums. 
\emph{Int. Math. Res. Not. IMRN} 2012, no. 21, 5006--5013.

\bibitem{GoLa14} L. Goldmakher and Y. Lamzouri, Large even order character sums.
\emph{Proc. Amer. Math. Soc.} 142 (2014), no. 8, 2609--2614.

\bibitem{HR74} H. Halberstam, and  H. E. Richert, Sieve methods
\emph{London Math. Soc. Monogr.}, No. 4
Academic Press, London-New York, 1974, xiv+364 pp.

\bibitem{Har12a} A. J. Harper, On a paper of K. Soundararajan on smooth numbers in arithmetic progressions.
\emph{J. Number Theory} 132 (2012), no. 1, 182--199.

\bibitem{Har12b} A. J. Harper, Bombieri--Vinogradov and Barban--Davenport--Halberstam type theorems for smooth numbers. 
\emph{arXiv e-print, 2012. http://arxiv.org/abs/1208.5992.}

\bibitem{Har16} A. J. Harper
Minor arcs, mean values, and restriction theory for exponential sums over smooth numbers.
\emph{Compos. Math.} 152 (2016), no. 6, 1121--1158.

\bibitem{HNR15} A. J. Harper,  A. Nikeghbali and M. Radziwi\l\l, 
A note on Helson's conjecture on moments of random multiplicative functions. \emph{Analytic number theory}, 
145--169. Springer, Cham, 2015.

\bibitem{Hi85} A. Hildebrand, 
Integers free of large prime divisors in short intervals.
\emph{Quart. J. Math. Oxford Ser. (2)} 36 (1985), no. 141, 57--69.

\bibitem{Hi86} A. Hildebrand, On the number of positive integers $\leq x$ and free of prime factors $> y$. \emph{J. Number theory}
22 (1986), no. 3, 289--307.


\bibitem{HiTe86} A. Hildebrand and G. Tenenbaum, 
On integers free of large prime factors.
\emph{Trans. Amer. Math. Soc.} 296 (1986), no. 1, 265--290.

\bibitem{Hu72} M. N. Huxley. On the difference between consecutive primes. \emph{Invent. Math.}, 15: 164--170, 1972.



\bibitem{LaMa22} Y. Lamzouri and A. P. Mangerel, Large odd order character sums and improvements to the P\'olya-Vinogradov inequality. \emph{Trans. Amer. Math. Soc.} 375 (2022), no. 6, 3759--3793.

\bibitem{Mai06} H. Maier,
Exponential sums with multiplicative coefficients over smooth integers.
\emph{Funct. Approx. Comment. Math.} 35 (2006), 209--218.

 
 \bibitem{MV77} H. L. Montgomery  and R. C. Vaughan,  Exponential
sums with multiplicative coefficients. \emph{Invent. Math}, {\bf 43}, (1977),  69-82. 

%\bibitem{Ra10} O. Ramar\'e, On Bombieri’s asymptotic sieve. \emph{J. Number Theory} 130 (2010), 1155--1189.

\bibitem{Shiu80}
P. Shiu, A Brun-Titchmarsh theorem for multiplicative functions. \emph{J. Reine Angew. Math.} 313 (1980), 161--170. 

\bibitem{So08} K. Soundararajan, 
The distribution of smooth numbers in arithmetic progressions. \emph{Anatomy of integers,} 115--128.
CRM Proc. Lecture Notes, 46
American Mathematical Society, Providence, RI, 2008.

\bibitem{Te93} G. Tenenbaum, 
Cribler les entiers sans grand facteur premier.
\emph{Philos. Trans. Roy. Soc. London Ser. A} 345 (1993), no. 1676, 377--384.
 
 

\end{thebibliography}
\end{document}